\documentclass[oneside,british]{amsart}
\usepackage[T1]{fontenc}
\usepackage[latin9]{inputenc}
\usepackage{textcomp}
\usepackage{amstext}
\usepackage{amsthm}
\usepackage{amssymb}

\makeatletter
\numberwithin{equation}{section}
\numberwithin{figure}{section}
\theoremstyle{plain}
\newtheorem{thm}{\protect\theoremname}[section]
\theoremstyle{remark}
\newtheorem{rem}[thm]{\protect\remarkname}
\theoremstyle{plain}
\newtheorem{cor}[thm]{\protect\corollaryname}
\theoremstyle{definition}
\newtheorem{defn}[thm]{\protect\definitionname}
\theoremstyle{plain}
\newtheorem{lem}[thm]{\protect\lemmaname}
\theoremstyle{remark}
\newtheorem*{rem*}{\protect\remarkname}
\theoremstyle{plain}
\newtheorem{prop}[thm]{\protect\propositionname}
\theoremstyle{remark}
\newtheorem{claim}[thm]{\protect\claimname}
\theoremstyle{plain}
\newtheorem*{thm*}{\protect\theoremname}

\makeatother

\usepackage{babel}
\providecommand{\claimname}{Claim}
\providecommand{\corollaryname}{Corollary}
\providecommand{\definitionname}{Definition}
\providecommand{\lemmaname}{Lemma}
\providecommand{\propositionname}{Proposition}
\providecommand{\remarkname}{Remark}
\providecommand{\theoremname}{Theorem}

\begin{document}

\title{Exact dimensionality and Ledrappier-Young formula for the Furstenberg
measure}

\author{\noindent Ariel Rapaport}


\thanks{This research was supported by the Herchel Smith Fund at the University
of Cambridge.}
\begin{abstract}
Assuming strong irreducibility and proximality, we prove that the
Furstenberg measure, corresponding to a finitely supported measure
on the general linear group of a finite dimensional real vector space,
is exact dimensional. We also establish a Ledrappier-Young type formula
for its dimension. The general strategy of the proof is based on the
argument given by Feng for the exact dimensionality of self-affine
measures.
\end{abstract}

\maketitle

\section{\label{sec:Introduction}Introduction}

\subsection{\label{subsec:Background}Background and the main result}

Let $V$ be a real vector space with $2\le\dim V<\infty$. Fix an
inner product $\left\langle \cdot,\cdot\right\rangle $ on $V$, and
denote the induced norm by $|\cdot|$. For a linear subspace $W$
of $V$ denote by $\mathrm{P}(W)$ its projective space. For $\overline{x},\overline{y}\in\mathrm{P}(V)$
set,
\[
d(\overline{x},\overline{y})=\left(1-\left\langle x,y\right\rangle ^{2}\right)^{1/2},
\]
where $x\in\overline{x}$ and $y\in\overline{y}$ are unit vectors.
It is easy to verify that this defines a metric on $\mathrm{P}(V)$.

The general linear group of $V$ acts on $\mathrm{P}(V)$ in a natural
way by setting,
\[
A\overline{x}=\overline{Ax}\text{ for }A\in\mathrm{GL}(V)\text{ and }\overline{x}\in\mathrm{P}(V)\:.
\]
Let $\mu\in\mathcal{M}(\mathrm{GL}(V))$, where for a standard Borel
space $X$ the collection of Borel probability measures on $X$ is
denoted by $\mathcal{M}(X)$. We say that $\nu\in\mathcal{M}(\mathrm{P}(V))$
is $\mu$-stationary if,
\[
\nu(F)=\int A\nu(F)\:d\mu(A)\;\text{ for every Borel set }F\subset\mathrm{P}(V),
\]
where $A\nu$ is the push-forward of $\nu$ via the map $\overline{x}\rightarrow A\overline{x}$.
Since $\mathrm{P}(V)$ is compact there always exists at least one
$\mu$-stationary measure.

Write $S_{\mu}$ for the smallest closed subsemigroup of $\mathrm{GL}(V)$
such that $\mu(S_{\mu})=1$. Suppose from now on that $S_{\mu}$ is
strongly irreducible and proximal. The first assumption means that
there does not exist a finite family of proper nonzero linear subspaces
$W_{1},...,W_{k}$ of $V$ such that,
\[
A(\cup_{i=1}^{k}W_{i})=\cup_{i=1}^{k}W_{i}\text{ for all }A\in S_{\mu}\:.
\]
The second assumption means that there exist $A_{1},A_{2},...\in S_{\mu}$
and $\alpha_{1},\alpha_{2},...\in\mathbb{R}$ such that $\{\alpha_{n}A_{n}\}_{n\ge1}$
converges to a rank $1$ endomorphism of $V$ in the norm topology.
From these assumptions it follows that there exists a unique $\nu\in\mathcal{M}(\mathrm{P}(V))$
which is $\mu$-stationary. It is called the Furstenberg measure corresponding
to the distribution $\mu$. For a proof see \cite[Theorem III.3.1]{BL}
or \cite[Proposition 4.7]{BQ}. 

The main purpose of this paper is to establish the exact dimensionality
of the Furstenberg measure, under the additional assumption of $\mu$
being finitely supported. A Borel probability measure $\theta$ on
a metric space $X$ is said to be exact dimensional if there exists
a number $\alpha\ge0$ such that,
\[
\underset{r\downarrow0}{\lim}\:\frac{\log\theta(B(x,r))}{\log r}=\alpha\text{ for }\theta\text{-a.e. }x\in X,
\]
where $B(x,r)$ is the closed ball in $X$ with centre $x$ and radius
$r$. If $\theta$ is exact dimensional then the number $\alpha$
is denoted $\dim\theta$, is called the dimension of $\theta$ and
is equal to the value given to $\theta$ by other commonly used notions
of dimension (see \cite[Chapter 10]{Fa}). In particular $\dim\theta$
is equal to the Hausdorff dimension of $\theta$, which is denoted
$\dim_{H}\theta$ and defined by,
\begin{equation}
\dim_{H}\theta=\inf\{\dim_{H}F\::\:F\subset X\text{ is Borel with }\theta(F)>0\},\label{eq:def of h-dim}
\end{equation}
where $\dim_{H}F$ is the Hausdorff dimension of $F$.

When $\dim V=2$ the exact dimensionality of the Furstenberg measure
$\nu$ was already established in previous works, without assuming
that $\mu$ is finitely supported. It was shown by Ledrappier (see
\cite{Le}) that in this case the function,
\[
\overline{x}\rightarrow\frac{\log\nu(B(\overline{x},r))}{\log r},
\]
converges in $\nu$-probability to the value $h_{F}(\nu)/(\lambda_{0}-\lambda_{1})$
as $r\rightarrow0$. Here $\lambda_{0}>\lambda_{1}$ are the Lyapunov
exponents corresponding to $\mu$ (see the next section), and $h_{F}(\nu)$
is the Furstenberg entropy of $\nu$ which is defined by,
\[
h_{F}(\nu)=\int\int\log\frac{dA\nu}{d\nu}(\overline{x})\:dA\nu(\overline{x})\:d\mu(A)\:.
\]
More recently, Hochman and Solomyak \cite{HS} (see the discussion
below) have shown that $\nu$ is exact dimensional with,
\[
\dim\nu=h_{F}(\nu)/(\lambda_{0}-\lambda_{1}),
\]
whenever $\dim V=2$. In a recent paper Lessa \cite{Les} has extended
these results to disintegrations along certain $1$-dimensional foliations,
of stationary measures on the space of complete flags.

In this paper we establish the exact dimensionality of $\nu$ also
in higher dimensions. The following theorem is our main result.
\begin{thm}
\label{thm:ED of nu}Let $\mu\in\mathcal{M}(\mathrm{GL}(V))$ be finitely
supported, and suppose that $S_{\mu}$ is strongly irreducible and
proximal. Let $\nu\in\mathcal{M}(\mathrm{P}(V))$ be the Furstenberg
measure corresponding to $\mu$. Then $\nu$ is exact dimensional,
and $\dim\nu$ satisfies a Ledrappier-Young type dimension formula.
\end{thm}

The precise formula satisfied by $\dim\nu$ will be given in the next
section. Its name comes from the work of Ledrappier and Young \cite{LY},
in which they have obtained a formula, in terms of conditional entropies
and Lyapunov exponents, for the local dimensions along stable and
unstable manifolds of invariant measures of $C^{2}$ smooth diffeomorphisms.

Let us provide some more background and mention other related results.
A measure $\theta$ on $\mathbb{R}^{d}$ is said to be self-affine
if it is stationary with respect to a finitely supported measure $\rho$
on the semigroup of affine invertible contractions of $\mathbb{R}^{d}$.
If $\rho$ is supported on the semigroup of contracting similarities,
then $\theta$ is said to be self-similar. The support of the measure
$\rho$ is sometimes referred to as an iterated function system (IFS).

Self-affine measures and Furstenberg measures share various features.
For instance, both can be realised as the image of a Bernoulli measure
on the symbolic space under an appropriate equivariant map. In the
case of self-affine measures this map is called the coding map. For
the Furstenberg measure it is called the Furstenberg boundary map
(see the next section). Additionally, both types of measures can be
represented as a weighted average of distorted copies of themselves,
which are sometimes referred to as cylinder measures. The Furstenberg
measure on the $1$-dimensional projective space resembles a self-similar
measure on the real line. In higher dimensions the Furstenberg measure
resembles a self-affine measure, for which the linear parts of the
maps in the IFS satisfy irreducibility and proximality assumptions
similar to ours.

Exact dimensionality plays an important role in the study of stationary
fractal measures, and the question of whether every self-affine measure
satisfies this property has received a lot of attention. In \cite{FH},
by introducing a notion of projection entropy, Feng and Hu have proved
that every self-similar measure on $\mathbb{R}^{d}$ is exact dimensional,
with dimension given by the projection entropy divided by the Lyapunov
exponent. In fact they have shown this, more generally, for the push-forward
of any ergodic measure under the coding map.

In \cite{BK} Bárány and Käenmäki proved that every planar self-affine
measure is exact dimensional. Moreover, they proved this for every
self-affine measure on $\mathbb{R}^{d}$ with $d$ distinct Lyapunov
exponents, and showed that its dimension is given by a Ledrappier-Young
type formula. Additionally, under further assumptions, they established
this for projections under the coding map of quasi-Bernoulli measures.

Lastly, in a recent paper Feng \cite{Fe} has managed to provide a
complete solution for this problem, and proved that all self-affine
measures are exact dimensional and satisfy a Ledrappier-Young type
formula. In fact he was able to show this for projections of general
ergodic measures and to systems which are only average contracting.
The general strategy for our proof of Theorem \ref{thm:ED of nu}
is based on Feng's argument.

Besides their intrinsic interest, exact dimensionality and Ledrappier-Young
type formulas have played an important role in some recent and significant
developments in the dimension theory of fractal measures. Hochman
\cite{Ho1,Ho2} has shown that, under a mild exponential separation
assumption on the maps in the IFS and an additional irreducibility
assumption in higher dimensions, the dimension of a self-similar measure
is equal to its natural upper bound. If $\rho$ is the corresponding
finitely supported measure on the contracting similarities, this upper
bound is equal to the minimum between the dimension of the ambient
space and the quotient obtained by dividing the Shanon entropy of
$\rho$ by its Lyapunov exponent. To be more precise regarding the
exponential separation assumption, it requires the existence of an
$\epsilon>0$ such that for all $n\ge1$ the distance between two
distinct compositions of length $n$ of map from the IFS is at least
$\epsilon^{n}$. These works rely on the exact dimensionality of self-similar
measures.

Bárány, Hochman and Rapaport \cite{BHR} have proved for planar self-affine
measures that if one assumes strong irreducibility and proximality
for the linear parts of the maps in the IFS, and that the supports
of the cylinder measures are disjoint, then the dimension is equal
to its natural upper bound. In this case the natural upper bound is
the minimum between $2$ and a quantity known as the Lyapunov dimension,
which generalises the entropy divided by exponent formula. The condition
regarding the disjointness of the supports is usually referred to
as the strong separation condition (SSC). Hochman and Rapaport \cite{HR}
have later established this statement under a much milder exponential
separation assumption instead of the SSC. Both of these results rely
on the exact dimensionality and Ledrappier-Young formula for planar
self-affine measures.

Lastly, in \cite{HS} Hochman and Solomyak have proved their main
result while relying on the exact dimensionality of the $1$-dimensional
Furstenberg measure, which they establish in the same paper. Stated
in the notation of Theorem \ref{thm:ED of nu}, this result says that
if $\mu$ is finitely supported, $S_{\mu}$ is strongly irreducible
and proximal, and the matrices in the support of $\mu$ satisfy an
exponential separation condition, then $\dim\nu$ is equal to its
natural upper bound, where $\nu$ is the Furstenberg measure on $\mathrm{P}(\mathbb{R}^{2})$.
The natural upper bound in this case is the minimum between $1$ and
the Shanon entropy of $\mu$ divided by the difference of the two
Lyapunov exponents.

In Section \ref{subsec:The-Lyapunov-dimension} we introduce a value
$\dim_{\mathrm{LY}}\mu$, which we call the Lyapunov dimension corresponding
to $\mu$. It extends the aforementioned upper bound for $\dim\nu$
in the $1$-dimensional case, and is analogous to the Lyapunov dimension
of a self-affine measure. From the dimension formula stated in the
next section it will follow easily that $\dim_{\mathrm{LY}}\mu$ is
always an upper bound for $\dim\nu$. Considering the results mentioned
above, it is reasonable to expect for these two values to be equal
under an additional exponential separation assumption. The results
of this paper should be a necessary ingredient in the proof of such
a statement.

\subsection{\label{subsec:Dimension-formulas}Dimension formulas}

In this section we provide a precise statement for the dimension formula
satisfied by the Furstenberg measure $\nu$. We also give similar
formulas for typical projections and slices of $\nu$. First we need
some more definitions and notations.

Let $\mu\in\mathcal{M}(\mathrm{GL}(V))$ be finitely supported. Then
there exist a finite index set $\Lambda$, distinct elements $\{A_{l}\}_{l\in\Lambda}$
of $\mathrm{GL}(V)$, and a probability vector $p=(p_{l})_{l\in\Lambda}$
with strictly positive coordinates, such that
\[
\mu=\sum_{l\in\Lambda}p_{l}\delta_{A_{l}}\:.
\]
Here $\delta_{A_{l}}\in\mathcal{M}(\mathrm{GL}(V))$ is the Dirac
mass at $A_{l}$. As before, denote by $S_{\mu}$ the smallest closed
subsemigroup of $\mathrm{GL}(V)$ such that $\mu(S_{\mu})=1$, and
suppose that $S_{\mu}$ is strongly irreducible and proximal. Let
$\nu\in\mathcal{M}(\mathrm{P}(V))$ be the Furstenberg measure corresponding
to $\mu$.

Write $\Omega=\Lambda^{\mathbb{Z}}$ and equip $\Omega$ with its
Borel $\sigma$-algebra, generated by the cylinder sets. Let $\beta$
be the Bernoulli measure on $\Omega$ corresponding to the probability
vector $p$, that is $\beta=p^{\mathbb{Z}}$. Let $\sigma:\Omega\rightarrow\Omega$
be the left shift map, i.e.
\[
(\sigma\omega)_{n}=\omega_{n+1}\text{ for }\omega\in\Omega\text{ and }n\in\mathbb{Z}\:.
\]

From our assumptions on $S_{\mu}$ it follows (see \cite[Lemma 2.17 and Proposition 4.7]{BQ})
that there exists a Borel map $\pi:\Omega\rightarrow\mathrm{P}(V)$
such that,
\begin{enumerate}
\item $\pi$ depends only on the nonnegative coordinates of $\Omega$;
\item $\pi\omega=A_{\omega_{0}}\pi\sigma\omega$ for $\beta$-a.e. $\omega$;
\item the distribution of $\pi$ with respect to $\beta$ is equal to $\nu$,
that is $\pi\beta=\nu$;
\item for $\beta$-a.e. $\omega$,
\[
\underset{n\rightarrow\infty}{\lim}\:A_{\omega_{0}}\cdot\cdot\cdot A_{\omega_{n}}\nu=\delta_{\pi(\omega)},
\]
where $\delta_{\pi(\omega)}$ is the Dirac mass at $\pi(\omega)$
and the convergence is in the weak-{*} topology.
\end{enumerate}
The map $\pi$ is often called the Furstenberg boundary map.

By the Oseledets' multiplicative ergodic theorem \cite{O}, applied
to the ergodic system $(\Omega,\beta,\sigma^{-1})$ and the matrix
cocycle $\omega\rightarrow A_{\omega_{-1}}$, there exist positive
integers $s,d_{0},...,d_{s}$, real numbers $\lambda_{0}>...>\lambda_{s}$
and linear subspaces,
\[
V=V_{\omega}^{-1}\supset V_{\omega}^{0}\supset...\supset V_{\omega}^{s}=\{0\}\text{ for }\omega\in\Omega,
\]
such that,
\begin{enumerate}
\item $\dim V_{\omega}^{i}=\sum_{k=i+1}^{s}d_{k}$ for $\omega\in\Omega$
and $-1\le i\le s$;
\item the map $\omega\rightarrow(V_{\omega}^{i})_{i=-1}^{s}$ is Borel measurable
and depends only on the negative coordinates of $\Omega$;
\item $V_{\sigma^{-1}\omega}^{i}=A_{\omega_{-1}}V_{\omega}^{i}$ for $\beta$-a.e.
$\omega$ and each $-1\le i\le s$;
\item for $\beta$-a.e. $\omega$ and each $0\le i\le s$,
\[
\underset{n\rightarrow\infty}{\lim}\:\frac{1}{n}\log|A_{\omega_{-n}}...A_{\omega_{-1}}x|=\lambda_{i}\text{ for }x\in V_{\omega}^{i-1}\setminus V_{\omega}^{i}\:.
\]
\end{enumerate}
The numbers $\lambda_{0},...,\lambda_{s}$ are called the Lyapunov
exponents corresponding to $\mu$. For $0\le i\le s$ the integer
$d_{i}$ is called the multiplicity of $\lambda_{i}$. Note that from
our assumptions on $S_{\mu}$ it follows that $d_{0}=1$ (see \cite[Theorem III.6.1]{BL}).
We set,
\[
\tilde{\lambda}_{i}=\lambda_{i}-\lambda_{0}\text{ for }1\le i\le s\:.
\]

\begin{rem}
Let $\theta$ be the distribution of the random flag $(V_{\omega}^{i})_{i=-1}^{s}$.
That is, for every Borel subset $B$ of the flag manifold,
\[
\theta(B)=\beta\{\omega\::\:(V_{\omega}^{i})_{i=-1}^{s}\in B\}\:.
\]
Write $\mu^{-}$ for the distribution,
\[
\sum_{l\in\Lambda}p_{l}\delta_{A_{l}^{-1}}\in\mathcal{M}(\mathrm{GL}(V))\:.
\]
Then from the identities $V_{\sigma^{-1}\omega}^{i}=A_{\omega_{-1}}V_{\omega}^{i}$
it follows that $\theta$ is $\mu^{-}$-stationary. Note that in general,
our assumptions do not guarantee the uniqueness of a $\mu^{-}$-stationary
measure on the flag manifold.
\end{rem}

For a proper linear subspace $W$ of $V$ write $P_{W^{\perp}}$ for
the orthogonal projection onto $W^{\perp}$. Note that $P_{W^{\perp}}$
defines a map from $\mathrm{P}(V)\setminus\mathrm{P}(W)$ to $\mathrm{P}(W^{\perp})$
by setting,
\[
P_{W^{\perp}}\overline{x}=\overline{P_{W^{\perp}}x}\quad\text{ for }\overline{x}\in\mathrm{P}(V)\setminus\mathrm{P}(W)\:.
\]
Let $\zeta_{W}$ be the partition of $\mathrm{P}(V)\setminus\mathrm{P}(W)$
such that for $\overline{x}\in\mathrm{P}(V)\setminus\mathrm{P}(W)$,
\[
\zeta_{W}(\overline{x})=\{\overline{y}\in\mathrm{P}(V)\setminus\mathrm{P}(W)\::\:P_{W^{\perp}}\overline{y}=P_{W^{\perp}}\overline{x}\}\:.
\]
Here $\zeta_{W}(\overline{x})$ denotes the unique element of $\zeta_{W}$
which contains $\overline{x}$. Since $W\ne V$, and because $S_{\mu}$
is strongly irreducible, is follows that $\nu(\mathrm{P}(W))=0$ (see
\cite[Proposition III.2.3]{BL}). Hence $P_{W^{\perp}}$ defines a
Borel map on $\mathrm{P}(V)$ outside a set of zero $\nu$-measure,
and the disintegration
\[
\{\nu_{\overline{x}}^{\zeta_{W}}\}_{\overline{x}\in\mathrm{P}(V)}\subset\mathcal{M}(\mathrm{P}(V)),
\]
of $\nu$ with respect to the measurable partition $\zeta_{W}$, is
$\nu$-a.e. well defined (see Section \ref{subsec:Disintegration}).
Note that if $W$ is of codimension $1$ then $\nu_{\overline{x}}^{\zeta_{W}}=\nu$
for $\nu$-a.e. $\overline{x}$. 

Denote by $\mathcal{P}$ the partition of $\Omega$ according to the
$0$-coordinate, that is
\[
\mathcal{P}=\{\{\omega\in\Omega\::\:\omega_{0}=l\}\::\:l\in\Lambda\}\:.
\]
Write $\mathcal{B}$ for the Borel $\sigma$-algebra of $\mathrm{P}(V)$.
For a proper linear subspace $W$ of $V$, write $\mathrm{H}_{\beta}(\mathcal{P}\mid\pi^{-1}P_{W^{\perp}}^{-1}\mathcal{B})$
for the conditional entropy of $\mathcal{P}$ given $\pi^{-1}P_{W^{\perp}}^{-1}\mathcal{B}$
with respect to $\beta$ (see Section \ref{subsec:info and ent}).
It is well defined since the identity $\pi\beta=\nu$ implies that
the composition $P_{W^{\perp}}\circ\pi$ defines a Borel map on $\Omega$
outside of a set of zero $\beta$-measure. Thus for $0\le i\le s$
we can set,
\begin{equation}
\mathrm{H}_{i}=\int\mathrm{H}_{\beta}(\mathcal{P}\mid\pi^{-1}P_{(V_{\omega}^{i})^{\perp}}^{-1}\mathcal{B})\:d\beta(\omega)\:.\label{eq:intro def of H_i}
\end{equation}
Note that since the subspaces $V_{\omega}^{0}$ are of codimension
$1$, the $\sigma$-algebras $\pi^{-1}P_{(V_{\omega}^{0})^{\perp}}^{-1}\mathcal{B}$
are trivial with respect to $\beta$. This implies that $\mathrm{H}_{0}=\mathrm{H}(p)$,
where $\mathrm{H}(p)$ is the entropy of the probability vector $p$.
Also observe that, since $V_{\omega}^{i+1}\subset V_{\omega}^{i}$
for $0\le i<s$ and $\omega\in\Omega$, the $\sigma$-algebras which
appear in the definition of $\mathrm{H}_{i+1}$ are finer than the
ones which appear in the definition of $\mathrm{H}_{i}$. This implies
that $\mathrm{H}_{i+1}\le\mathrm{H}_{i}$ for $0\le i<s$.

We are now ready to state our dimension formulas for $\nu$, its projections
and its slices.
\begin{thm}
\label{thm:LY formula and ED}Suppose that $\mu$ is finitely supported,
and that $S_{\mu}$ is strongly irreducible and proximal. Let $\nu$
be the Furstenberg measure corresponding to $\mu$. Then, in the notations
above, for $\beta$-a.e. $\omega\in\Omega$, $\nu$-a.e. $\overline{x}\in\mathrm{P}(V)$
and every $0\le i<k\le s$, the following statements are satisfied.
\begin{enumerate}
\item \label{enu:LY for nu}$\nu$ is exact dimensional with,
\[
\dim\nu=\sum_{j=0}^{s-1}\frac{\mathrm{H}_{j+1}-\mathrm{H}_{j}}{\tilde{\lambda}_{j+1}};
\]
\item \label{enu:LY for proj}$P_{(V_{\omega}^{k})^{\perp}}\nu$ is exact
dimensional with,
\[
\dim P_{(V_{\omega}^{k})^{\perp}}\nu=\sum_{j=0}^{k-1}\frac{\mathrm{H}_{j+1}-\mathrm{H}_{j}}{\tilde{\lambda}_{j+1}};
\]
\item \label{enu:LY for slices}$\nu_{\overline{x}}^{\zeta_{V_{\omega}^{i}}}$
is exact dimensional with,
\[
\dim\nu_{\overline{x}}^{\zeta_{V_{\omega}^{i}}}=\sum_{j=i}^{s-1}\frac{\mathrm{H}_{j+1}-\mathrm{H}_{j}}{\tilde{\lambda}_{j+1}};
\]
\item \label{enu:LY for proj of slices}$P_{(V_{\omega}^{k})^{\perp}}\nu_{\overline{x}}^{\zeta_{V_{\omega}^{i}}}$
is exact dimensional with,
\[
\dim P_{(V_{\omega}^{k})^{\perp}}\nu_{\overline{x}}^{\zeta_{V_{\omega}^{i}}}=\sum_{j=i}^{k-1}\frac{\mathrm{H}_{j+1}-\mathrm{H}_{j}}{\tilde{\lambda}_{j+1}}\:.
\]
\end{enumerate}
\end{thm}

\begin{rem}
\label{rem:only last part needed}For every $\omega\in\Omega$ the
subspace $V_{\omega}^{s}$ is trivial and $V_{\omega}^{0}$ is of
codimension $1$. Hence $P_{(V_{\omega}^{s})^{\perp}}$ is the identity
and $\nu_{\overline{x}}^{\zeta_{V_{\omega}^{0}}}=\nu$ for $\nu$-a.e.
$\overline{x}$. Thus in order to prove Theorems \ref{thm:LY formula and ED}
and \ref{thm:ED of nu}, it is enough to establish part (\ref{enu:LY for proj of slices})
of Theorem \ref{thm:LY formula and ED}.
\end{rem}

The above theorem yields a dimension conservation result for the Furstenberg
measure. Let $W$ be a proper linear subspace of $V$, and let $\theta\in\mathcal{M}(\mathrm{P}(V))$
be with $\theta(\mathrm{P}(W))=0$. Following Furstenberg \cite{Fu},
we say that $\theta$ is dimension conserving with respect to $P_{W^{\perp}}$
if,
\[
\dim_{H}P_{W^{\perp}}\theta+\dim_{H}\theta_{\overline{x}}^{\zeta_{W}}=\dim_{H}\theta\text{ for }\theta\text{-a.e. }\overline{x},
\]
where $\dim_{H}$ is as defined in (\ref{eq:def of h-dim}). The following
corollary follows directly from parts (\ref{enu:LY for nu})-(\ref{enu:LY for slices})
of Theorem \ref{thm:LY formula and ED}.
\begin{cor}
\label{cor:dim cons}Assume the conditions of Theorem \ref{thm:LY formula and ED}
are satisfied. Then $\nu$ is dimension conserving with respect to
$P_{(V_{\omega}^{i})^{\perp}}$ for $\beta$-a.e. $\omega\in\Omega$
and every $0\le i\le s$.
\end{cor}

For self-affine measures results analogous to Corollary \ref{cor:dim cons}
were obtained in \cite{BK} and \cite{Fe}. It is worth pointing out
that in \cite{FJ} Falconer and Jin proved that self-similar measures
on $\mathbb{R}^{d}$ with finite rotation groups are dimension conserving
with respect to any orthogonal projection. For self-similar sets with
finite rotation groups this result was first obtained by Furstenberg
\cite{Fu}, who introduced this notion.

\subsection{\label{subsec:The-Lyapunov-dimension}The Lyapunov dimension}

In this section we introduce the upper bound for $\dim\nu$, which
was mentioned in the discussion at the end of Section \ref{subsec:Background}.
We continue to use the notations from the previous section, and assume
the conditions of Theorem \ref{thm:LY formula and ED} are satisfied.
As before write $\mathrm{H}(p)$ for the entropy of the probability
vector $p$. We also set,
\[
L_{i}=-\sum_{j=1}^{i}\tilde{\lambda}_{j}d_{j}\;\text{ for }0\le i\le s\:.
\]

\begin{defn}
Let $m=m(\mu)$ be such that,
\[
m=\max\{0\le i\le s\::\:\mathrm{H}(p)\ge L_{i}\},
\]
and write,
\[
\dim_{\mathrm{LY}}\mu=\begin{cases}
\sum_{j=1}^{m}d_{j}+\frac{\mathrm{H}(p)-L_{m}}{-\tilde{\lambda}_{m+1}} & \text{, if }m<s\\
\dim V-1 & \text{, if }m=s
\end{cases}\:.
\]
We call the number $\dim_{\mathrm{LY}}\mu$ the Lyapunov dimension
corresponding to $\mu$.
\end{defn}

This definition is analogous to the one given by Jordan, Pollicott
and Simon in \cite{JPS} for the Lyapunov dimension $\dim_{\mathrm{LY}}\theta$
of a self-affine measure $\theta$ on $\mathbb{R}^{d}$. It was shown
there that $\dim_{\mathrm{LY}}\theta$ is always an upper bound for
$\dim_{H}\theta$, and that if the linear parts of the maps in the
IFS are fixed and all have norm strictly less than $1/2$, then
\[
\dim_{H}\theta=\min\{\dim_{\mathrm{LY}}\theta,d\}
\]
for Lebesgue a.e. selection of the translations.

Now let $\Delta$ be the set of numbers of the form $-\sum_{i=1}^{s}\tilde{\lambda}_{i}^{-1}x_{i}$,
where $x_{1},...,x_{s}$ are nonnegative real numbers which satisfy,
\[
\sum_{i=1}^{s}x_{i}\le\mathrm{H}(p)\;\text{ and }\;x_{i}\le-\tilde{\lambda}_{i}d_{i}\;\text{ for }1\le i\le s\:.
\]
From,
\[
0>\tilde{\lambda}_{1}>...>\tilde{\lambda}_{s}\;\text{ and }\;\sum_{i=1}^{s}d_{i}=\dim V-1,
\]
it follows easily that,
\begin{equation}
\dim_{\mathrm{LY}}\mu=\max\Delta\:.\label{eq:=00003Dmax}
\end{equation}

On the other hand, as a simple consequence of part (\ref{enu:LY for proj of slices})
of Theorem \ref{thm:LY formula and ED} (see Lemma \ref{lem:ub on dif H}),
\begin{equation}
0\le\mathrm{H}_{i}-\mathrm{H}_{i+1}\le-\tilde{\lambda}_{i+1}d_{i+1}\;\text{ for }0\le i<s\:.\label{eq:ub dim H}
\end{equation}
Also note that,
\[
\mathrm{H}(p)=\mathrm{H}_{0}\ge\sum_{i=0}^{s-1}(\mathrm{H}_{i}-\mathrm{H}_{i+1})\:.
\]
Combining the last inequality with (\ref{eq:ub dim H}) and part (\ref{enu:LY for nu})
of Theorem \ref{thm:LY formula and ED}, we obtain that $\dim\nu$
is a member of $\Delta$. This together with (\ref{eq:=00003Dmax})
yields the following corollary.
\begin{cor}
Assume the conditions of Theorem \ref{thm:LY formula and ED} are
satisfied, then $\dim\nu\le\dim_{\mathrm{LY}}\mu$.
\end{cor}

As mentioned in Section \ref{subsec:Background}, it is reasonable
to expect for the equality $\dim\nu=\dim_{\mathrm{LY}}\mu$ to hold
under an additional exponential separation assumption.

\subsection{\label{subsec:About-the-proof}About the proof of Theorem \ref{thm:LY formula and ED}}

The starting point of the argument is the observation that, under
our standing assumptions, the matrices
\[
A_{\omega_{-n}...\omega_{-1}}:=A_{\omega_{-n}}\cdot\cdot\cdot A_{\omega_{-1}}
\]
contract the projective space, outside of the set $\mathrm{P}(V_{\omega}^{0})$,
for $\beta$-a.e. $\omega$ and for $n\ge1$ large. Here $V_{\omega}^{0}$
is the linear hyperplane obtained by the Oseledets' theorem. This
contraction property follows from the fact that the multiplicity of
the top Lyapunov exponent is equal to $1$, and it makes it possible
to employ techniques used in the study of self-affine measures on
$\mathbb{R}^{d}$.

As mentioned before, the general idea of the proof is based on Feng's
argument for the exact dimensionality of self-affine measures. Nevertheless
our proof contains nontrivial differences and new features. Some of
these come from the fact that the matrices $A_{\omega_{-n}...\omega_{-1}}$
only contract most of the projective space. In the self-affine case,
or even in the more general average contracting case considered in
\cite{Fe}, these matrices uniformly contract all of the Euclidean
space for $\beta$-a.e. $\omega$.

In order to deal with this issue, for $\beta$-a.e. $\omega$ we use
the Oseledets splitting of $V$ at $\omega$ (see Section \ref{subsec:Oseledets})
in order to construct a coordinate chart for $\mathrm{P}(V)$, whose
domain contains $\pi\omega$ and on which the matrices $A_{\omega_{-n}...\omega_{-1}}$
are uniformly contracting. Recall that $\pi$ is the Furstenberg boundary
map. For this approach to succeed we have to make sure that, for $n\ge1$
large, the lines $\pi\sigma^{n}\omega$ do not become to close to
the boundary of the coordinate domains. This is achieved by applying
a result of Guivarc'h \cite{Gu} (see Section \ref{subsec:boundary map}).
It yields a regularity property for the Furstenberg measure $\nu$,
which controls the $\nu$-measure of neighbourhoods of projective
spaces of linear hyperplanes of $V$.

As in \cite{Fe}, a key part of the argument involves the estimation
of the so-called transverse dimensions, which are the local dimensions
of projections of certain conditional measures of $\beta$. These
conditional measures correspond to measurable partitions $\xi_{0},...,\xi_{s}$
of $\Omega$, which are similar to the ones constructed in \cite{Fe}.
When they are projected via $\pi$, one obtains the conditional measures
of $\nu$ which appear in the statement of Theorem \ref{thm:LY formula and ED}.
In order to deal with the estimation of the transverse dimensions
we employ an idea used in \cite{Fe}, which involves an induced dynamics
and makes it possible to focus on trajectories where the angles between
the Oseledets subspaces are not too small.

The result of Guivarc'h, which provides the regularity property for
$\nu$, requires the strong irreducibility and proximality assumptions.
These assumptions also insure that the multiplicity of the top Lyapunov
exponent is $1$, a fact which is crucial for our development. Strong
irreducibility also implies the necessary fact that $\nu(\mathrm{P}(W))=0$
for every proper linear subspace $W$ of $V$, though this may be
regarded as a very mild form of the regularity property. The assumption
of $\mu$ being finitely supported is needed in order to cary out
entropy computations and to guarantee the integrability of a certain
dominating function (see Lemma \ref{lem:prep for maker} and Remark
\ref{rem:fin sup assump} following it). It seems reasonable to expect
for the exact dimensionality of $\mu$-stationary measures to hold
under weaker assumptions than ours, but this will probably require
a different method of proof.

\subsection*{Structure of the paper}

In Section \ref{sec:Preliminaries} we develop necessary notations
and background. In Section \ref{sec:Construction-of-local} we construct
the coordinate charts mentioned above, and prove some related auxiliary
results. In Section \ref{sec:measurable partitions} we construct
the measurable partitions $\xi_{0},...,\xi_{s}$, and derive some
necessary properties of them. In Section \ref{sec:Transverse-dimensions}
we estimate the transverse dimensions. In Section \ref{sec:Proof-of-results}
we complete the proof of our main result Theorem \ref{thm:LY formula and ED}.

\section{\label{sec:Preliminaries}Preliminaries}

\subsection{\label{subsec:General-notations}General notations}

For a metric space $X$, $x\in X$ and $r>0$, we denote by $B(x,r)$
the closed ball in $X$ with centre $x$ and radius $r$. Given a
set $Y$, a map $\phi:Y\rightarrow X$, $y\in Y$ and $r>0$, we often
write $B^{\phi}(y,r)$ in place of $\phi^{-1}(B(\phi(y),r))$.

It will sometimes be convenient to use the little-o notation. For
parameters $\alpha_{1},...,\alpha_{k}$ we write $o_{\alpha_{1},...,\alpha_{n}}(n)$
in order to denote an unspecified function $f:\mathbb{N}\rightarrow\mathbb{R}$,
which depends on $\alpha_{1},...,\alpha_{n}$ and satisfies $\frac{1}{n}f(n)\rightarrow0$
as $n\rightarrow\infty$.

\subsection{\label{subsec:info and ent}Conditional information and entropy}

We give here the definitions and basic properties of the entropy and
information functions. For more details see \cite[Section 2]{Pa}
for instance.

Let $(X,\mathcal{B},\rho)$ be a probability space. For a sub-$\sigma$-algebra
$\mathcal{F}$ of $\mathcal{B}$ and $f\in L^{1}(\rho)$ we denote
the conditional expectation of $f$ given $\mathcal{F}$ by $\mathrm{E}_{\rho}(f\mid\mathcal{F})$.
Given a finite measurable partition $\mathcal{E}$ of $X$ we write
$\mathrm{I}_{\rho}(\mathcal{E}\mid\mathcal{F})$ for the conditional
information of $\mathcal{E}$ given $\mathcal{F}$. That is,
\[
\mathrm{I}_{\rho}(\mathcal{E}\mid\mathcal{F})=-\sum_{E\in\mathcal{E}}1_{E}\log\mathrm{E}_{\rho}(1_{E}\mid\mathcal{F}),
\]
where $1_{E}$ is the indicator function of $E$. The conditional
entropy of $\mathcal{E}$ given $\mathcal{F}$ is denoted $\mathrm{H}_{\rho}(\mathcal{E}\mid\mathcal{F})$
and defined by,
\[
\mathrm{H}_{\rho}(\mathcal{E}\mid\mathcal{F})=\int\mathrm{I}_{\rho}(\mathcal{E}\mid\mathcal{F})\:d\rho\:.
\]
When $\mathcal{F}$ is the trivial $\sigma$-algebra we write $\mathrm{H}_{\rho}(\mathcal{E})$
in place of $\mathrm{H}_{\rho}(\mathcal{E}\mid\mathcal{F})$.

If $\mathcal{G}$ is a sub-$\sigma$-algebra of $\mathcal{F}$,
\begin{equation}
\mathrm{H}_{\rho}(\mathcal{E}\mid\mathcal{F})\le\mathrm{H}_{\rho}(\mathcal{E}\mid\mathcal{G})\:.\label{eq:monoto of ent wrt sig-alg}
\end{equation}
If $L^{1}(\rho)$ is separable as a metric space, and $\mathcal{C}$
is another finite measurable partition of $X$,
\begin{equation}
\mathrm{I}_{\rho}(\mathcal{E}\vee\mathcal{C}\mid\mathcal{F})=\mathrm{I}_{\rho}(\mathcal{E}\mid\mathcal{F})+\mathrm{I}_{\rho}(\mathcal{C}\mid\mathcal{F}\vee\widehat{\mathcal{E}})\:.\label{eq:cond info formula}
\end{equation}
Here $\mathcal{E}\vee\mathcal{C}$ is the common refinement of $\mathcal{E}$
and $\mathcal{C}$, and $\widehat{\mathcal{E}}$ is the $\sigma$-algebra
generated by $\mathcal{E}$. Integrating the last equality we obtain,
\begin{equation}
\mathrm{H}_{\rho}(\mathcal{E}\vee\mathcal{C}\mid\mathcal{F})=\mathrm{H}_{\rho}(\mathcal{E}\mid\mathcal{F})+\mathrm{H}_{\rho}(\mathcal{C}\mid\mathcal{F}\vee\widehat{\mathcal{E}})\:.\label{eq:cond ent formula}
\end{equation}
If $T:X\rightarrow X$ is measure preserving,
\begin{equation}
\mathrm{I}_{\rho}(\mathcal{E}\mid\mathcal{F})\circ T=\mathrm{I}_{\rho}(T^{-1}\mathcal{E}\mid T^{-1}\mathcal{F})\:.\label{eq:info comp MP map}
\end{equation}

\subsection{\label{subsec:Disintegration}Disintegration of measures}

We give the necessary facts regarding disintegration of measures.
For more details see \cite[Section 5]{EW}.

We call a measurable space $(X,\mathcal{B})$ a Borel space, if $X$
is a Borel subset of a compact metric space $\overline{X}$ and $\mathcal{B}$
is the restriction of the Borel $\sigma$-algebra of $\overline{X}$
to $X$. We denote the collection of probability measures on $(X,\mathcal{B})$
by $\mathcal{M}(X)$. If $\rho\in\mathcal{M}(X)$ we say that $(X,\mathcal{B},\rho)$
is a Borel probability space.

Suppose $(X,\mathcal{B})$ is a Borel space. Given a partition $\xi$
of $X$ into measurable sets and $x\in X$, we write $\xi(x)$ for
the unique element of $\xi$ which contains $x$. A subset $F$ of
$X$ is said to be $\xi$-saturated if it contains $\xi(x)$ for every
$x\in F$. The sub-$\sigma$-algebra of $\mathcal{B}$ determined
by $\xi$ is denoted $\widehat{\xi}$ and defined by,
\[
\widehat{\xi}=\{F\in\mathcal{B}\::\:F\text{ is \ensuremath{\xi}-saturated}\}\:.
\]
We say that $\xi$ is a measurable partition if it is generated by
a countable collection of measurable sets. That is if there exist
$F_{1},F_{2},...\in\mathcal{B}$ such that,
\[
\xi(x)=\bigcap_{x\in F_{n}}F_{n}\cap\bigcap_{x\notin F_{n}}(X\setminus F_{n})\text{ for all }x\in X\:.
\]
If $(Y,\mathcal{F})$ is another Borel space, $\zeta$ is a measurable
partition of $Y$ and $\varphi:X\rightarrow Y$ is measurable, then
\[
\varphi^{-1}\zeta:=\{\varphi^{-1}\zeta(y)\::\:y\in Y\}
\]
is easily seen to be a measurable partition of $X$. In particular
this is the case for the partition $\{\varphi^{-1}\{y\}\}_{y\in Y}$
into level sets of $\varphi$.
\begin{thm}
\label{thm:def of cond measures}Let $(X,\mathcal{B},\rho)$ be a
Borel probability space and let $\xi$ be a measurable partition of
$X$. Then there exists a collection $\{\rho_{x}^{\xi}\}_{x\in X}\subset\mathcal{M}(X)$
such that,
\begin{enumerate}
\item for every $f\in L^{1}(\rho)$,
\[
\int f\:d\rho_{x}^{\xi}=\mathrm{E}_{\rho}(f\mid\widehat{\xi})(x)\text{ for }\rho\text{-a.e. }x;
\]
\item $\rho_{x}^{\xi}(\xi(x))=1$ for $x\in X$;
\item $\rho_{x}^{\xi}=\rho_{y}^{\xi}$ for $x,y\in X$ with $\xi(x)=\xi(y)$.
\end{enumerate}
Moreover, these properties uniquely determine $\{\rho_{x}^{\xi}\}_{x\in X}$
up to a set of zero $\rho$-measure. We call the collection $\{\rho_{x}^{\xi}\}_{x\in X}$
the disintegration of $\rho$ with respect to the partition $\xi$.
\end{thm}

\begin{lem}
\label{lem:slices of slices}Let $(X,\mathcal{B},\rho)$ be a Borel
probability space and let $\xi$ and $\zeta$ be a measurable partitions
of $X$. Suppose that $\xi$ is finer that $\zeta$, that is $\xi(x)\subset\zeta(x)$
for all $x\in X$. Then for $\rho$-a.e. $x$,
\[
(\rho_{y}^{\zeta})_{y}^{\xi}=(\rho_{x}^{\zeta})_{y}^{\xi}=\rho_{y}^{\xi}\text{ for }\rho_{x}^{\zeta}\text{-a.e. }y\:.
\]
\end{lem}

\begin{lem}
\label{lem:push of slices}Let $(X,\mathcal{B},\rho)$ and $(Y,\mathcal{F},\tau)$
be Borel probability spaces, let $T:X\rightarrow Y$ be measure preserving
and let $\xi$ be a measurable partition of $Y$. Then,
\[
T\rho_{x}^{T^{-1}\xi}=\tau_{Tx}^{\xi}\text{ for }\rho\text{-a.e. }x\:.
\]
\end{lem}

We say that a complete separable metric space $Y$ is a Besicovitch
space if the Besicovitch covering lemma (see e.g. \cite{Mat}) holds
in $Y$. Besicovitch spaces include, for instance, Euclidean spaces
and compact finite-dimensional Riemannian manifolds. The following
lemma is stated in \cite[Lemma 2.5]{Fe}. Its proof for the case $Y=\mathbb{R}^{d}$
is given in \cite[Lemma 3.3]{FH}. Recall the notation $B^{\phi}(x,r)$
from Section \ref{subsec:General-notations}.
\begin{lem}
\label{lem:con exp as lim}Let $\phi:X\rightarrow Y$ be a measurable
mapping from a Borel probability space $(X,\mathcal{B}_{X},\rho)$
to a Besicovitch space $Y$. Denote by $\mathcal{B}_{Y}$ the Borel
$\sigma$-algebra of $Y$. Let $\xi$ be a measurable partition of
$X$ and let $A\in\mathcal{B}_{X}$. Then for $\rho$-a.e. $x\in X$,
\[
\underset{r\downarrow0}{\lim}\:\frac{\rho_{x}^{\xi}(B^{\phi}(x,r)\cap A)}{\rho_{x}^{\xi}(B^{\phi}(x,r))}=\mathrm{E}_{\rho}(1_{A}\mid\widehat{\xi}\vee\phi^{-1}(\mathcal{B}_{Y}))(x)\:.
\]
\end{lem}

\subsection{A metric on the projective space}

Recall that $V$ is a real vector space with $2\le\dim V<\infty$.
Fix an inner product $\left\langle \cdot,\cdot\right\rangle $ on
$V$, and denote its induced norm by $|\cdot|$. Given a linear subspace
$W$ of $V$, write $\mathrm{P}(W)$ for its projective space and
$P_{W}$ for the orthogonal projection onto $W$ (by definition $\mathrm{P}(\{0\})=\emptyset$).
For $0\ne x\in V$ denote by $\overline{x}$ the unique element of
$\mathrm{P}(V)$ which contains $x$. For $0\le k\le d$ write $\mathrm{Gr}(k,V)$
for Grassmannian manifold of $k$-dimensional linear subspaces of
$V$.

Let $V^{*}$ be the dual of $V$. Denote by $\mathrm{A}^{2}(V)$ the
vector space of alternating $2$-forms on $V^{*}$. Let $\left\langle \cdot,\cdot\right\rangle $
be the inner product on $\mathrm{A}^{2}(V)$ which satisfies,
\[
\left\langle x_{1}\wedge x_{2},y_{1}\wedge y_{2}\right\rangle =\det\left(\begin{array}{cc}
\left\langle x_{1},y_{1}\right\rangle  & \left\langle x_{1},y_{2}\right\rangle \\
\left\langle x_{2},y_{1}\right\rangle  & \left\langle x_{2},y_{2}\right\rangle 
\end{array}\right)\text{ for }x_{1},x_{2},y_{1},y_{2}\in V\:.
\]
We denote the norm induced by this inner product by $\Vert\cdot\Vert$.
Given an endomorphism $T$ of $V$, write $\mathrm{A}^{2}T$ for the
endomorphism of $\mathrm{A}^{2}(V)$ which satisfies,
\begin{equation}
\mathrm{A}^{2}T(x\wedge y)=(Tx)\wedge(Ty)\text{ for }x,y\in V\:.\label{eq:def of end on alt}
\end{equation}
It is easy to verify that if $P:V\rightarrow V$ is an orthogonal
projection, then $\mathrm{A}^{2}P$ is also an orthogonal projection
(defined on $\mathrm{A}^{2}(V)$).

For $\overline{x},\overline{y}\in\mathrm{P}(V)$ write,
\[
d(\overline{x},\overline{y})=\left(1-\left\langle x,y\right\rangle ^{2}\right)^{1/2},
\]
where $x\in\overline{x}$ and $y\in\overline{y}$ are unit vectors.
It is easy to verify that this defines a metric on $\mathrm{P}(V)$.
Note that,
\[
d(\overline{x},\overline{y})=|x|^{-1}|y|^{-1}\Vert x\wedge y\Vert\text{ for any }0\ne x\in\overline{x}\text{ and }0\ne y\in\overline{y}\:.
\]
For a subset $Y\subset\mathrm{P}(V)$ we set,
\[
d(\overline{x},Y)=\inf\{d(\overline{x},\overline{y})\::\:\overline{y}\in Y\}\:.
\]

The following simple lemma will be used is Section \ref{sec:Proof-of-results}. 
\begin{lem}
\label{lem:ub on dist of proj}Let $W\ne\{0\}$ be a proper linear
subspace of $V$ and let $\overline{x},\overline{y}\in\mathrm{P}(V)$.
Suppose that $\overline{x},\overline{y}\notin\mathrm{P}(W^{\perp})$,
then
\[
d(P_{W}\overline{x},P_{W}\overline{y})\le d(\overline{x},\mathrm{P}(W^{\perp}))^{-1}d(\overline{y},\mathrm{P}(W^{\perp}))^{-1}d(\overline{x},\overline{y})\:.
\]
\end{lem}

\begin{proof}
Let $x\in\overline{x}$ and $y\in\overline{y}$ be with $|x|=|y|=1$.
If $\overline{x}\in\mathrm{P}(W)$,
\[
|P_{W}x|=|x|=1=d(\overline{x},\mathrm{P}(W^{\perp}))\:.
\]
If $\overline{x}\notin\mathrm{P}(W)$,
\begin{multline*}
d(\overline{x},\mathrm{P}(W^{\perp}))\le d(\overline{x},P_{W^{\perp}}\overline{x})=|P_{W^{\perp}}x|^{-1}\Vert x\wedge P_{W^{\perp}}x\Vert\\
=|P_{W^{\perp}}x|^{-1}\Vert P_{W}x\wedge P_{W^{\perp}}x\Vert=|P_{W}x|\cdot d(P_{W}\overline{x},P_{W^{\perp}}\overline{x})=|P_{W}x|\:.
\end{multline*}
Similarly we always have $|P_{W}y|\ge d(\overline{y},\mathrm{P}(W^{\perp}))$.
Additionally, since $\mathrm{A}^{2}P_{W}$ is an orthogonal projection,
\[
\Vert P_{W}x\wedge P_{W}y\Vert=\Vert\mathrm{A}^{2}P_{W}(x\wedge y)\Vert\le\Vert x\wedge y\Vert\:.
\]
Hence,
\begin{eqnarray*}
d(P_{W}\overline{x},P_{W}\overline{y}) & = & |P_{W}x|^{-1}|P_{W}y|^{-1}\Vert P_{W}x\wedge P_{W}y\Vert\\
 & \le & d(\overline{x},\mathrm{P}(W^{\perp}))^{-1}d(\overline{y},\mathrm{P}(W^{\perp}))^{-1}\Vert x\wedge y\Vert\\
 & = & d(\overline{x},\mathrm{P}(W^{\perp}))^{-1}d(\overline{y},\mathrm{P}(W^{\perp}))^{-1}d(\overline{x},\overline{y}),
\end{eqnarray*}
which completes the proof of the lemma.
\end{proof}

\subsection{\label{subsec:boundary map}The Furstenberg measure and the boundary
map}

Recall from Section \ref{subsec:Dimension-formulas} that $\Lambda$
is a finite index set, $\{A_{l}\}_{l\in\Lambda}$ are distinct elements
of $\mathrm{GL}(V)$, $p=(p_{l})_{l\in\Lambda}$ is a probability
vector with strictly positive coordinates and,
\[
\mu=\sum_{l\in\Lambda}p_{l}\delta_{A_{l}}\in\mathcal{M}(\mathrm{GL}(V))\:.
\]
As before, let $S_{\mu}$ be the smallest closed subsemigroup of $\mathrm{GL}(V)$
such that $\mu(S_{\mu})=1$. We shall always assume from now on that
$S_{\mu}$ is strongly irreducible and proximal. Let $\nu$ be the
Furstenberg measure corresponding to $\mu$, which means that $\nu$
is unique $\mu$-stationary member of $\mathcal{M}(\mathrm{P}(V))$.

Write $\Omega=\Lambda^{\mathbb{Z}}$ and let $\sigma:\Omega\rightarrow\Omega$
be the left shift map. That is,
\[
(\sigma\omega)_{n}=\omega_{n+1}\text{ for }\omega\in\Omega\text{ and }n\in\mathbb{Z}\:.
\]
Denote by $\mathcal{P}$ the partition of $\Omega$ according to the
$0$-coordinate, i.e.
\[
\mathcal{P}=\{\{\omega\in\Omega\::\:\omega_{0}=l\}\::\:l\in\Lambda\}\:.
\]
For integers $m\le n$ set,
\[
\mathcal{P}_{m}^{n}=\bigvee_{j=m}^{n}\sigma^{-j}\mathcal{P}\:.
\]
The atoms of these partitions are called the cylinder sets of $\Omega$.
We equip $\Omega$ with the $\sigma$-algebra generated by its cylinder
sets, which makes it into a Borel space. Let $\beta$ be the Bernoulli
measure on $\Omega$ corresponding to the probability vector $p$,
that is $\beta=p^{\mathbb{Z}}$. The triple $(\Omega,\beta,\sigma)$
is an invertible ergodic measure preserving system.

As mentioned in Section \ref{subsec:Dimension-formulas}, from our
assumptions on $S_{\mu}$ we obtain the following statement. Given
a finite word $l_{1}...l_{n}$ over the alphabet $\Lambda$, we write
$A_{l_{1}...l_{n}}$ in place of $A_{l_{1}}\cdot\cdot\cdot A_{l_{n}}$.
\begin{thm}
\label{thm:def of bd map}There exist a Borel set $\Omega_{0}\subset\Omega$,
with $\sigma(\Omega_{0})=\Omega_{0}$ and $\beta(\Omega_{0})=1$,
and a Borel map $\pi:\Omega_{0}\rightarrow\mathrm{P}(V)$, called
the Furstenberg boundary map, such that:
\begin{enumerate}
\item $\pi$ depends only on the nonnegative coordinates of $\Omega$;
\item \label{enu:equi of pi}$\pi\omega=A_{\omega_{0}}\pi\sigma\omega$
for $\omega\in\Omega_{0}$;
\item the distribution of $\pi$ with respect to $\beta$ is equal to $\nu$,
that is $\pi\beta=\nu$;
\item for every $\omega\in\Omega_{0}$,
\[
\underset{n\rightarrow\infty}{\lim}\:A_{\omega_{0}...\omega_{n}}\nu=\delta_{\pi(\omega)}\text{ in the weak-* topology}\:.
\]
\end{enumerate}
\end{thm}

The following theorem, due to Guivarc'h \cite[Theorem 7']{Gu}, is
used in Lemma \ref{lem:bd on mass near hyperplane} to bound the mass
given by $\nu=\pi\beta$ to neighbourhoods of projective spaces of
hyperplanes. A proof of this theorem can also be found in \cite[Theorem 14.1]{BQ}.
\begin{thm}
\label{thm:guivarch}Assume, as we do, that $\mu$ is finitely supported
and that $S_{\mu}$ is strongly irreducible and proximal. Then there
exist $0<\alpha\le1$ and $1<C_{0}<\infty$ such that for all $y\in V$
with $|y|=1$,
\[
\int\left(\frac{|x|}{|\left\langle x,y\right\rangle |}\right)^{\alpha}\:d\pi\beta(\overline{x})\le C_{0}\:.
\]
\end{thm}

\begin{rem*}
Theorem \ref{thm:guivarch} remains true if instead of assuming that
$\mu$ is finitely supported it is assumed that it has a finite exponential
moment.
\end{rem*}
\begin{lem}
\label{lem:bd on mass near hyperplane}There exist $0<\alpha\le1$
and $1<C<\infty$ such that for every $W\in\mathrm{Gr}(\dim V-1,V)$
and $r>0$,
\[
\pi\beta\{\overline{x}\::\:d(\overline{x},\mathrm{P}(W))\le r\}\le Cr^{\alpha}\:.
\]
\end{lem}

\begin{proof}
Let $\alpha$ and $C_{0}$ be as in Theorem \ref{thm:guivarch}. Fix
$W\in\mathrm{Gr}(d-1,V)$ and $0<r<1$. Let $y\in W^{\perp}$ be with
$|y|=1$, and $x\in V$ be with $|x|=1$ and $d(\overline{x},\mathrm{P}(W))\le r$.
There exists $w\in W$ which satisfies $|w|=1$ and $d(\overline{x},\overline{w})\le r$.
We have,
\[
r^{2}\ge d(\overline{x},\overline{w})^{2}=1-\left\langle x,w\right\rangle ^{2}=(1-\left\langle x,w\right\rangle )(1+\left\langle x,w\right\rangle )\:.
\]
Thus, by replacing $w$ with $-w$ if necessary, we may assume that
$r^{2}\ge1-\left\langle x,w\right\rangle $. Hence,
\[
|x-w|^{2}=2-2\left\langle x,w\right\rangle \le2r^{2},
\]
and so,
\[
|x|^{-1}\cdot|\left\langle x,y\right\rangle |=|\left\langle x-w,y\right\rangle |\le|x-w|\le2^{1/2}r\:.
\]
From this we get,
\begin{eqnarray*}
\beta\left\{ \omega\::\:d(\pi\omega,\mathrm{P}(W))\le r\right\}  & \le & \pi\beta\left\{ \overline{x}\::\:|x|^{-1}\cdot|\left\langle x,y\right\rangle |\le2^{1/2}r\right\} \\
 & = & \pi\beta\left\{ \overline{x}\::\:|x|^{\alpha}\cdot|\left\langle x,y\right\rangle |^{-\alpha}\ge2^{-\alpha/2}r^{-\alpha}\right\} \\
 & \le & 2^{\alpha/2}r^{\alpha}\cdot\int\left(|x|\cdot|\left\langle x,y\right\rangle |^{-1}\right)^{\alpha}\:d\pi\beta(\overline{x})\\
 & \le & 2^{\alpha/2}C_{0}r^{\alpha},
\end{eqnarray*}
which completes the proof of the lemma with $C=C_{0}2^{\alpha/2}$.
\end{proof}

\subsection{\label{subsec:Oseledets}Oseledets' multiplicative ergodic theorem}

The following statement follows directly from Oseledets theorem (e.g.
see \cite[Section 3]{Ru}), applied to the system $(\Omega,\beta,\sigma^{-1})$
and the matrix cocycle $\omega\rightarrow A_{\omega_{-1}}$, and by
removing a set of zero $\beta$-measure from $\Omega_{0}$ without
changing the notation (while still maintaining $\sigma(\Omega_{0})=\Omega_{0}$).
\begin{thm}
\label{thm:from Oseledets}There exist positive integers $s,d_{0},...,d_{s}$,
with $\dim V=d_{0}+...+d_{s}$, and real numbers $\lambda_{0}>...>\lambda_{s}$,
so that for every $\omega\in\Omega_{0}$ there exist linear subspaces
$E_{\omega}^{0},...,E_{\omega}^{s}\subset V$ such that,
\begin{enumerate}
\item \label{enu:splitting of V}$V=\oplus_{i=0}^{s}E_{\omega}^{i}$ and
$\dim E_{\omega}^{i}=d_{i}$ for $0\le i\le s$;
\item \label{enu:equi of E^i}$E_{\sigma^{-1}\omega}^{i}=A_{\omega_{-1}}E_{\omega}^{i}$
for $0\le i\le s$;
\item \label{enu:asym of forw mul}for $0\le i\le s$ and $0\ne x\in E_{\omega}^{i}$,
\[
\underset{n\rightarrow\infty}{\lim}\:\frac{1}{n}\log\left|A_{\omega_{-n}...\omega_{-1}}x\right|=\lambda_{i},
\]
with uniform convergence on any compact subset of $E_{\omega}^{i}\setminus\{0\}$;
\item \label{enu:asym of back mul}for $0\le i\le s$,
\[
\underset{n\rightarrow\infty}{\lim}\:\frac{1}{n}\:\underset{x\in E_{\sigma^{n}\omega}^{i},|x|=1}{\max}\log\left|A_{\omega_{0}...\omega_{n-1}}x\right|=\lambda_{i},
\]
\[
\underset{n\rightarrow\infty}{\lim}\:\frac{1}{n}\:\underset{x\in E_{\sigma^{n}\omega}^{i},|x|=1}{\min}\log\left|A_{\omega_{0}...\omega_{n-1}}x\right|=\lambda_{i};
\]
\item \label{enu:asym dist of lines}$\underset{n\rightarrow\pm\infty}{\lim}\:\frac{1}{n}\log\kappa(\sigma^{n}\omega)=0$,
where for $\eta\in\Omega_{0}$
\[
\kappa(\eta)=\min\{d(\overline{x},\overline{y})\::\:0\ne x\in\oplus_{i\in I}E_{\eta}^{i},\:0\ne y\in\oplus_{j\in J}E_{\eta}^{j}\text{ and }I\cap J=\emptyset\};
\]
\item the map $\omega\rightarrow E_{\omega}^{i}$ is Borel measurable for
each $0\le i\le s$.
\end{enumerate}
\end{thm}

\begin{rem}
\label{rem:d_0 =00003D1}By the strong irreducibility and proximality
of $S_{\mu}$ it follows that $d_{0}=1$ (see \cite[Theorem III.6.1]{BL}).
This fact will play an important role in our development.
\end{rem}

\begin{rem}
The numbers $\lambda_{0},...,\lambda_{s}$ are called the Lyapunov
exponents corresponding to $\mu$. For $0\le i\le s$ the integer
$d_{i}$ is called the multiplicity of $\lambda_{i}$. The decomposition
$V=\oplus_{i=0}^{s}E_{\omega}^{i}$ is called the Oseledets splitting
of $V$ at $\omega$. The subspaces $E_{\omega}^{0},...,E_{\omega}^{s}$
are called the Oseledets subspaces corresponding to $\omega$.
\end{rem}

\begin{rem}
Property (\ref{enu:asym of back mul}) of Theorem \ref{thm:from Oseledets}
is not stated in \cite{Ru} in its present form. On the other hand,
it follows from the development carried out there that for $\omega\in\Omega_{0}$,
$0\le i\le s$ and $0\ne x\in E_{\omega}^{i}$,
\[
\underset{n\rightarrow\infty}{\lim}\:\frac{1}{n}\log\left|A_{\omega_{n-1}}^{-1}\cdot\cdot\cdot A_{\omega_{0}}^{-1}x\right|=-\lambda_{i},
\]
with uniform convergence on any compact subset of $E_{\omega}^{i}\setminus\{0\}$.
This together with property (\ref{enu:equi of E^i}) easily imply
property (\ref{enu:asym of back mul}).
\end{rem}

For $1\le i\le s$ write $\tilde{\lambda}_{i}=\lambda_{i}-\lambda_{0}$.
For $0\le i\le s$ and $\omega\in\Omega_{0}$ set $V_{\omega}^{i}=\oplus_{k=i+1}^{s}E_{\omega}^{k}$.
Note that $V_{\omega}^{0}$ is of codimension $1$. Also note that
$V_{\omega}^{s}=\{0\}$ and that for $0\le i<s$,
\begin{equation}
V_{\omega}^{i}=\{x\in V\::\:\underset{n\rightarrow\infty}{\lim}\:\frac{1}{n}\log\left|A_{\omega_{-n}...\omega_{-1}}x\right|\le\lambda_{i+1}\}\:.\label{eq:char of V^i_om}
\end{equation}
This shows that the Borel maps $\omega\rightarrow V_{\omega}^{i}$
depend only on the negative coordinates of $\omega$. It is worth
pointing out that this is not true for the Oseledets subspaces $E_{\omega}^{0},...,E_{\omega}^{s}$.

\section{\label{sec:Construction-of-local}Construction of local coordinates}

In this section we use the Oseledets splittings $V=\oplus_{i=0}^{s}E_{\omega}^{i}$
in order to construct coordinate charts $g_{\omega}$ for $\mathrm{P}(V)$,
whose domains are $\mathrm{P}(V)\setminus\mathrm{P}(V_{\omega}^{0})$.
We then derive some useful properties for these charts. First we show
that the domains just mentioned are neighbourhoods of the points $\pi\omega$,
and that $\pi\sigma^{n}\omega$ does not escape exponentially fast
to the boundary of the domains as $n\rightarrow\infty$.

\subsection{Coordinate neighbourhoods for $\pi\omega$}
\begin{lem}
\label{lem:pi om not in P(V^0_om)}We have,
\[
\beta\{\omega\::\:\pi\omega\in\mathrm{P}(V_{\omega}^{0})\}=0\:.
\]
Thus, by removing a subset of zero $\beta$-measure from $\Omega_{0}$
without changing the notation, we may assume that $\pi\omega\notin\mathrm{P}(V_{\omega}^{0})$
for all $\omega\in\Omega_{0}$.
\end{lem}

\begin{proof}
Since $S_{\mu}$ is strongly irreducible we have $\pi\beta(\mathrm{P}(W))=0$
for $W\in\mathrm{Gr}(\dim V-1,V)$ (see \cite[Proposition III.2.3]{BL}).
Recall that $\pi$ depends only on the nonnegative coordinates and
that $\omega\rightarrow V_{\omega}^{0}$ depends only on the negative
coordinates. Thus, since $\beta$ is a Bernoulli measure,
\[
\beta\{\omega\::\:\pi\omega\in\mathrm{P}(V_{\omega}^{0})\}=\int\pi\beta(\mathrm{P}(V_{\omega}^{0}))\:d\beta(\omega)=0,
\]
which is what we wanted.
\end{proof}
\begin{lem}
\label{lem:dist not exp small}For $\beta$-a.e. $\omega\in\Omega_{0}$,
\begin{equation}
\underset{n\rightarrow\infty}{\lim}\:\frac{1}{n}\log d(\pi\sigma^{n}\omega,\mathrm{P}(V_{\sigma^{n}\omega}^{0}))=0\:.\label{eq:dist not exp small}
\end{equation}
Thus, by removing a subset of zero $\beta$-measure from $\Omega_{0}$
without changing the notation, we may assume that (\ref{eq:dist not exp small})
holds for all $\omega\in\Omega_{0}$.
\end{lem}

\begin{proof}
Let $\epsilon>0$ and for $n\ge1$ set,
\[
F_{n}=\{\omega\in\Omega_{0}\::\:d(\pi\omega,\mathrm{P}(V_{\omega}^{0}))<e^{-n\epsilon}\}\:.
\]
As in the proof of the previous lemma, since $\beta$ is a Bernoulli
measure,
\[
\beta(F_{n})=\int\pi\beta\{\overline{x}\::\:d(\overline{x},\mathrm{P}(V_{\omega}^{0}))\le e^{-\epsilon n}\}\:d\beta(\omega)\:.
\]
Let $\alpha$ and $C$ be as in Lemma \ref{lem:bd on mass near hyperplane}.
Then since $\sigma$ preserves $\beta$,
\[
\beta(\sigma^{-n}F_{n})=\beta(F_{n})\le Ce^{-\epsilon\alpha n}\:.
\]
Thus, by the Borel-Cantelli lemma, for $\beta$-a.e. $\omega$ there
exists $N_{\epsilon,\omega}\ge1$ such that,
\[
d(\pi\sigma^{n}\omega,\mathrm{P}(V_{\sigma^{n}\omega}^{0}))\ge e^{-n\epsilon}\text{ for all }n\ge N_{\epsilon,\omega},
\]
which completes the proof of the lemma.
\end{proof}

\subsection{\label{subsec:The-coordinate-maps g}The coordinate maps $g_{\omega}$}

For $\omega\in\Omega_{0}$ and $0\le i\le s$ let $L_{\omega}^{i}$
be the linear projection of $V$ onto $E_{\omega}^{i}$ with respect
to the splitting $\oplus_{k=0}^{s}E_{\omega}^{k}$. That is for $x\in V$,
\[
x=\sum_{i=0}^{s}L_{\omega}^{i}(x)\text{ with }L_{\omega}^{i}(x)\in E_{\omega}^{i}\text{ for }0\le i\le s\:.
\]
Write,
\[
L_{\omega}(x)=(L_{\omega}^{0}(x),...,L_{\omega}^{s}(x))\:.
\]

For every $\omega\in\Omega_{0}$ fix a unit vector $u_{\omega}^{0}$
in $E_{\omega}^{0}$. Recall that $\dim E_{\omega}^{0}=1$, and so
$E_{\omega}^{0}=\mathrm{span}\{u_{\omega}^{0}\}$. Let $f_{\omega}^{0}:V\rightarrow\mathbb{R}$
be the linear functional with,
\[
L_{\omega}^{0}(x)=f_{\omega}^{0}(x)u_{\omega}^{0}\text{ for }x\in V\:.
\]
Note that $f_{\omega}^{0}(x)=0$ if and only if $x\in V_{\omega}^{0}$.
For $\overline{x}\in\mathrm{P}(V)\setminus\mathrm{P}(V_{\omega}^{0})$
set,
\[
g_{\omega}^{i}(\overline{x})=L_{\omega}^{i}(x)/f_{\omega}^{0}(x)\text{ for }0\le i\le s,
\]
and,
\[
g_{\omega}(\overline{x})=(g_{\omega}^{0}(\overline{x}),...,g_{\omega}^{s}(\overline{x}))\:.
\]
Many times we shall use the fact that $L_{\omega}(x)=g_{\omega}(\overline{x})$
for $x\in\overline{x}$ with $f_{\omega}^{0}(x)=1$. Given a vector
$v=(v^{0},...,v^{s})$, with $v^{i}\in E_{\omega}^{i}$ for $0\le i\le s$,
write
\[
\Vert v\Vert_{\infty}=\underset{0\le i\le s}{\max}\:|v^{i}|\:.
\]
Note that $\Vert g_{\omega}(\overline{x})\Vert_{\infty}\ge1$ for
all $\overline{x}\in\mathrm{P}(V)\setminus\mathrm{P}(V_{\omega}^{0})$.

\subsection{Useful properties}

Recall from (\ref{eq:char of V^i_om}) that the subspaces $V_{\omega}^{i}$
can be characterised in terms of the growth rate of $|A_{\omega_{-n}...\omega_{-1}}x|$.
The following lemma provides a similar characterisation for certain
foliations of $\mathrm{P}(V)$, which are defined in terms of the
maps $g_{\omega}$. Recall from Section \ref{subsec:Oseledets} that
for $1\le i\le s$ we write $\tilde{\lambda}_{i}=\lambda_{i}-\lambda_{0}$.
\begin{lem}
\label{lem:g^k equal implies}Let $\omega\in\Omega_{0}$ and $\overline{x},\overline{y}\in\mathrm{P}(V)\setminus\mathrm{P}(V_{\omega}^{0})$
be with $\overline{x}\ne\overline{y}$. Let $0\le i<s$ be such that,
\[
g_{\omega}^{i+1}(\overline{x})\ne g_{\omega}^{i+1}(\overline{y})\text{ and }g_{\omega}^{k}(\overline{x})=g_{\omega}^{k}(\overline{y})\text{ for }0\le k\le i\:.
\]
Then,
\[
\underset{n\rightarrow\infty}{\lim}\:\frac{1}{n}\log(d(A_{\omega_{-n}...\omega_{-1}}\overline{x},A_{\omega_{-n}...\omega_{-1}}\overline{y}))=\tilde{\lambda}_{i+1}\:.
\]
\end{lem}

\begin{proof}
Since $\overline{x},\overline{y}\notin\mathrm{P}(V_{\omega}^{0})$
there exist $x\in\overline{x}$ and $y\in\overline{y}$ with $f_{\omega}^{0}(x)=f_{\omega}^{0}(y)=1$.
Write,
\[
v_{x}=\sum_{k=0}^{i}L_{\omega}^{k}x,\quad w_{x}=\sum_{k=i+1}^{s}L_{\omega}^{k}x,\quad v_{y}=\sum_{k=0}^{i}L_{\omega}^{k}y,\quad w_{y}=\sum_{k=i+1}^{s}L_{\omega}^{k}y\:.
\]
By the definition of $i$ we have $v_{x}=v_{y}$, hence
\begin{eqnarray}
x\wedge y & = & v_{x}\wedge v_{y}+v_{x}\wedge w_{y}+w_{x}\wedge v_{y}+w_{x}\wedge w_{y}\label{eq:x wedge y}\\
 & = & v_{x}\wedge(w_{y}-w_{x})+w_{x}\wedge w_{y}\:.\nonumber 
\end{eqnarray}

For $n\ge1$ set $A_{\omega,n}:=A_{\omega_{-n}...\omega_{-1}}$. Note
that $L_{\omega}^{0}v_{x}\ne0$ and $L_{\omega}^{i+1}(w_{y}-w_{x})\ne0$
by the definition of $i$. Additionally,
\[
L_{\omega}^{k}w_{x}=L_{\omega}^{k}w_{y}=0\text{ for }0\le k\le i\:.
\]
Combining these facts together with part (\ref{enu:asym of forw mul})
of Theorem \ref{thm:from Oseledets} gives,
\begin{equation}
|A_{\omega,n}v_{x}|=e^{n\lambda_{0}+o_{\omega,\overline{x},\overline{y}}(n)},\label{eq:norm A_n v_x}
\end{equation}
\begin{equation}
|A_{\omega,n}(w_{y}-w_{x})|=e^{n\lambda_{i+1}+o_{\omega,\overline{x},\overline{y}}(n)},\label{eq:norm A_n dif w}
\end{equation}
and,
\begin{equation}
\Vert A_{\omega,n}w_{x}\wedge A_{\omega,n}w_{y}\Vert\le e^{2n\lambda_{i+1}+o_{\omega,\overline{x},\overline{y}}(n)}\:.\label{eq:norm wedge A_n w}
\end{equation}

Note that,
\[
v_{x}\in\oplus_{k=0}^{i}E_{\omega}^{k}\quad\text{ and }\quad w_{y}-w_{x}\in\oplus_{k=i+1}^{s}E_{\omega}^{k},
\]
so by part (\ref{enu:equi of E^i}) of Theorem \ref{thm:from Oseledets},
\[
A_{\omega,n}v_{x}\in\oplus_{k=0}^{i}E_{\sigma^{-n}\omega}^{k}\quad\text{ and }\quad A_{\omega,n}(w_{y}-w_{x})\in\oplus_{k=i+1}^{s}E_{\sigma^{-n}\omega}^{k}\:.
\]
Hence by part (\ref{enu:asym dist of lines}) of Theorem \ref{thm:from Oseledets},
\[
1\ge\left\Vert \frac{A_{\omega,n}v_{x}}{|A_{\omega,n}v_{x}|}\wedge\frac{A_{\omega,n}(w_{y}-w_{x})}{|A_{\omega,n}(w_{y}-w_{x})|}\right\Vert \ge\kappa(\sigma^{-n}\omega)=e^{o_{\omega}(n)}\:.
\]
Thus by (\ref{eq:norm A_n v_x}) and (\ref{eq:norm A_n dif w}),
\[
\left\Vert A_{\omega,n}v_{x}\wedge A_{\omega,n}(w_{y}-w_{x})\right\Vert =e^{n(\lambda_{0}+\lambda_{i+1})+o_{\omega,\overline{x},\overline{y}}(n)}\:.
\]
From this, (\ref{eq:x wedge y}) and (\ref{eq:norm wedge A_n w}),
\begin{equation}
\Vert A_{\omega,n}x\wedge A_{\omega,n}y\Vert=e^{n(\lambda_{0}+\lambda_{i+1})+o_{\omega,\overline{x},\overline{y}}(n)}\:.\label{eq:norm A_n x wedge y}
\end{equation}

From $L_{\omega}^{0}x\ne0$, $L_{\omega}^{0}y\ne0$ and part (\ref{enu:asym of forw mul})
of Theorem \ref{thm:from Oseledets},
\[
|A_{\omega,n}x|=e^{n\lambda_{0}+o_{\omega,\overline{x}}(n)}\quad\text{ and }\quad|A_{\omega,n}y|=e^{n\lambda_{0}+o_{\omega,\overline{y}}(n)}\:.
\]
Hence by (\ref{eq:norm A_n x wedge y}),
\begin{eqnarray*}
d(A_{\omega,n}\overline{x},A_{\omega,n}\overline{y}) & = & |A_{\omega,n}x|^{-1}|A_{\omega,n}y|^{-1}\Vert A_{\omega,n}x\wedge A_{\omega,n}y\Vert\\
 & = & e^{n(\lambda_{i+1}-\lambda_{0})+o_{\omega,\overline{x},\overline{y}}(n)}\:.
\end{eqnarray*}
Since $\tilde{\lambda}_{i+1}=\lambda_{i+1}-\lambda_{0}$, this completes
the proof of the lemma.
\end{proof}
In the remaining part of this section we show that $d(\overline{x},\mathrm{P}(V_{\omega}^{0}))$
is comparable with $\Vert g_{\omega}(\overline{x})\Vert_{\infty}^{-1}$
in a manner depending on $\kappa(\omega)$. For this we need the following
lemma.
\begin{lem}
\label{lem:lb on norm}Let $\omega\in\Omega_{0}$ and $x\in V$ be
given. Then,
\[
|x|\ge2^{-s/2}\kappa(\omega)^{s}\Vert L_{\omega}(x)\Vert_{\infty}\:.
\]
\end{lem}

\begin{proof}
For $0\le k\le s$ set $x_{k}=\sum_{i=0}^{k}L_{\omega}^{i}(x)$. We
show by induction that for every $0\le k\le s$,
\begin{equation}
|x_{k}|\ge2^{-k/2}\kappa(\omega)^{k}\Vert L_{\omega}(x_{k})\Vert_{\infty}\:.\label{eq:ind claim}
\end{equation}
Since $x_{s}=x$ this will prove the lemma. Note that $\Vert L_{\omega}(x_{0})\Vert_{\infty}=|x_{0}|$,
hence (\ref{eq:ind claim}) holds for $k=0$.

Let $0\le k<s$ be such that (\ref{eq:ind claim}) is satisfied for
$k$. If $L_{\omega}^{k+1}(x)=0$ then $x_{k+1}=x_{k}$, and so (\ref{eq:ind claim})
holds also for $k+1$. If $x_{k}=0$ then $|x_{k+1}|=\Vert L_{\omega}(x_{k+1})\Vert_{\infty}$,
and so (\ref{eq:ind claim}) clearly holds for $k+1$. It follows
that we may assume that $L_{\omega}^{k+1}(x)\ne0$ and $x_{k}\ne0$.

Write $u=L_{\omega}^{k+1}(x)/|L_{\omega}^{k+1}(x)|$ and $v=x_{k}/|x_{k}|$.
From $u\in E_{\omega}^{k+1}$, $v\in\oplus_{i=0}^{k}E_{\omega}^{i}$
and the definition of $\kappa(\omega)$,
\[
1-\left|\left\langle u,v\right\rangle \right|=\frac{d(\overline{u},\overline{v})^{2}}{1+\left|\left\langle u,v\right\rangle \right|}\ge\kappa(\omega)^{2}/2\:.
\]
Thus,
\begin{eqnarray*}
|x_{k+1}|^{2} & = & \left\langle x_{k}+L_{\omega}^{k+1}(x),x_{k}+L_{\omega}^{k+1}(x)\right\rangle \\
 & \ge & |x_{k}|^{2}+|L_{\omega}^{k+1}(x)|^{2}-2|x_{k}|\cdot|L_{\omega}^{k+1}(x)|\cdot\left|\left\langle v,u\right\rangle \right|\\
 & = & (|x_{k}|^{2}+|L_{\omega}^{k+1}(x)|^{2})(1-\left|\left\langle v,u\right\rangle \right|)+\left|\left\langle v,u\right\rangle \right|(|x_{k}|-|L_{\omega}^{k+1}(x)|)^{2}\\
 & \ge & 2^{-1}\kappa(\omega)^{2}(|x_{k}|^{2}+|L_{\omega}^{k+1}(x)|^{2})\:.
\end{eqnarray*}
From this and since (\ref{eq:ind claim}) holds for $k$,
\begin{eqnarray*}
|x_{k+1}|^{2} & \ge & 2^{-1}\kappa(\omega)^{2}(2^{-k}\kappa(\omega)^{2k}\Vert L_{\omega}(x_{k})\Vert_{\infty}^{2}+|L_{\omega}^{k+1}(x)|^{2})\\
 & \ge & 2^{-k-1}\kappa(\omega)^{2(k+1)}\Vert L_{\omega}(x_{k+1})\Vert_{\infty}^{2}\:.
\end{eqnarray*}
This shows that (\ref{eq:ind claim}) holds for $k+1$, which completes
the induction and the proof of the lemma.
\end{proof}
\begin{lem}
\label{lem:ub on dist to hyperplane}Let $\omega\in\Omega_{0}$ and
$\overline{x}\in\mathrm{P}(V)\setminus\mathrm{P}(V_{\omega}^{0})$
be given. Then,
\[
d(\overline{x},\mathrm{P}(V_{\omega}^{0}))\le s2^{s}\kappa(\omega)^{-2s}\Vert g_{\omega}(\overline{x})\Vert_{\infty}^{-1}\:.
\]
\end{lem}

\begin{proof}
Write $M$ for $\Vert g_{\omega}(\overline{x})\Vert_{\infty}$. If
$M=1$ the statement is trivial, so we may assume that $M>1$. Let
$x\in\overline{x}$ be with $f_{\omega}^{0}(x)=1$ and set $y=\sum_{i=1}^{s}L_{\omega}^{i}(x)$.
From $f_{\omega}^{0}(x)=1$ and $M>1$ it follows that $g_{\omega}(\overline{x})=L_{\omega}(x)$
and $\Vert L_{\omega}(y)\Vert_{\infty}=M$. Thus by Lemma \ref{lem:lb on norm},
\[
|x|,|y|\ge2^{-s/2}\kappa(\omega)^{s}M\:.
\]
Additionally,
\[
x\wedge y=(u_{\omega}^{0}+y)\wedge y=u_{\omega}^{0}\wedge y=\sum_{i=1}^{s}(u_{\omega}^{0}\wedge L_{\omega}^{i}(x))\:.
\]
Hence,
\[
\Vert x\wedge y\Vert\le\sum_{i=1}^{s}\Vert u_{\omega}^{0}\wedge L_{\omega}^{i}(x)\Vert\le sM\:.
\]
From $y\in V_{\omega}^{0}$ and these estimates we obtain,
\[
d(\overline{x},\mathrm{P}(V_{\omega}^{0}))\le d(\overline{x},\overline{y})=|x|^{-1}|y|^{-1}\Vert x\wedge y\Vert\le s2^{s}\kappa(\omega)^{-2s}M^{-1},
\]
which completes the proof of the lemma.
\end{proof}
\begin{lem}
\label{lem:lb on dist to hyperplane}Let $\omega\in\Omega_{0}$ and
$\overline{x}\in\mathrm{P}(V)\setminus\mathrm{P}(V_{\omega}^{0})$
be given. Then,
\[
d(\overline{x},\mathrm{P}(V_{\omega}^{0}))\ge(2s)^{-1}\kappa(\omega)\Vert g_{\omega}(\overline{x})\Vert_{\infty}^{-1}\:.
\]
\end{lem}

\begin{proof}
Let $x\in\overline{x}$ be with $f_{\omega}^{0}(x)=1$, $y\in V_{\omega}^{0}$
be with $|y|=1$ and $z\in(V_{\omega}^{0})^{\perp}$ be with $|z|=1$.
Recall that $V_{\omega}^{0}$ is of codimension $1$, so $z$ spans
$(V_{\omega}^{0})^{\perp}$. For every $w\in V_{\omega}^{0}$ with
$0<|w|\le1$,
\[
1-\left|\left\langle u_{\omega}^{0},w/|w|\right\rangle \right|=\frac{d(\overline{u_{\omega}^{0}},\overline{w})^{2}}{1+\left|\left\langle u_{\omega}^{0},w/|w|\right\rangle \right|}\ge\kappa(\omega)^{2}/2\:.
\]
Hence,
\[
\left|\left\langle u_{\omega}^{0},w\right\rangle \right|\le\left|\left\langle u_{\omega}^{0},w/|w|\right\rangle \right|\le1-\kappa(\omega)^{2}/2\:.
\]
From this and,
\[
u_{\omega}^{0}=\left\langle u_{\omega}^{0},z\right\rangle z+P_{V_{\omega}^{0}}u_{\omega}^{0},
\]
it follows,
\[
1=|u_{\omega}^{0}|^{2}=\left\langle u_{\omega}^{0},z\right\rangle ^{2}+\left\langle u_{\omega}^{0},P_{V_{\omega}^{0}}u_{\omega}^{0}\right\rangle \le\left\langle u_{\omega}^{0},z\right\rangle ^{2}+1-\kappa(\omega)^{2}/2\:.
\]
This together with $f_{\omega}^{0}(x)=1$ implies,
\[
\left|\left\langle x,z\right\rangle \right|=\left|\left\langle u_{\omega}^{0},z\right\rangle \right|\ge\kappa(\omega)/2\:.
\]
Now since $\Vert z\wedge y\Vert=1$, $\left\langle y,z\right\rangle =0$
and $|y|=1$,
\[
\Vert x\wedge y\Vert\ge\left|\left\langle x\wedge y,z\wedge y\right\rangle \right|=\left|\det\left(\begin{array}{cc}
\left\langle x,z\right\rangle  & \left\langle x,y\right\rangle \\
\left\langle y,z\right\rangle  & \left\langle y,y\right\rangle 
\end{array}\right)\right|=|\left\langle x,z\right\rangle |\ge\kappa(\omega)/2\:.
\]
From this and since,
\[
|x|\le s\Vert L_{\omega}(x)\Vert_{\infty}=s\Vert g_{\omega}(\overline{x})\Vert_{\infty},
\]
we get,
\[
d(\overline{x},\overline{y})=|x|^{-1}\Vert x\wedge y\Vert\ge(2s)^{-1}\kappa(\omega)\Vert g_{\omega}(\overline{x})\Vert_{\infty}^{-1}\:.
\]
Since $y$ is an arbitrary unit vector in $V_{\omega}^{0}$ this completes
the proof of the lemma.
\end{proof}

\section{\label{sec:measurable partitions}Construction and properties of
measurable partitions}

In this section we construct measurable partitions for $\Omega_{0}$,
similar to the ones defined in \cite[Section 4]{Fe}. We then establish
some useful properties for these partitions and their corresponding
conditional measures.

\subsection{Construction}

Let $\xi_{0}$ be the partition of $\Omega_{0}$ according to the
negative coordinates. That is for $\omega\in\Omega_{0}$,
\[
\xi_{0}(\omega)=\{\eta\in\Omega_{0}\::\:\eta_{j}=\omega_{j}\text{ for }j\le-1\}\:.
\]
Recall that for $0\le i\le s$ the Borel map $\omega\rightarrow V_{\omega}^{i}$
depends only on the negative coordinates (see Section \ref{subsec:Oseledets}).
Thus $V_{\omega}^{i}=V_{\eta}^{i}$ whenever $\xi_{0}(\omega)=\xi_{0}(\eta)$.
Also recall that $\pi\omega\notin\mathrm{P}(V_{\omega}^{0})$ for
all $\omega\in\Omega_{0}$ (see Lemma \ref{lem:pi om not in P(V^0_om)}).
Hence for each $0\le i\le s$ the map which takes $\omega\in\Omega_{0}$
to $P_{(V_{\omega}^{i})^{\perp}}\pi\omega\in\mathrm{P}((V_{\omega}^{i})^{\perp})$
is well defined. It is clear that this map is Borel measurable. For
$1\le i\le s$ let $\xi_{i}$ be the partition of $\Omega_{0}$ such
that for every $\omega\in\Omega_{0}$,
\[
\xi_{i}(\omega)=\{\eta\in\xi_{0}(\omega)\::\:P_{(V_{\omega}^{i})^{\perp}}\pi\eta=P_{(V_{\omega}^{i})^{\perp}}\pi\omega\}\:.
\]
It is easy to see that for every $0\le i\le s$ there exists a Borel
space $Y_{i}$ and a Borel map $\varphi_{i}:\Omega_{0}\rightarrow Y_{i}$,
such that the partition $\{\varphi_{i}^{-1}\{y\}\}_{y\in Y_{i}}$
into the level sets of $\varphi_{i}$ is equal to $\xi_{i}$. Thus
$\xi_{0},...,\xi_{s}$ are measurable partitions (see Section \ref{subsec:Disintegration}).

\subsection{Properties of the partitions}

The following lemma shows that it is possible to describe the partitions
$\xi_{i}$ in terms of the coordinate maps $g_{\omega}$.
\begin{lem}
\label{lem:dif char of xi_i}Let $0\le i\le s$ and $\omega,\eta\in\Omega_{0}$
be given, and suppose that $\xi_{0}(\omega)=\xi_{0}(\eta)$. Then
$\xi_{i}(\omega)=\xi_{i}(\eta)$ if and only if $g_{\omega}^{k}(\pi\eta)=g_{\omega}^{k}(\pi\omega)$
for $0\le k\le i$.
\end{lem}

\begin{proof}
Since $g_{\omega}^{0}(\pi\eta)$ and $g_{\omega}^{0}(\pi\omega)$
are both equal to $u_{\omega}^{0}$, the lemma holds trivially when
$i=0$. Since $V_{\omega}^{s}=\{0\}$, we have $\xi_{s}(\omega)=\xi_{s}(\eta)$
if and only if $\pi\omega=\pi\eta$, which clearly holds if and only
if $g_{\omega}(\pi\eta)=g_{\omega}(\pi\omega)$. Thus the lemma is
also clear when $i=s$, and so we may assume that $1\le i<s$.

Let $x_{\omega}\in\pi\omega$ and $x_{\eta}\in\pi\eta$ be with $f_{\omega}^{0}(x_{\omega})=f_{\omega}^{0}(x_{\eta})=1$.
Since $L_{\omega}^{k}x_{\omega},L_{\omega}^{k}x_{\eta}\in V_{\omega}^{i}$
for each $i<k\le s$,
\begin{equation}
P_{(V_{\omega}^{i})^{\perp}}x_{\omega}=\sum_{k=0}^{i}P_{(V_{\omega}^{i})^{\perp}}L_{\omega}^{k}x_{\omega}\quad\text{ and }\quad P_{(V_{\omega}^{i})^{\perp}}x_{\eta}=\sum_{k=0}^{i}P_{(V_{\omega}^{i})^{\perp}}L_{\omega}^{k}x_{\eta}\:.\label{eq:proj of x's}
\end{equation}
Moreover, it is not hard to see that,
\begin{equation}
P_{(V_{\omega}^{i})^{\perp}}\text{ is injective on }E_{\omega}^{k}\text{ for each }0\le k\le i,\label{eq:proj injective}
\end{equation}
and that,
\begin{equation}
V=V_{\omega}^{i}\oplus(\oplus_{k=0}^{i}P_{(V_{\omega}^{i})^{\perp}}E_{\omega}^{k})\:.\label{eq:V as sum of images}
\end{equation}

We have $\xi_{i}(\omega)=\xi_{i}(\eta)$ if and only if $P_{(V_{\omega}^{i})^{\perp}}\pi\eta=P_{(V_{\omega}^{i})^{\perp}}\pi\omega$,
which holds if and only if $P_{(V_{\omega}^{i})^{\perp}}x_{\eta}=cP_{(V_{\omega}^{i})^{\perp}}x_{\omega}$
for some $0\ne c\in\mathbb{R}$. By (\ref{eq:proj of x's}) this holds
if and only if,
\[
\sum_{k=0}^{i}P_{(V_{\omega}^{i})^{\perp}}L_{\omega}^{k}x_{\eta}=\sum_{k=0}^{i}cP_{(V_{\omega}^{i})^{\perp}}L_{\omega}^{k}x_{\omega}\text{ for some }0\ne c\in\mathbb{R}\:.
\]
But by (\ref{eq:proj injective}), (\ref{eq:V as sum of images})
and $f_{\omega}^{0}(x_{\omega})=f_{\omega}^{0}(x_{\eta})=1$ this
holds if and only if there exists $c\ne0$ with $g_{\omega}^{k}(\pi\eta)=cg_{\omega}^{k}(\pi\omega)$
for $0\le k\le i$. Since $g_{\omega}^{0}(\pi\eta)=g_{\omega}^{0}(\pi\omega)=u_{\omega}^{0}$
this completes the proof of the lemma.
\end{proof}
The following lemma describes the partitions $\sigma^{-n}\xi_{i}$.
Its proof is similar to that of \cite[Lemma 4.4(1)]{Fe}. Recall the
finite partitions $\mathcal{P}_{m}^{n}$ from Section \ref{subsec:boundary map}.
\begin{lem}
\label{lem:eq of partitions}Let $\omega\in\Omega_{0}$, $0\le i\le s$
and $n\ge1$ be given. Then,
\[
\xi_{i}(\omega)\cap\mathcal{P}_{0}^{n-1}(\omega)=\sigma^{-n}(\xi_{i}(\sigma^{n}\omega))\:.
\]
As a consequence,
\[
\xi_{i}\vee\mathcal{P}_{0}^{n-1}=\sigma^{-n}\xi_{i}\:.
\]
\end{lem}

\begin{proof}
Let $\eta\in\xi_{i}(\omega)\cap\mathcal{P}_{0}^{n-1}(\omega)$, then
$\eta_{j}=\omega_{j}$ for $j<n$ and by Lemma \ref{lem:dif char of xi_i},
\[
g_{\omega}^{k}(\pi\eta)=g_{\omega}^{k}(\pi\omega)\text{ for }0\le k\le i\:.
\]
Thus, since $A_{\omega_{0}...\omega_{n-1}}=A_{\eta_{0}...\eta_{n-1}}$,
by part (\ref{enu:equi of pi}) of Theorem \ref{thm:def of bd map}
and by Lemma \ref{lem:g^k equal implies},
\begin{multline*}
\underset{m\rightarrow\infty}{\lim}\:\frac{1}{n+m}\log d(A_{\omega_{-m}...\omega_{n-1}}\pi\sigma^{n}\eta,A_{\omega_{-m}...\omega_{n-1}}\pi\sigma^{n}\omega)\\
=\underset{m\rightarrow\infty}{\lim}\:\frac{1}{m}\log d(A_{\omega_{-m}...\omega_{-1}}\pi\eta,A_{\omega_{-m}...\omega_{-1}}\pi\omega)\le\tilde{\lambda}_{i+1},
\end{multline*}
(where $\tilde{\lambda}_{s+1}$ is interpreted as $-\infty$ in the
case $i=s$). Now another application of Lemma \ref{lem:g^k equal implies}
gives,
\[
g_{\sigma^{n}\omega}^{k}(\pi\sigma^{n}\eta)=g_{\sigma^{n}\omega}^{k}(\pi\sigma^{n}\omega)\text{ for }0\le k\le i\:.
\]
This together with Lemma \ref{lem:dif char of xi_i} implies $\sigma^{n}\eta\in\xi_{i}(\sigma^{n}\omega)$,
which shows
\[
\xi_{i}(\omega)\cap\mathcal{P}_{0}^{n-1}(\omega)\subset\sigma^{-n}(\xi_{i}(\sigma^{n}\omega))\:.
\]
The reverse containment is proven similarly, which completes the proof
of the lemma.
\end{proof}
For $n\ge1$ and $\epsilon>0$ set,
\begin{equation}
Q_{n,\epsilon}=\{\omega\in\Omega_{0}\::\:d(\pi\sigma^{j}\omega,\mathrm{P}(V_{\sigma^{j}\omega}^{0}))\ge e^{-j\epsilon}\text{ for }j\ge n\}\:.\label{eq:def of Q_n,eps}
\end{equation}
For $1\le i\le s$, $\omega\in\Omega_{0}$ and $r>0$ write,
\begin{equation}
\Gamma_{i}(\omega,r)=\{\eta\in\xi_{0}(\omega)\::\:d(P_{(V_{\omega}^{i})^{\perp}}\pi\omega,P_{(V_{\omega}^{i})^{\perp}}\pi\eta)\le r\}\:.\label{eq:def of Gamma_i}
\end{equation}
The following proposition, whose statement resembles that of \cite[Lemma 4.4(2)]{Fe},
will be used in Section \ref{sec:Proof-of-results} when we prove
our main result.
\begin{prop}
\label{prop:int cont in ball}Let $1\le j\le s$ and $0\le i<j$ be
given. Then for every $\epsilon>0$ there exists a Borel map $N_{\epsilon}:\Omega_{0}\rightarrow\mathbb{N}$
such that for $\omega\in\Omega_{0}$ and $n\ge N_{\epsilon}(\omega)$,
\[
Q_{n,\epsilon}\cap\xi_{i}(\omega)\cap\mathcal{P}_{0}^{n-1}(\omega)\subset\Gamma_{j}(\omega,e^{n(\tilde{\lambda}_{i+1}+5\epsilon)})\:.
\]
\end{prop}

\begin{proof}
Let $0<\epsilon<-\tilde{\lambda}_{1}/3$ and $\omega\in\Omega_{0}$
be given. Let $n\ge1$ be large with respect to $\epsilon$ and $\omega$
in a manner described during the proof. Since the maps $\pi$ and
$\omega\rightarrow E_{\omega}^{k}$ are all Borel measurable, it will
be clear that the conditions imposed on how large $n$ should be,
are all Borel measurable as well.

Let $\eta\in Q_{n,\epsilon}\cap\xi_{i}(\omega)\cap\mathcal{P}_{0}^{n-1}(\omega)$.
From Lemma \ref{lem:eq of partitions} it follows that $\sigma^{n}\eta\in\xi_{i}(\sigma^{n}\omega)$.
Thus $V_{\sigma^{n}\eta}^{0}=V_{\sigma^{n}\omega}^{0}$ and by Lemma
\ref{lem:dif char of xi_i}, 
\begin{equation}
g_{\sigma^{n}\omega}^{k}(\pi\sigma^{n}\omega)=g_{\sigma^{n}\omega}^{k}(\pi\sigma^{n}\eta)\text{ for }0\le k\le i\:.\label{eq:g^k of sig^n eq}
\end{equation}
Let $x_{\sigma^{n}\omega}\in\pi\sigma^{n}\omega$ and $x_{\sigma^{n}\eta}\in\pi\sigma^{n}\eta$
be with $f_{\sigma^{n}\omega}^{0}(x_{\sigma^{n}\omega})=f_{\sigma^{n}\omega}^{0}(x_{\sigma^{n}\eta})=1$.

By part (\ref{enu:asym dist of lines}) of Theorem \ref{thm:from Oseledets},
and by assuming that $n$ is sufficiently large with respect to $\epsilon$
and $\omega$, we get $\kappa(\sigma^{n}\omega)\ge e^{-n\epsilon/(2s)}$.
From $\eta\in Q_{n,\epsilon}$, $V_{\sigma^{n}\eta}^{0}=V_{\sigma^{n}\omega}^{0}$
and Lemma \ref{lem:ub on dist to hyperplane},
\[
e^{-n\epsilon}\le d(\pi\sigma^{n}\eta,\mathrm{P}(V_{\sigma^{n}\omega}^{0}))\le s2^{s}\kappa(\sigma^{n}\omega)^{-2s}\Vert g_{\sigma^{n}\omega}(\pi\sigma^{n}\eta)\Vert_{\infty}^{-1}\:.
\]
Thus,
\begin{equation}
\Vert L_{\sigma^{n}\omega}x_{\sigma^{n}\eta}\Vert_{\infty}=\Vert g_{\sigma^{n}\omega}(\pi\sigma^{n}\eta)\Vert_{\infty}\le s2^{s}e^{2n\epsilon}\:.\label{eq:ub norm of L sig eta}
\end{equation}
From Lemma \ref{lem:dist not exp small}, and by assuming that $n$
is large enough with respect to $\omega$ and $\epsilon$, we get
$d(\pi\sigma^{n}\omega,\mathrm{P}(V_{\sigma^{n}\omega}^{0}))\ge e^{-n\epsilon}$.
Hence the same argument as above gives,
\begin{equation}
\Vert L_{\sigma^{n}\omega}x_{\sigma^{n}\omega}\Vert_{\infty}=\Vert g_{\sigma^{n}\omega}(\pi\sigma^{n}\omega)\Vert_{\infty}\le s2^{s}e^{2n\epsilon}\:.\label{eq:ub norm of L sig om}
\end{equation}

Write $A_{\omega,n}$ for $A_{\omega_{0}...\omega_{n-1}}$, and note
that from $\eta\in\mathcal{P}_{0}^{n-1}(\omega)$ it follows that
$A_{\omega,n}=A_{\eta_{0}...\eta_{n-1}}$. If $u_{\omega}^{0}\notin\mathrm{P}((V_{\omega}^{j})^{\perp})$
then,
\[
1-|P_{V_{\omega}^{j}}u_{\omega}^{0}|^{2}=1-\left|\left\langle u_{\omega}^{0},P_{V_{\omega}^{j}}u_{\omega}^{0}\right\rangle \right|=\frac{d(\overline{u_{\omega}^{0}},\overline{P_{V_{\omega}^{j}}u_{\omega}^{0}})^{2}}{1+\left|\left\langle u_{\omega}^{0},P_{V_{\omega}^{j}}u_{\omega}^{0}\right\rangle \right|}\ge\kappa(\omega)^{2}/2\:.
\]
Thus,
\[
1=|u_{\omega}^{0}|^{2}=|P_{V_{\omega}^{j}}u_{\omega}^{0}|^{2}+|P_{(V_{\omega}^{j})^{\perp}}u_{\omega}^{0}|^{2}\le1-\kappa(\omega)^{2}/2+|P_{(V_{\omega}^{j})^{\perp}}u_{\omega}^{0}|^{2},
\]
which gives $|P_{(V_{\omega}^{j})^{\perp}}u_{\omega}^{0}|\ge\kappa(\omega)/2$.
Note that this inequality holds trivially if $u_{\omega}^{0}\in\mathrm{P}((V_{\omega}^{j})^{\perp})$.
By part (\ref{enu:equi of E^i}) of Theorem \ref{thm:from Oseledets}
it follows that $A_{\omega,n}u_{\sigma^{n}\omega}^{0}\in\overline{u_{\omega}^{0}}$,
hence
\[
|P_{(V_{\omega}^{j})^{\perp}}A_{\omega,n}u_{\sigma^{n}\omega}^{0}|\ge\frac{1}{2}\kappa(\omega)|A_{\omega,n}u_{\sigma^{n}\omega}^{0}|\:.
\]
Now from this, from part (\ref{enu:asym of back mul}) of Theorem
\ref{thm:from Oseledets} and from (\ref{eq:ub norm of L sig eta}),
\begin{eqnarray*}
|P_{(V_{\omega}^{j})^{\perp}}A_{\omega,n}x_{\sigma^{n}\eta}| & = & |P_{(V_{\omega}^{j})^{\perp}}A_{\omega,n}u_{\sigma^{n}\omega}^{0}+\sum_{k=1}^{s}P_{(V_{\omega}^{j})^{\perp}}A_{\omega,n}L_{\sigma^{n}\omega}^{k}x_{\sigma^{n}\eta}|\\
 & \ge & \frac{1}{2}\kappa(\omega)|A_{\omega,n}u_{\sigma^{n}\omega}^{0}|-\sum_{k=1}^{s}|A_{\omega,n}L_{\sigma^{n}\omega}^{k}x_{\sigma^{n}\eta}|\\
 & = & e^{n\lambda_{0}+o_{\omega}(n)}-\sum_{k=1}^{s}e^{n\lambda_{k}+o_{\omega}(n)}|L_{\sigma^{n}\omega}^{k}x_{\sigma^{n}\eta}|\\
 & \ge & e^{n\lambda_{0}+o_{\omega}(n)}-\sum_{k=1}^{s}e^{n\lambda_{k}+o_{\omega}(n)}s2^{s}e^{2n\epsilon}\\
 & = & e^{n\lambda_{0}+o_{\omega}(n)},
\end{eqnarray*}
where the last equality follows from $\epsilon<-\tilde{\lambda}_{1}/3$.
Similarly by using (\ref{eq:ub norm of L sig om}) we obtain,
\[
|P_{(V_{\omega}^{j})^{\perp}}A_{\omega,n}x_{\sigma^{n}\omega}|\ge e^{n\lambda_{0}+o_{\omega}(n)}\:.
\]

Next we estimate the norm of,
\[
\mathrm{A}^{2}P_{(V_{\omega}^{j})^{\perp}}(A_{\omega,n}x_{\sigma^{n}\omega}\wedge A_{\omega,n}x_{\sigma^{n}\eta}),
\]
where $\mathrm{A}^{2}P_{(V_{\omega}^{j})^{\perp}}$ is defined in
(\ref{eq:def of end on alt}). Write,
\[
v_{\sigma^{n}\omega}=\sum_{k=0}^{i}L_{\sigma^{n}\omega}^{k}(x_{\sigma^{n}\omega})\text{ and }w_{\sigma^{n}\omega}=\sum_{k=i+1}^{s}L_{\sigma^{n}\omega}^{k}(x_{\sigma^{n}\omega}),
\]
and similarly,
\[
v_{\sigma^{n}\eta}=\sum_{k=0}^{i}L_{\sigma^{n}\omega}^{k}(x_{\sigma^{n}\eta})\text{ and }w_{\sigma^{n}\eta}=\sum_{k=i+1}^{s}L_{\sigma^{n}\omega}^{k}(x_{\sigma^{n}\eta})\:.
\]
From (\ref{eq:g^k of sig^n eq}) we get $v_{\sigma^{n}\omega}=v_{\sigma^{n}\eta}$.
Hence,
\begin{eqnarray}
x_{\sigma^{n}\omega}\wedge x_{\sigma^{n}\eta} & = & (v_{\sigma^{n}\omega}+w_{\sigma^{n}\omega})\wedge(v_{\sigma^{n}\eta}+w_{\sigma^{n}\eta})\nonumber \\
 & = & v_{\sigma^{n}\omega}\wedge w_{\sigma^{n}\eta}+w_{\sigma^{n}\omega}\wedge v_{\sigma^{n}\eta}+w_{\sigma^{n}\omega}\wedge w_{\sigma^{n}\eta}\:.\label{eq:wedge of sig^n}
\end{eqnarray}
By applying part (\ref{enu:asym of back mul}) of Theorem \ref{thm:from Oseledets}
and then (\ref{eq:ub norm of L sig eta}), it follows that for each
$0\le k\le s$,
\[
|A_{\omega,n}L_{\sigma^{n}\omega}^{k}x_{\sigma^{n}\eta}|=e^{n\lambda_{k}+o_{\omega}(n)}|L_{\sigma^{n}\omega}^{k}x_{\sigma^{n}\eta}|\le e^{n(\lambda_{k}+2\epsilon)+o_{\omega}(n)},
\]
and similarly by (\ref{eq:ub norm of L sig om}),
\[
|A_{\omega,n}L_{\sigma^{n}\omega}^{k}x_{\sigma^{n}\omega}|\le e^{n(\lambda_{k}+2\epsilon)+o_{\omega}(n)}\:.
\]
Thus from (\ref{eq:wedge of sig^n}) we get,
\[
\Vert A_{\omega,n}x_{\sigma^{n}\omega}\wedge A_{\omega,n}x_{\sigma^{n}\eta}\Vert\le e^{n(\lambda_{0}+\lambda_{i+1}+4\epsilon)+o_{\omega}(n)}\:.
\]
Since $P_{(V_{\omega}^{j})^{\perp}}$ is an orthogonal projection
the same holds for $\mathrm{A}^{2}P_{(V_{\omega}^{j})^{\perp}}$.
Hence,
\[
\Vert\mathrm{A}^{2}P_{(V_{\omega}^{j})^{\perp}}(A_{\omega,n}x_{\sigma^{n}\omega}\wedge A_{\omega,n}x_{\sigma^{n}\eta})\Vert\le e^{n(\lambda_{0}+\lambda_{i+1}+4\epsilon)+o_{\omega}(n)}\:.
\]
Now set $T=P_{(V_{\omega}^{j})^{\perp}}A_{\omega,n}$. Then from the
last inequality, from part (\ref{enu:equi of pi}) of Theorem \ref{thm:def of bd map}
and by the lower bounds on $|Tx_{\sigma^{n}\eta}|$ and $|Tx_{\sigma^{n}\omega}|$
obtained above,
\begin{eqnarray*}
d(P_{(V_{\omega}^{j})^{\perp}}\pi\omega,P_{(V_{\omega}^{j})^{\perp}}\pi\eta) & = & d(T\pi\sigma^{n}\omega,T\pi\sigma^{n}\eta)\\
 & = & |Tx_{\sigma^{n}\omega}|^{-1}|Tx_{\sigma^{n}\eta}|^{-1}\Vert Tx_{\sigma^{n}\omega}\wedge Tx_{\sigma^{n}\eta}\Vert\\
 & \le & e^{-2n\lambda_{0}+o_{\omega}(n)}e^{n(\lambda_{0}+\lambda_{i+1}+4\epsilon)+o_{\omega}(n)}\\
 & = & e^{n(\tilde{\lambda}_{i+1}+4\epsilon)+o_{\omega}(n)}\:.
\end{eqnarray*}
Thus, by assuming that $n$ is sufficiently large with respect to
$\omega$ and $\epsilon$ we get,
\[
\eta\in\Gamma_{j}(\omega,e^{n(\tilde{\lambda}_{i+1}+5\epsilon)}),
\]
which completes the proof of the proposition.
\end{proof}
For $0\le i\le s$ we write $\mathrm{H}_{i}$ in place of $\mathrm{H}_{\beta}(\mathcal{P}\mid\widehat{\xi_{i}})$,
where recall that the last expression is the conditional entropy of
$\mathcal{P}$ given the $\sigma$-algebra $\widehat{\xi_{i}}$ (see
Sections \ref{subsec:info and ent} and \ref{subsec:Disintegration}).
It is easy to verify that this definition of $\mathrm{H}_{i}$ is
consistent with the one given at the introduction in (\ref{eq:intro def of H_i})
(see Lemma \ref{lem:H_i same val}). The proof of the following Lemma
is similar to that of \cite[Lemma 4.6]{Fe}.
\begin{lem}
\label{lem:asym of ergo sum of info}Let $0\le i\le s$, then for
$\beta$-a.e. $\omega\in\Omega_{0}$ and each $n\ge1$,
\begin{equation}
-\log\:\beta_{\omega}^{\xi_{i}}(\mathcal{P}_{0}^{n-1}(\omega))=\mathrm{I}_{\beta}(\mathcal{P}_{0}^{n-1}\mid\widehat{\xi_{i}})(\omega)=\sum_{j=0}^{n-1}\mathrm{I}_{\beta}(\mathcal{P}\mid\widehat{\xi_{i}})(\sigma^{j}\omega),\label{eq:=00003Derg sum of info}
\end{equation}
and,
\begin{equation}
-\underset{n}{\lim}\:\frac{1}{n}\log\:\beta_{\omega}^{\xi_{i}}(\mathcal{P}_{0}^{n-1}(\omega))=\mathrm{H}_{i}\;\text{ for }\beta\text{-a.e. }\omega\:.\label{eq:=00003DH_i}
\end{equation}
\end{lem}

\begin{proof}
By Lemma \ref{lem:eq of partitions} we have $\xi_{i}\vee\mathcal{P}=\sigma^{-1}\xi_{i}$.
It is easy to verify that this implies,
\[
\widehat{\xi_{i}}\vee\widehat{\mathcal{P}}\underset{\beta}{=}\sigma^{-1}\widehat{\xi_{i}},
\]
which means that for every $B\in\widehat{\xi_{i}}\vee\widehat{\mathcal{P}}$
there exists $B'\in\sigma^{-1}\widehat{\xi_{i}}$ with $\beta(B\Delta B')=0$
and vice versa. From this, together with (\ref{eq:cond info formula})
and (\ref{eq:info comp MP map}) in Section \ref{subsec:info and ent},
it follows that for $n\ge1$, 
\begin{eqnarray*}
\mathrm{I}_{\beta}(\mathcal{P}_{0}^{n-1}\mid\widehat{\xi_{i}}) & = & \mathrm{I}_{\beta}(\mathcal{P}\mid\widehat{\xi_{i}})+\mathrm{I}_{\beta}(\mathcal{P}_{1}^{n-1}\mid\widehat{\xi_{i}}\vee\widehat{\mathcal{P}})\\
 & = & \mathrm{I}_{\beta}(\mathcal{P}\mid\widehat{\xi_{i}})+\mathrm{I}_{\beta}(\sigma^{-1}\mathcal{P}_{0}^{n-2}\mid\sigma^{-1}\widehat{\xi_{i}})\\
 & = & \mathrm{I}_{\beta}(\mathcal{P}\mid\widehat{\xi_{i}})+\mathrm{I}_{\beta}(\mathcal{P}_{0}^{n-2}\mid\widehat{\xi_{i}})\circ\sigma\:.
\end{eqnarray*}
Iterating this we get,
\[
\mathrm{I}_{\beta}(\mathcal{P}_{0}^{n-1}\mid\widehat{\xi_{i}})=\sum_{j=0}^{n-1}\mathrm{I}_{\beta}(\mathcal{P}\mid\widehat{\xi_{i}})\circ\sigma^{j}\:.
\]
Additionally, by the definitions of the conditional information and
measures (see Theorem \ref{thm:def of cond measures}),
\[
-\log\:\beta_{\omega}^{\xi_{i}}(\mathcal{P}_{0}^{n-1}(\omega))=\mathrm{I}_{\beta}(\mathcal{P}_{0}^{n-1}\mid\widehat{\xi_{i}})(\omega)\text{ for }\beta\text{-a.e. }\omega,
\]
which gives (\ref{eq:=00003Derg sum of info}). Birkhoff's ergodic
theorem combined with (\ref{eq:=00003Derg sum of info}) implies (\ref{eq:=00003DH_i}),
which completes the proof of the lemma.
\end{proof}
The following two lemmas will be used in Section \ref{sec:Proof-of-results}
when we prove our main result. Recall the sets $Q_{n,\epsilon}$ from
(\ref{eq:def of Q_n,eps}).
\begin{lem}
\label{lem:density of Q}For $\epsilon>0$ and $0\le i\le s$,
\[
\underset{n\rightarrow\infty}{\lim}\:\frac{\beta_{\omega}^{\xi_{i}}(Q_{n,\epsilon}\cap\mathcal{P}_{0}^{n-1}(\omega))}{\beta_{\omega}^{\xi_{i}}(\mathcal{P}_{0}^{n-1}(\omega))}=1\text{ for }\beta\text{-a.e. }\omega\in\Omega_{0}\:.
\]
\end{lem}

\begin{proof}
From Lemma \ref{lem:slices of slices} and Theorem \ref{thm:def of cond measures}
it follows that for $n,m\ge1$,
\begin{equation}
\frac{\beta_{\omega}^{\xi_{i}}(Q_{m,\epsilon}\cap\mathcal{P}_{0}^{n-1}(\omega))}{\beta_{\omega}^{\xi_{i}}(\mathcal{P}_{0}^{n-1}(\omega))}=\mathrm{E}_{\beta}(1_{Q_{m,\epsilon}}\mid\widehat{\xi_{i}}\vee\widehat{\mathcal{P}_{0}^{n-1}})(\omega)\text{ for }\beta\text{-a.e. }\omega\:.\label{eq:expre as cond exp}
\end{equation}
Note that the sequence of $\sigma$-algebras,
\[
\left\{ \widehat{\xi_{i}}\vee\widehat{\mathcal{P}_{0}^{n-1}}\right\} _{n\ge1},
\]
increases to the Borel $\sigma$-algebra of $\Omega$. Thus, from
(\ref{eq:expre as cond exp}) and the increasing martingale theorem
(see \cite[Section 2.1]{Pa}),
\[
\underset{n\rightarrow\infty}{\lim}\:\frac{\beta_{\omega}^{\xi_{i}}(Q_{m,\epsilon}\cap\mathcal{P}_{0}^{n-1}(\omega))}{\beta_{\omega}^{\xi_{i}}(\mathcal{P}_{0}^{n-1}(\omega))}=1_{Q_{m,\epsilon}}(\omega)\text{ for }\beta\text{-a.e. }\omega\:.
\]
Additionally we have $Q_{m,\epsilon}\subset Q_{n,\epsilon}$ whenever
$n\ge m$, hence
\[
\underset{n\rightarrow\infty}{\liminf}\:\frac{\beta_{\omega}^{\xi_{i}}(Q_{n,\epsilon}\cap\mathcal{P}_{0}^{n-1}(\omega))}{\beta_{\omega}^{\xi_{i}}(\mathcal{P}_{0}^{n-1}(\omega))}\ge1_{Q_{m,\epsilon}}(\omega)\text{ for }\beta\text{-a.e. }\omega\:.
\]
Now since by Lemma \ref{lem:dist not exp small},
\[
\Omega_{0}=\cup_{m\ge1}Q_{m,\epsilon},
\]
this completes the proof of the lemma.
\end{proof}
Recall the sets $\Gamma_{k}(\omega,r)$ from (\ref{eq:def of Gamma_i}).
\begin{lem}
\label{lem:pos density}Let $1\le k\le s$ and $0\le i\le s$ be given,
and let $F\subset\Omega_{0}$ be a Borel set. Then for $\beta$-a.e.
$\omega\in F$,
\[
\underset{r\downarrow0}{\lim}\:\frac{\beta_{\omega}^{\xi_{i}}(\Gamma_{k}(\omega,r)\cap F)}{\beta_{\omega}^{\xi_{i}}(\Gamma_{k}(\omega,r))}>0\:.
\]
\end{lem}

\begin{proof}
Define a map $\phi_{k}:\Omega_{0}\rightarrow\mathrm{P}(V)$ by $\phi_{k}(\omega)=P_{(V_{\omega}^{k})^{\perp}}\pi\omega$.
Recall that $\pi\omega\notin\mathrm{P}(V_{\omega}^{0})$ for each
$\omega\in\Omega_{0}$, so $\phi_{k}$ is well defined. Since $\pi$
and $\omega\rightarrow V_{\omega}^{k}$ are Borel measurable the same
holds for $\phi_{k}$. It follows directly from the definitions that
for $r>0$ and $\omega\in\Omega_{0}$,
\[
\Gamma_{k}(\omega,r)=\xi_{0}(\omega)\cap B^{\phi_{k}}(\omega,r),
\]
where the notation $B^{\phi_{k}}(\omega,r)$ was defined in Section
\ref{subsec:General-notations}. Write $\mathcal{B}$ for the Borel
$\sigma$-algebra of $\mathrm{P}(V)$. Then by Lemma \ref{lem:con exp as lim}
it follows that for $\beta$-a.e. $\omega$,
\[
\underset{r\downarrow0}{\lim}\:\frac{\beta_{\omega}^{\xi_{i}}(\Gamma_{k}(\omega,r)\cap F)}{\beta_{\omega}^{\xi_{i}}(\Gamma_{k}(\omega,r))}=\mathrm{E}_{\beta}(1_{F}\mid\widehat{\xi_{i}}\vee\phi_{k}^{-1}\mathcal{B})(\omega)\:.
\]
Now, by the definition of the conditional expectation, it follows
easily (see \cite[Lemma 3.10]{FH}) that for $\beta$-a.e. $\omega\in F$,
\[
\mathrm{E}_{\beta}(1_{F}\mid\widehat{\xi_{i}}\vee\phi_{k}^{-1}\mathcal{B})(\omega)>0,
\]
which completes the proof of the lemma.
\end{proof}
The following lemma, which is similar to \cite[Lemma 4.5(3)]{Fe},
will be used in the next section.
\begin{lem}
\label{lem:form for con meas}Let $0\le i\le s$ and $n\ge1$ be given.
Then for $\beta$-a.e. $\omega$ we have for any Borel set $F\subset\Omega_{0}$,
\[
\beta_{\omega}^{\xi_{i}}(\sigma^{-n}F\cap\mathcal{P}_{0}^{n-1}(\omega))=\beta_{\sigma^{n}\omega}^{\xi_{i}}(F)\beta_{\omega}^{\xi_{i}}(\mathcal{P}_{0}^{n-1}(\omega))\:.
\]
\end{lem}

\begin{proof}
By Lemma \ref{lem:push of slices},
\begin{equation}
\sigma^{-n}\beta_{\sigma^{n}\omega}^{\xi_{i}}=\beta_{\omega}^{\sigma^{-n}\xi_{i}}\text{ for }\beta\text{-a.e. }\omega\:.\label{eq:equi of cond meas}
\end{equation}
By Lemmas \ref{lem:eq of partitions} and \ref{lem:slices of slices}
it follows that for $\beta$-a.e. $\omega$,
\[
\beta_{\eta}^{\sigma^{-n}\xi_{i}}=\beta_{\eta}^{\xi_{i}\vee\mathcal{P}_{0}^{n-1}}=(\beta_{\eta}^{\xi_{i}})_{\eta}^{\xi_{i}\vee\mathcal{P}_{0}^{n-1}}\;\text{ for }\beta_{\omega}^{\xi_{i}}\text{-a.e. }\eta\:.
\]
Thus for $\beta$-a.e. $\omega$ and any $F\subset\Omega_{0}$ Borel,
\[
\beta_{\omega}^{\sigma^{-n}\xi_{i}}(F)=\frac{\beta_{\omega}^{\xi_{i}}(F\cap\mathcal{P}_{0}^{n-1}(\omega))}{\beta_{\omega}^{\xi_{i}}(\mathcal{P}_{0}^{n-1}(\omega))}\:.
\]
From this and (\ref{eq:equi of cond meas}) it follows that for $\beta$-a.e.
$\omega$ and any $F\subset\Omega_{0}$ Borel,
\[
\frac{\beta_{\omega}^{\xi_{i}}(\sigma^{-n}F\cap\mathcal{P}_{0}^{n-1}(\omega))}{\beta_{\omega}^{\xi_{i}}(\mathcal{P}_{0}^{n-1}(\omega))}=\beta_{\omega}^{\sigma^{-n}\xi_{i}}(\sigma^{-n}F)=\sigma^{-n}\beta_{\sigma^{n}\omega}^{\xi_{i}}(\sigma^{-n}F)=\beta_{\sigma^{n}\omega}^{\xi_{i}}(F),
\]
which completes the proof of the lemma.
\end{proof}

\section{\label{sec:Transverse-dimensions}Transverse dimensions}

In this section we prove an inequality for the transverse dimensions.
Recall that,
\[
\mathrm{H}_{i}=\mathrm{H}_{\beta}(\mathcal{P}\mid\widehat{\xi_{i}})\text{ for }0\le i\le s,
\]
and that for $1\le i\le s$, $\omega\in\Omega_{0}$ and $r>0$,
\[
\Gamma_{i}(\omega,r)=\{\eta\in\xi_{0}(\omega)\::\:d(P_{(V_{\omega}^{i})^{\perp}}\pi\omega,P_{(V_{\omega}^{i})^{\perp}}\pi\eta)\le r\}\:.
\]
We also set,
\[
\vartheta_{i-1}(\omega)=\underset{r\downarrow0}{\liminf}\:\frac{\log\beta_{\omega}^{\xi_{i-1}}(\Gamma_{i}(\omega,r))}{\log r}\:.
\]
Following \cite{Fe} we call $\vartheta_{0},...,\vartheta_{s-1}$
the transverse dimensions of $\beta$. The purpose of this section
is to prove the following proposition. Its proof is a modification
of that of \cite[Proposition 5.1]{Fe}.
\begin{prop}
\label{prop:transvers dim =00003D}For $1\le i\le s$ and $\beta$-a.e.
$\omega$,
\[
\vartheta_{i-1}(\omega)\ge\frac{\mathrm{H}_{i}-\mathrm{H}_{i-1}}{\tilde{\lambda}_{i}}\:.
\]
\end{prop}

\subsection{\label{subsec:Preparations-for-transv}Preparations for the proof
of Proposition \ref{prop:transvers dim =00003D}}

For $1\le i\le s$, $\omega\in\Omega_{0}$ and $r>0$ set,
\[
T_{i}(\omega,r)=\{\eta\in\xi_{i-1}(\omega)\::\:|g_{\omega}^{i}(\pi\omega)-g_{\omega}^{i}(\pi\eta)|\le r\}\:.
\]
In this subsection we mainly study the relation between the sets $\Gamma_{i}(\omega,r)$
and $T_{i}(\omega,r)$. Later we establish other facts which will
be needed for the proof of Proposition \ref{prop:transvers dim =00003D}.
We start with the following containment.
\begin{lem}
\label{lem:T sub Gamma}Let $1\le i\le s$, $\omega\in\Omega_{0}$
and $r>0$ be given. Then,
\[
T_{i}(\omega,r)\subset\Gamma_{i}\left(\omega,s^{3}2^{3+s}\Vert g_{\omega}\pi\omega\Vert_{\infty}\kappa(\omega)^{-2s-2}r\right)\:.
\]
\end{lem}

\begin{proof}
Let $\eta\in T_{i}(\omega,r)$, and let $x_{\omega}\in\pi\omega$
and $x_{\eta}\in\pi\eta$ be with $f_{\omega}^{0}(x_{\omega})=f_{\omega}^{0}(x_{\eta})=1$.
First we show that,
\begin{equation}
|P_{(V_{\omega}^{i})^{\perp}}x_{\omega}|,|P_{(V_{\omega}^{i})^{\perp}}x_{\eta}|\ge s^{-1}2^{-1-(s/2)}\kappa(\omega)^{s+1}\:.\label{eq:lb on norm of proj}
\end{equation}
We prove this inequality only for $x_{\eta}$, the proof for $x_{\omega}$
is similar. By Lemma \ref{lem:lb on norm},
\[
|x_{\eta}|\ge2^{-s/2}\kappa(\omega)^{s}\Vert g_{\omega}\pi\eta\Vert_{\infty}\:.
\]
If $\pi\eta\in\mathrm{P}((V_{\omega}^{i})^{\perp})$ then since $\Vert g_{\omega}\pi\eta\Vert_{\infty}\ge1$,
\[
|P_{(V_{\omega}^{i})^{\perp}}x_{\eta}|=|x_{\eta}|\ge2^{-s/2}\kappa(\omega)^{s},
\]
and so we may assume that $\pi\eta\notin\mathrm{P}((V_{\omega}^{i})^{\perp})$.
Now since $P_{V_{\omega}^{i}}\pi\eta\in\mathrm{P}(V_{\omega}^{0})$,
\begin{eqnarray*}
d(\pi\eta,\mathrm{P}(V_{\omega}^{0})) & \le & d(\pi\eta,P_{V_{\omega}^{i}}\pi\eta)\\
 & = & |x_{\eta}|^{-1}|P_{V_{\omega}^{i}}x_{\eta}|^{-1}\Vert x_{\eta}\wedge P_{V_{\omega}^{i}}x_{\eta}\Vert\\
 & \le & 2^{s/2}\kappa(\omega)^{-s}\Vert g_{\omega}\pi\eta\Vert_{\infty}^{-1}|P_{V_{\omega}^{i}}x_{\eta}|^{-1}\Vert P_{(V_{\omega}^{i})^{\perp}}x_{\eta}\wedge P_{V_{\omega}^{i}}x_{\eta}\Vert\\
 & = & 2^{s/2}\kappa(\omega)^{-s}\Vert g_{\omega}\pi\eta\Vert_{\infty}^{-1}|P_{(V_{\omega}^{i})^{\perp}}x_{\eta}|\:.
\end{eqnarray*}
This together with Lemma \ref{lem:lb on dist to hyperplane} gives,
\[
(2s)^{-1}\kappa(\omega)\Vert g_{\omega}\pi\eta\Vert_{\infty}^{-1}\le d(\pi\eta,\mathrm{P}(V_{\omega}^{0}))\le2^{s/2}\kappa(\omega)^{-s}\Vert g_{\omega}\pi\eta\Vert_{\infty}^{-1}|P_{(V_{\omega}^{i})^{\perp}}x_{\eta}|,
\]
which implies (\ref{eq:lb on norm of proj}).

Now set $M=\Vert g_{\omega}\pi\omega\Vert_{\infty}$ and let us show
that,
\begin{equation}
\Vert P_{(V_{\omega}^{i})^{\perp}}x_{\omega}\wedge P_{(V_{\omega}^{i})^{\perp}}x_{\eta}\Vert\le2sMr\:.\label{eq:ub on wedge of proj}
\end{equation}
Write $y=\sum_{k=0}^{i-1}g_{\omega}^{k}\pi\omega$. From $\eta\in T_{i}(\omega,r)\subset\xi_{i-1}(\omega)$
and Lemma \ref{lem:dif char of xi_i} it follows,
\[
y=\sum_{k=0}^{i-1}g_{\omega}^{k}\pi\eta\quad\text{ and }\quad|g_{\omega}^{i}\pi\omega-g_{\omega}^{i}\pi\eta|\le r\:.
\]
Since $g_{\omega}^{k}\pi\eta\in V_{\omega}^{i}$ for all $i<k\le s$,
\[
P_{(V_{\omega}^{i})^{\perp}}x_{\eta}=P_{(V_{\omega}^{i})^{\perp}}\sum_{k=0}^{s}g_{\omega}^{k}\pi\eta=P_{(V_{\omega}^{i})^{\perp}}(y+g_{\omega}^{i}\pi\eta)\:.
\]
Similarly,
\[
P_{(V_{\omega}^{i})^{\perp}}x_{\omega}=P_{(V_{\omega}^{i})^{\perp}}(y+g_{\omega}^{i}\pi\omega)\:.
\]
Since $P_{(V_{\omega}^{i})^{\perp}}$ is an orthogonal projection
the same holds for $\mathrm{A}^{2}P_{(V_{\omega}^{i})^{\perp}}$ (which
is defined in (\ref{eq:def of end on alt})). Hence,
\begin{eqnarray*}
\Vert P_{(V_{\omega}^{i})^{\perp}}x_{\omega}\wedge P_{(V_{\omega}^{i})^{\perp}}x_{\eta}\Vert & = & \Vert\mathrm{A}^{2}P_{(V_{\omega}^{i})^{\perp}}((y+g_{\omega}^{i}\pi\omega)\wedge(y+g_{\omega}^{i}\pi\eta))\Vert\\
 & \le & \Vert(y+g_{\omega}^{i}\pi\omega)\wedge(y+g_{\omega}^{i}\pi\eta)\Vert\\
 & \le & \Vert y\wedge(g_{\omega}^{i}\pi\eta-g_{\omega}^{i}\pi\omega)\Vert+\Vert g_{\omega}^{i}\pi\omega\wedge(g_{\omega}^{i}\pi\eta-g_{\omega}^{i}\pi\omega)\Vert\\
 & \le & |y|\cdot|g_{\omega}^{i}\pi\eta-g_{\omega}^{i}\pi\omega|+|g_{\omega}^{i}\pi\omega|\cdot|g_{\omega}^{i}\pi\eta-g_{\omega}^{i}\pi\omega|\\
 & \le & 2sMr\:.
\end{eqnarray*}

Combining (\ref{eq:lb on norm of proj}) with (\ref{eq:ub on wedge of proj})
we obtain,
\begin{eqnarray*}
d(P_{(V_{\omega}^{i})^{\perp}}\pi\omega,P_{(V_{\omega}^{i})^{\perp}}\pi\eta) & = & |P_{(V_{\omega}^{i})^{\perp}}x_{\omega}|^{-1}|P_{(V_{\omega}^{i})^{\perp}}x_{\eta}|^{-1}\Vert P_{(V_{\omega}^{i})^{\perp}}x_{\omega}\wedge P_{(V_{\omega}^{i})^{\perp}}x_{\eta}\Vert\\
 & \le & s^{3}2^{3+s}M\cdot\kappa(\omega)^{-2s-2}r,
\end{eqnarray*}
which completes the proof of the lemma.
\end{proof}
The containment in the other direction, which is proven in Lemma \ref{lem:Gamma sub T},
requires a bit more work. For $\omega\in\Omega_{0}$ and $0\le i\le s$
we write $W_{\omega}^{i}$ for $\oplus_{k=0}^{i}E_{\omega}^{k}$.
\begin{lem}
\label{lem:lb on norm of proj of wedge}Let $1\le i\le s$ be given,
then for every $\epsilon>0$ there exists $\delta=\delta(\epsilon)>0$
such that the following holds. Let $\omega\in\Omega_{0}$ be with
$\kappa(\omega)\ge\epsilon$. Then for every $x,y\in W_{\omega}^{i}$,
\[
\Vert\mathrm{A}^{2}P_{(V_{\omega}^{i})^{\perp}}(x\wedge y)\Vert\ge\delta\Vert x\wedge y\Vert\:.
\]
\end{lem}

\begin{proof}
Since $V_{\omega}^{s}=\{0\}$ for every $\omega\in\Omega_{0}$, the
lemma holds trivially when $i=s$. Assume by contradiction that the
lemma is false for some $1\le i<s$ and $\epsilon>0$. Then for every
$n\ge1$ there exist $\omega_{n}\in\Omega_{0}$ and $x_{n},y_{n}\in W_{\omega_{n}}^{i}$
such that $\kappa(\omega_{n})\ge\epsilon$, $\Vert x_{n}\wedge y_{n}\Vert=1$
and $\Vert\mathrm{A}^{2}P_{(V_{\omega_{n}}^{i})^{\perp}}(x_{n}\wedge y_{n})\Vert\le\frac{1}{n}$.
Note that from $\Vert x_{n}\wedge y_{n}\Vert=1$ it follows that,
\[
x_{n}\wedge y_{n}=|x_{n}|^{-1}|P_{\overline{x_{n}}^{\perp}}y_{n}|^{-1}(x_{n}\wedge P_{\overline{x_{n}}^{\perp}}y_{n})\:.
\]
From this, from $P_{\overline{x_{n}}^{\perp}}y_{n}\in W_{\omega_{n}}^{i}$
and by replacing the vectors $x_{n}$ and $y_{n}$ with the vectors
$|x_{n}|^{-1}x_{n}$ and $|P_{\overline{x_{n}}^{\perp}}y_{n}|^{-1}P_{\overline{x_{n}}^{\perp}}y_{n}$
if necessary, it follows that we may assume to begin with that $|x_{n}|=|y_{n}|=1$.

Recall that by part (\ref{enu:splitting of V}) of Theorem \ref{thm:from Oseledets}
we have $\dim E_{\omega}^{k}=d_{k}$ for $\omega\in\Omega_{0}$ and
$0\le k\le s$. Set $q_{1}=\sum_{k=0}^{i}d_{k}$ and $q_{2}=\sum_{k=i+1}^{s}d_{k}$,
then $W_{\omega_{n}}^{i}\in\mathrm{Gr}(q_{1},V)$ and $V_{\omega_{n}}^{i}\in\mathrm{Gr}(q_{2},V)$
for $n\ge1$. By moving to a subsequence without changing the notation,
we may assume that there exist $W\in\mathrm{Gr}(q_{1},V)$, $U\in\mathrm{Gr}(q_{2},V)$
and $x,y\in W$ such that $W_{\omega_{n}}^{i}\overset{n}{\rightarrow}W$,
$V_{\omega_{n}}^{i}\overset{n}{\rightarrow}U$, $x_{n}\overset{n}{\rightarrow}x$
and $y_{n}\overset{n}{\rightarrow}y$. Since $\kappa(\omega_{n})\ge\epsilon$
for $n\ge1$, it follows from the definition of $\kappa$ that,
\[
d(\overline{w},\overline{u})\ge\epsilon\text{ for all }\overline{w}\in W\text{ and }\overline{u}\in U\:.
\]
In particular we have $V=W\oplus U$. From,
\[
\Vert\mathrm{A}^{2}P_{(V_{\omega_{n}}^{i})^{\perp}}(x_{n}\wedge y_{n})\Vert\le\frac{1}{n}\;\text{ and }\;\Vert x_{n}\wedge y_{n}\Vert=1\text{ for }n\ge1,
\]
it follows that $P_{U^{\perp}}x\wedge P_{U^{\perp}}y=0$ and $x\wedge y\ne0$.

Since $P_{U^{\perp}}x\wedge P_{U^{\perp}}y=0$ there exists $c_{x},c_{y}\in\mathbb{R}$,
not both $0$, such that $c_{x}P_{U^{\perp}}x+c_{y}P_{U^{\perp}}y=0$,
and so $P_{U^{\perp}}(c_{x}x+c_{y}y)=0$. From $x\wedge y\ne0$ it
follows that $c_{x}x+c_{y}y\ne0$, which shows that the restriction
of $P_{U^{\perp}}$ to $W$ is not injective. But this clearly contradicts
$V=W\oplus U$, which completes the proof of the lemma.
\end{proof}
\begin{lem}
\label{lem:lb on wedge}Let $0\le i<s$ be given, then for every $\epsilon>0$
there exists $\delta=\delta(\epsilon)>0$ such that the following
holds. Let $\omega\in\Omega_{0}$, $w\in W_{\omega}^{i}$ and $v_{1},v_{2}\in V_{\omega}^{i}$
be with $\kappa(\omega)\ge\epsilon$, $\epsilon\le|w|\le\epsilon^{-1}$
and $|v_{1}|\le\epsilon^{-1}$. Then,
\[
\Vert(w+v_{1})\wedge v_{2}\Vert\ge\delta|w+v_{1}|\cdot|v_{2}|\:.
\]
\end{lem}

\begin{proof}
Assume by contradiction that the lemma is false for some $0\le i<s$
and $\epsilon>0$. Then for every $n\ge1$ there exist $\omega_{n}\in\Omega_{0}$,
$w_{n}\in W_{\omega_{n}}^{i}$ and $v_{1,n},v_{2,n}\in V_{\omega_{n}}^{i}$
such that $\kappa(\omega_{n})\ge\epsilon$, $\epsilon\le|w_{n}|\le\epsilon^{-1}$,
$|v_{1,n}|\le\epsilon^{-1}$, $|v_{2,n}|=1$ and,
\[
\Vert(w_{n}+v_{1,n})\wedge v_{2,n}\Vert\le\frac{1}{n}|w_{n}+v_{1,n}|\:.
\]
Set $q_{1}=\sum_{k=0}^{i}d_{k}$ and $q_{2}=\sum_{k=i+1}^{s}d_{k}$.
By moving to a subsequence without changing the notation we may assume
that there exist $W\in\mathrm{Gr}(q_{1},V)$, $U\in\mathrm{Gr}(q_{2},V)$,
$w\in W$ and $u_{1},u_{2}\in U$ such that $W_{\omega_{n}}^{i}\overset{n}{\rightarrow}W$,
$V_{\omega_{n}}^{i}\overset{n}{\rightarrow}U$, $w_{n}\overset{n}{\rightarrow}w$,
$v_{1,n}\overset{n}{\rightarrow}u_{1}$ and $v_{2,n}\overset{n}{\rightarrow}u_{2}$.

As in the proof of Lemma \ref{lem:lb on norm of proj of wedge}, since
$\kappa(\omega_{n})\ge\epsilon$ for $n\ge1$ we have $V=W\oplus U$.
From $|w_{n}|\ge\epsilon$ and $|v_{2,n}|=1$ for $n\ge1$ it follows
that $w\ne0$ and $u_{2}\ne0$. Since,
\[
\frac{1}{n}|w_{n}+v_{1,n}|\le2/(\epsilon n)\text{ for }n\ge1,
\]
it is also clear that $(w+u_{1})\wedge u_{2}=0$. Thus there exists
$c\in\mathbb{R}$ such that $w=cu_{2}-u_{1}$. But this contradicts
$V=W\oplus U$ and $w\ne0$, which completes the proof of the lemma.
\end{proof}
\begin{lem}
\label{lem:Gamma sub T}For every $\epsilon>0$ there exists $M=M(\epsilon)>1$
such that the following holds. Let $1\le i\le s$, $\omega\in\Omega_{0}$
and $r>0$ be given. Suppose that $\kappa(\omega)\ge\epsilon$, $\Vert g_{\omega}\pi\omega\Vert_{\infty}\le\epsilon^{-1}$
and $r<M^{-1}$, then
\[
\Gamma_{i}(\omega,r)\cap\xi_{i-1}(\omega)\subset T_{i}(\omega,Mr)\:.
\]
\end{lem}

\begin{proof}
Let $\epsilon>0$, let $\delta>0$ be small with respect to $\epsilon$
and $s$, and let $M>1$ be large with respect to $\delta$. Let $1\le i\le s$,
$\omega\in\Omega_{0}$ and $r>0$, and suppose that $\kappa(\omega)\ge\epsilon$,
$\Vert g_{\omega}\pi\omega\Vert_{\infty}\le\epsilon^{-1}$ and $r<M^{-1}$.
Let $\eta\in\Gamma_{i}(\omega,r)\cap\xi_{i-1}(\omega)$ and fix $x_{\omega}\in\pi\omega$
and $x_{\eta}\in\pi\eta$ with $f_{\omega}^{0}(x_{\omega})=f_{\omega}^{0}(x_{\eta})=1$.

Write $y=\sum_{k=0}^{i-1}g_{\omega}^{k}\pi\omega$. From $\eta\in\xi_{i-1}(\omega)$
and Lemma \ref{lem:dif char of xi_i} it follows that $y=\sum_{k=0}^{i-1}g_{\omega}^{k}\pi\eta$.
Since $g_{\omega}^{k}\pi\omega,g_{\omega}^{k}\pi\eta\in V_{\omega}^{i}$
for all $i<k\le s$,
\[
P_{(V_{\omega}^{i})^{\perp}}x_{\omega}=P_{(V_{\omega}^{i})^{\perp}}(y+g_{\omega}^{i}\pi\omega)\;\text{ and }\;P_{(V_{\omega}^{i})^{\perp}}x_{\eta}=P_{(V_{\omega}^{i})^{\perp}}(y+g_{\omega}^{i}\pi\eta)\:.
\]
Thus, by Lemma \ref{lem:lb on norm of proj of wedge} and by assuming
that $\delta$ is small enough with respect to $\epsilon$,
\begin{eqnarray}
\Vert P_{(V_{\omega}^{i})^{\perp}}x_{\omega}\wedge P_{(V_{\omega}^{i})^{\perp}}x_{\eta}\Vert & = & \Vert\mathrm{A}^{2}P_{(V_{\omega}^{i})^{\perp}}\left((y+g_{\omega}^{i}\pi\omega)\wedge(y+g_{\omega}^{i}\pi\eta)\right)\Vert\nonumber \\
 & \ge & \delta\Vert(y+g_{\omega}^{i}\pi\omega)\wedge(y+g_{\omega}^{i}\pi\eta)\Vert\label{eq:first lb for wedge of proj}\\
 & = & \delta\Vert(y+g_{\omega}^{i}\pi\omega)\wedge(g_{\omega}^{i}\pi\eta-g_{\omega}^{i}\pi\omega)\Vert\:.\nonumber 
\end{eqnarray}
By Lemma \ref{lem:lb on norm} and $|L_{\omega}^{0}y|=1$,
\[
|y|\ge2^{-s/2}\kappa(\omega)^{s}\Vert L_{\omega}(y)\Vert_{\infty}\ge2^{-s/2}\epsilon^{s}\:.
\]
From $\Vert g_{\omega}\pi\omega\Vert_{\infty}\le\epsilon^{-1}$ it
follows that $|y|\le s\epsilon^{-1}$ and $|g_{\omega}^{i}\pi\omega|\le\epsilon^{-1}$.
Thus by (\ref{eq:first lb for wedge of proj}), Lemma \ref{lem:lb on wedge}
and by assuming that $\delta$ is small enough with respect to $\epsilon$
and $s$,
\[
\Vert P_{(V_{\omega}^{i})^{\perp}}x_{\omega}\wedge P_{(V_{\omega}^{i})^{\perp}}x_{\eta}\Vert\ge\delta^{2}|y+g_{\omega}^{i}\pi\omega|\cdot|g_{\omega}^{i}\pi\eta-g_{\omega}^{i}\pi\omega|\:.
\]
From Lemma \ref{lem:lb on norm} we get $|y+g_{\omega}^{i}\pi\omega|\ge2^{-s/2}\epsilon^{s}$.
Hence we may assume that,
\[
\Vert P_{(V_{\omega}^{i})^{\perp}}x_{\omega}\wedge P_{(V_{\omega}^{i})^{\perp}}x_{\eta}\Vert\ge\delta^{3}|g_{\omega}^{i}\pi\eta-g_{\omega}^{i}\pi\omega|\:.
\]
Additionally,
\[
|P_{(V_{\omega}^{i})^{\perp}}x_{\omega}|\le|x_{\omega}|\le(s+1)\Vert g_{\omega}\pi\omega\Vert_{\infty}\le(s+1)\epsilon^{-1}\le\delta^{-1}\:.
\]
Thus from $\eta\in\Gamma_{i}(\omega,r)$,
\begin{eqnarray}
r & \ge & d(P_{(V_{\omega}^{i})^{\perp}}\pi\omega,P_{(V_{\omega}^{i})^{\perp}}\pi\eta)\nonumber \\
 & = & |P_{(V_{\omega}^{i})^{\perp}}x_{\omega}|^{-1}|P_{(V_{\omega}^{i})^{\perp}}x_{\eta}|^{-1}\Vert P_{(V_{\omega}^{i})^{\perp}}x_{\omega}\wedge P_{(V_{\omega}^{i})^{\perp}}x_{\eta}\Vert\label{eq:lb of r}\\
 & \ge & \delta^{4}\cdot|P_{(V_{\omega}^{i})^{\perp}}x_{\eta}|^{-1}\cdot|g_{\omega}^{i}\pi\eta-g_{\omega}^{i}\pi\omega|\:.\nonumber 
\end{eqnarray}

Now assume by contradiction that $|g_{\omega}^{i}\pi\eta|>2\Vert g_{\omega}\pi\omega\Vert_{\infty}$,
then
\[
|g_{\omega}^{i}\pi\eta-g_{\omega}^{i}\pi\omega|\ge|g_{\omega}^{i}\pi\eta|-\Vert g_{\omega}\pi\omega\Vert_{\infty}\ge\frac{1}{2}|g_{\omega}^{i}\pi\eta|\:.
\]
Also, by assuming that $\delta$ is small enough with respect to $\epsilon$
and $s$,
\[
|P_{(V_{\omega}^{i})^{\perp}}x_{\eta}|=|P_{(V_{\omega}^{i})^{\perp}}(y+g_{\omega}^{i}\pi\eta)|\le|y|+|g_{\omega}^{i}\pi\eta|\le\delta^{-1}|g_{\omega}^{i}\pi\eta|\:.
\]
Hence by (\ref{eq:lb of r}),
\[
r\ge\delta^{5}\cdot|g_{\omega}^{i}\pi\eta|^{-1}\cdot\frac{1}{2}|g_{\omega}^{i}\pi\eta|=\delta^{5}/2\:.
\]
But by assuming that $M>2\delta^{-5}$ this contradicts $r<M^{-1}$,
and so we must have
\[
|g_{\omega}^{i}\pi\eta|\le2\Vert g_{\omega}\pi\omega\Vert_{\infty}\le2\epsilon^{-1}\:.
\]
This gives,
\[
|P_{(V_{\omega}^{i})^{\perp}}x_{\eta}|\le|y|+|g_{\omega}^{i}\pi\eta|\le(s+2)\epsilon^{-1}\le\delta^{-1},
\]
and so by (\ref{eq:lb of r}),
\[
r\ge\delta^{5}|g_{\omega}^{i}\pi\eta-g_{\omega}^{i}\pi\omega|\:.
\]
Assuming $M\ge\delta^{-5}$ this gives $\eta\in T_{i}(\omega,Mr)$,
which completes the proof of the lemma.
\end{proof}
The advantage of working with the sets $T_{i}(\omega,r)$ is that
they behave relatively well with respect to the shift $\sigma$. This
is displayed by the following lemma. For $1\le i\le s$, $\omega\in\Omega_{0}$
and $n\ge0$ set,
\[
a_{\omega,n}^{i}:=\frac{1}{|A_{\omega_{0}...\omega_{n-1}}u_{\sigma^{n}\omega}^{0}|}\cdot\underset{x\in E_{\sigma^{n}\omega}^{i},|x|=1}{\min}\:|A_{\omega_{0}...\omega_{n-1}}x|\:.
\]

\begin{lem}
\label{lem:semi-equi of T_i}Let $1\le i\le s$, $\omega\in\Omega_{0}$,
$n\ge0$ and $r>0$ be given. Then,
\[
T_{i}(\omega,a_{\omega,n}^{i}r)\cap\mathcal{P}_{0}^{n-1}(\omega)\subset\sigma^{-n}T_{i}(\sigma^{n}\omega,r)\:.
\]
\end{lem}

\begin{proof}
Write $A_{\omega,n}$ for $A_{\omega_{0}...\omega_{n-1}}$. By part
(\ref{enu:equi of E^i}) of Theorem \ref{thm:from Oseledets},
\begin{equation}
E_{\omega}^{k}=A_{\omega,n}E_{\sigma^{n}\omega}^{k}\text{ for }0\le k\le s\:.\label{eq:AE_sig=00003D}
\end{equation}
Since $\dim E_{\eta}^{0}=1$ for all $\eta\in\Omega_{0}$, there exists
$b_{\omega,n}=\pm1$ such that,
\[
\frac{A_{\omega,n}u_{\sigma^{n}\omega}^{0}}{|A_{\omega,n}u_{\sigma^{n}\omega}^{0}|}=b_{\omega,n}u_{\omega}^{0}\:.
\]

Let $\eta\in T_{i}(\omega,a_{\omega,n}^{i}r)\cap\mathcal{P}_{0}^{n-1}(\omega)$
and fix $x_{\sigma^{n}\eta}\in\pi\sigma^{n}\eta$ with $f_{\sigma^{n}\omega}^{0}(x_{\sigma^{n}\eta})=1$.
We have,
\begin{eqnarray*}
A_{\omega,n}x_{\sigma^{n}\eta} & = & A_{\omega,n}u_{\sigma^{n}\omega}^{0}+\sum_{k=1}^{s}A_{\omega,n}L_{\sigma^{n}\omega}^{k}x_{\sigma^{n}\eta}\\
 & = & \left|A_{\omega,n}u_{\sigma^{n}\omega}^{0}\right|b_{\omega,n}u_{\omega}^{0}+\sum_{k=1}^{s}A_{\omega,n}L_{\sigma^{n}\omega}^{k}x_{\sigma^{n}\eta}\:.
\end{eqnarray*}
From this and (\ref{eq:AE_sig=00003D}),
\[
L_{\omega}^{i}A_{\omega,n}x_{\sigma^{n}\eta}=A_{\omega,n}L_{\sigma^{n}\omega}^{i}x_{\sigma^{n}\eta}\;\text{ and }\;f_{\omega}^{0}(A_{\omega,n}x_{\sigma^{n}\eta})=\left|A_{\omega,n}u_{\sigma^{n}\omega}^{0}\right|b_{\omega,n}\:.
\]
Now note that $\pi\eta=A_{\omega,n}\pi\sigma^{n}\eta$, and so $0\ne A_{\omega,n}x_{\sigma^{n}\eta}\in\pi\eta$.
This implies,
\begin{eqnarray*}
g_{\omega}^{i}\pi\eta & = & (L_{\omega}^{i}A_{\omega,n}x_{\sigma^{n}\eta})/f_{\omega}^{0}(A_{\omega,n}x_{\sigma^{n}\eta})\\
 & = & \left|A_{\omega,n}u_{\sigma^{n}\omega}^{0}\right|^{-1}b_{\omega,n}^{-1}\cdot A_{\omega,n}L_{\sigma^{n}\omega}^{i}x_{\sigma^{n}\eta}\\
 & = & \left|A_{\omega,n}u_{\sigma^{n}\omega}^{0}\right|^{-1}b_{\omega,n}^{-1}\cdot A_{\omega,n}g_{\sigma^{n}\omega}^{i}(\pi\sigma^{n}\eta)\:.
\end{eqnarray*}
A similar argument gives,
\[
g_{\omega}^{i}\pi\omega=\left|A_{\omega,n}u_{\sigma^{n}\omega}^{0}\right|^{-1}b_{\omega,n}^{-1}\cdot A_{\omega,n}g_{\sigma^{n}\omega}^{i}(\pi\sigma^{n}\omega)\:.
\]
From these formulas, the definition of $a_{\omega,n}^{i}$ and $\eta\in T_{i}(\omega,a_{\omega,n}^{i}r)$,
we get
\begin{eqnarray*}
a_{\omega,n}^{i}r & \ge & |g_{\omega}^{i}\pi\omega-g_{\omega}^{i}\pi\eta|\\
 & = & \left|A_{\omega,n}u_{\sigma^{n}\omega}^{0}\right|^{-1}\left|A_{\omega,n}(g_{\sigma^{n}\omega}^{i}(\pi\sigma^{n}\omega)-g_{\sigma^{n}\omega}^{i}(\pi\sigma^{n}\eta))\right|\\
 & \ge & a_{\omega,n}^{i}|g_{\sigma^{n}\omega}^{i}(\pi\sigma^{n}\omega)-g_{\sigma^{n}\omega}^{i}(\pi\sigma^{n}\eta)|\:.
\end{eqnarray*}
This, together with Lemma \ref{lem:eq of partitions}, implies that
$\sigma^{n}\eta\in T_{i}(\sigma^{n}\omega,r)$, which completes the
proof of the lemma.
\end{proof}
The following theorem, which will be used in the next subsection,
is due to Maker \cite{Mak}.
\begin{thm}
\label{thm:Maker}Let $(X,\mathcal{B},\rho,T)$ be an ergodic measure
preserving system and let $h,h_{1},h_{2},...\in L^{1}(\rho)$. Suppose
that $h_{n}(x)\overset{n}{\rightarrow}h(x)$ for $\rho$-a.e. $x$
and that $\underset{n\ge1}{\sup}\:|h_{n}|$ is integrable. Then,
\[
\frac{1}{n}\sum_{j=0}^{n-1}h_{n-j}\circ T^{j}(x)\overset{n}{\rightarrow}\int h\:d\rho\quad\text{ for }\rho\text{-a.e. }x\:.
\]
\end{thm}

\subsection{\label{subsec:Proof-of-Theorem transverse}Proof of Proposition \ref{prop:transvers dim =00003D}}

We are now ready to begin the proof of Proposition \ref{prop:transvers dim =00003D}.
Let $0<\epsilon<1$ be small and let $M>1$ be large with respect
to $\epsilon$ and $s$. Set,
\[
F_{0}(\epsilon)=\{\omega\in\Omega_{0}\::\:\kappa(\omega)\ge\epsilon\text{ and }\Vert g_{\omega}\pi\omega\Vert_{\infty}\le\epsilon^{-1}\},
\]
then $\beta(F_{0}(\epsilon))>0$ by assuming that $\epsilon$ is sufficiently
small. By part (\ref{enu:asym of back mul}) of Theorem \ref{thm:from Oseledets}
there exist an integer $N=N(\epsilon,M)\ge1$ and a Borel set $F=F(\epsilon,M)\subset F_{0}(\epsilon)$,
such that $\beta(F)\ge(1-\epsilon)\beta(F_{0}(\epsilon))>0$ and for
every $1\le i\le s$,
\begin{equation}
M^{-1}\ge a_{\omega,n}^{i}\ge M\exp(n(\tilde{\lambda}_{i}-\epsilon))\text{ for }\omega\in F\text{ and }n\ge N\:.\label{eq:>=00003D a >=00003D}
\end{equation}
It is clear that $\beta(F)\rightarrow1$ as $\epsilon\rightarrow0$.
Thus in order to prove Proposition \ref{prop:transvers dim =00003D}
it suffices to show that for each $1\le i\le s$,
\begin{equation}
\vartheta_{i-1}(\omega)\ge\frac{\mathrm{H}_{i}-\mathrm{H}_{i-1}}{\tilde{\lambda}_{i}-\epsilon}\text{ for }\beta\text{-a.e. }\omega\in F\:.\label{eq:suffices to show}
\end{equation}

By the Poincaré recurrence theorem, and by removing a Borel set of
zero $\beta$-measure from $F$, we may assume that
\begin{equation}
\#\{n\ge1\::\:\sigma^{nN}\omega\in F\}=\infty\text{ for all }\omega\in F\:.\label{eq:inf many times}
\end{equation}
Let $\sigma_{F}:F\rightarrow F$ be the transformation induced by
$\sigma^{N}$ on the set $F$. That is $\sigma_{F}(\omega)=\sigma^{Nr_{F}(\omega)}(\omega)$
for $\omega\in F$, where
\[
r_{F}(\omega)=\inf\{n\ge1\::\:\sigma^{nN}\omega\in F\}\:.
\]
Let $\beta_{F}$ be the Borel probability measure on $F$ which satisfies,
\[
\beta_{F}(D)=\beta(F\cap D)/\beta(F)\text{ for any Borel set }D\subset F\:.
\]
Since $(\Omega,\sigma^{N},\beta)$ is an ergodic measure preserving
system the same is true for the system $(F,\sigma_{F},\beta_{F})$
(e.g. see \cite[Lemma 2.43]{EW}).

For $\omega\in F$ and $1\le i\le s$ set,
\[
\ell(\omega)=Nr_{F}(\omega)\;\text{ and }\;\rho(i,\omega)=\exp(\ell(\omega)(\tilde{\lambda}_{i}-\epsilon))\:.
\]
In the proof of the following lemma we are going to use the auxiliary
results obtained in Section \ref{subsec:Preparations-for-transv}.
\begin{lem}
\label{lem:equiv of Gamma}Let $\omega\in F$, $1\le i\le s$ and
$0<r\le1$ be given. Then,
\[
\xi_{i-1}(\omega)\cap\Gamma_{i}(\omega,\rho(i,\omega)r)\cap\mathcal{P}_{0}^{\ell(\omega)-1}(\omega)\subset\sigma^{-\ell(\omega)}\Gamma_{i}(\sigma_{F}\omega,r)\:.
\]
\end{lem}

\begin{proof}
Write,
\[
L=\xi_{i-1}(\omega)\cap\Gamma_{i}(\omega,\rho(i,\omega)r)\cap\mathcal{P}_{0}^{\ell(\omega)-1}(\omega)\:.
\]
By (\ref{eq:>=00003D a >=00003D}) and since $\ell(\omega)\ge N$,
\[
M^{-1}\ge a_{\omega,\ell(\omega)}^{i}\ge M\rho(i,\omega)\;.
\]
Since $F\subset F_{0}(\epsilon)$, we have $\kappa(\omega)\ge\epsilon$
and $\Vert g_{\omega}\pi\omega\Vert_{\infty}\le\epsilon^{-1}$. From
this, $\rho(i,\omega)r<M^{-1}$, Lemma \ref{lem:Gamma sub T} and
by assuming that $M$ is sufficiently large with respect to $\epsilon$,
\[
L\subset T_{i}(\omega,M^{1/2}\rho(i,\omega)r)\cap\mathcal{P}_{0}^{\ell(\omega)-1}(\omega)\:.
\]
Thus from $M\rho(i,\omega)\le a_{\omega,\ell(\omega)}^{i}$ and Lemma
\ref{lem:semi-equi of T_i},
\begin{equation}
L\subset T_{i}(\omega,M^{-1/2}a_{\omega,\ell(\omega)}^{i}r)\cap\mathcal{P}_{0}^{\ell(\omega)-1}(\omega)\subset\sigma^{-\ell(\omega)}T_{i}(\sigma^{\ell(\omega)}\omega,M^{-1/2}r)\:.\label{eq:L sub sid^-ell}
\end{equation}

Write,
\[
R=s^{3}2^{3+s}\Vert g_{\sigma^{\ell(\omega)}\omega}\pi\sigma^{\ell(\omega)}\omega\Vert_{\infty}\kappa(\sigma^{\ell(\omega)}\omega)^{-2s-2}\:.
\]
Then by Lemma \ref{lem:T sub Gamma},
\begin{equation}
T_{i}(\sigma^{\ell(\omega)}\omega,M^{-1/2}r)\subset\Gamma_{i}(\sigma^{\ell(\omega)}\omega,RM^{-1/2}r)\:.\label{eq:T_i(sig^ell) sub}
\end{equation}
Since $\sigma^{\ell(\omega)}\omega\in F$ we have,
\[
\kappa(\sigma^{\ell(\omega)}\omega)\ge\epsilon\quad\text{ and }\quad\Vert g_{\sigma^{\ell(\omega)}\omega}\pi\sigma^{\ell(\omega)}\omega\Vert_{\infty}\le\epsilon^{-1}\:.
\]
Hence, by taking $M$ to be sufficiently large with respect to $\epsilon$
and $s$ we may assume that $RM^{-1/2}\le1$. From this, (\ref{eq:L sub sid^-ell})
and (\ref{eq:T_i(sig^ell) sub}) we now get,
\[
L\subset\sigma^{-\ell(\omega)}\Gamma_{i}(\sigma^{\ell(\omega)}\omega,r)\:.
\]
Since $\sigma^{\ell(\omega)}\omega=\sigma_{F}\omega$ this completes
the proof of the lemma.
\end{proof}
The rest of the proof of Proposition \ref{prop:transvers dim =00003D}
is similar to the proof of \cite[Proposition 5.1]{Fe}. For completeness
and clarity we essentially provide full details. We shall need to
establish some more lemmas before we can continue with the proof of
(\ref{eq:suffices to show}).

For $n\ge1$ write,
\[
F_{n}=\{\omega\in F\::\:r_{F}(\omega)=n\}\:.
\]
The following lemma will enable us to apply Maker's ergodic theorem,
which was stated above.
\begin{lem}
\label{lem:prep for maker}Let $1\le i\le s$ be given. Then for $\beta$-a.e.
$\omega\in F$,
\[
\underset{r\downarrow0}{\lim}\:\log\frac{\beta_{\omega}^{\xi_{i-1}}(\Gamma_{i}(\omega,r)\cap\mathcal{P}_{0}^{\ell(\omega)-1}(\omega))}{\beta_{\omega}^{\xi_{i-1}}(\Gamma_{i}(\omega,r))}=-\sum_{0\le j<Nr_{F}(\omega)}\:\mathrm{I}_{\beta}(\mathcal{P}\mid\widehat{\xi_{i}})(\sigma^{j}\omega)\:.
\]
Furthermore, set
\[
q(\omega)=-\underset{r>0}{\inf}\:\log\frac{\beta_{\omega}^{\xi_{i-1}}(\Gamma_{i}(\omega,r)\cap\mathcal{P}_{0}^{\ell(\omega)-1}(\omega))}{\beta_{\omega}^{\xi_{i-1}}(\Gamma_{i}(\omega,r))},
\]
then $q\ge0$ and $q\in L^{1}(\beta_{F})$.
\end{lem}

\begin{proof}
For $\omega\in F$ and $r>0$ with $\beta_{\omega}^{\xi_{i-1}}(\Gamma_{i}(\omega,r))>0$
write,
\[
\alpha(\omega,r)=\log\frac{\beta_{\omega}^{\xi_{i-1}}(\Gamma_{i}(\omega,r)\cap\mathcal{P}_{0}^{\ell(\omega)-1}(\omega))}{\beta_{\omega}^{\xi_{i-1}}(\Gamma_{i}(\omega,r))}\:.
\]
As in the proof of Lemma \ref{lem:pos density}, define a Borel map
$\phi_{i}:\Omega_{0}\rightarrow\mathrm{P}(V)$ by $\phi_{i}(\omega)=P_{(V_{\omega}^{i})^{\perp}}\pi\omega$.
Note that,
\[
\Gamma_{i}(\omega,r)=\xi_{0}(\omega)\cap B^{\phi_{i}}(\omega,r)\text{ for }\omega\in\Omega_{0}\text{ and }r>0\:.
\]
From this and $F=\cup_{k\ge1}F_{k}$, we get that for $\beta$-a.e.
$\omega\in F$,
\begin{eqnarray*}
\alpha(\omega,r) & = & \log\frac{\beta_{\omega}^{\xi_{i-1}}(B^{\phi_{i}}(\omega,r)\cap\mathcal{P}_{0}^{\ell(\omega)-1}(\omega))}{\beta_{\omega}^{\xi_{i-1}}(B^{\phi_{i}}(\omega,r))}\\
 & = & \sum_{k\ge1}\sum_{A\in\mathcal{P}_{0}^{kN-1}}\:1_{F_{k}\cap A}(\omega)\log\frac{\beta_{\omega}^{\xi_{i-1}}(B^{\phi_{i}}(\omega,r)\cap A)}{\beta_{\omega}^{\xi_{i-1}}(B^{\phi_{i}}(\omega,r))}\:.
\end{eqnarray*}
Denote the Borel $\sigma$-algebra of $\mathrm{P}(V)$ by $\mathcal{B}$.
It is easy to verify that,
\[
\widehat{\xi_{i-1}}\vee\phi_{i}^{-1}(\mathcal{B})\underset{\beta}{=}\widehat{\xi_{i}}\:.
\]
From this and Lemma \ref{lem:con exp as lim} it follows that for
$\beta$-a.e. $\omega\in F$,
\begin{eqnarray*}
\underset{r\downarrow0}{\lim}\:\alpha(\omega,r) & = & \sum_{k\ge1}\sum_{A\in\mathcal{P}_{0}^{kN-1}}\:1_{F_{k}\cap A}(\omega)\log\mathrm{E}_{\beta}(1_{A}\mid\widehat{\xi_{i-1}}\vee\phi_{i}^{-1}(\mathcal{B}))(\omega)\\
 & = & \sum_{k\ge1}1_{F_{k}}(\omega)\sum_{A\in\mathcal{P}_{0}^{kN-1}}\:1_{A}(\omega)\log\mathrm{E}_{\beta}(1_{A}\mid\widehat{\xi_{i}})(\omega)\\
 & = & -\sum_{k\ge1}1_{F_{k}}(\omega)\mathrm{I}_{\beta}(\mathcal{P}_{0}^{kN-1}\mid\widehat{\xi_{i}})(\omega)\:.
\end{eqnarray*}
This together with (\ref{eq:=00003Derg sum of info}) shows that for
$\beta$-a.e. $\omega\in F$,
\begin{eqnarray*}
\underset{r\downarrow0}{\lim}\:\alpha(\omega,r) & = & -\sum_{k\ge1}1_{F_{k}}(\omega)\sum_{j=0}^{kN-1}\mathrm{I}_{\beta}(\mathcal{P}\mid\widehat{\xi_{i}})(\sigma^{j}\omega)\\
 & = & -\sum_{0\le j<Nr_{F}(\omega)}\:\mathrm{I}_{\beta}(\mathcal{P}\mid\widehat{\xi_{i}})(\sigma^{j}\omega),
\end{eqnarray*}
which is the first part of the lemma.

The proof of $q\in L^{1}(\beta_{F})$ is exactly the same as the proof
of the analogous fact in \cite[Proposition 5.5]{Fe}, and is therefore
omitted.
\end{proof}
\begin{rem}
\label{rem:fin sup assump}It is worth pointing out that if in our
main result, $\mu$ is only assumed to be discrete and with finite
Shanon entropy (instead of being finitely supported), then the argument
in \cite[Proposition 5.5]{Fe} which gives the integrability of $q$
does not seem to work.
\end{rem}

We continue towards proving (\ref{eq:suffices to show}). Since $(\Omega,\sigma^{N},\beta)$
is ergodic, a classical result due to Kac \cite{Ka} gives,
\begin{equation}
\int_{F}r_{F}\:d\beta_{F}=\beta(F)^{-1}\:.\label{eq:Kac}
\end{equation}
The following Lemma follows directly from \cite[Lemma 2.11]{Fe},
the ergodicity of $(F,\sigma_{F},\beta_{F})$ and (\ref{eq:Kac}).
\begin{lem}
\label{lem:integ in induced}Let $h\in L^{1}(\beta)$ and set,
\[
\tilde{h}(\omega)=\sum_{0\le j<Nr_{F}(\omega)}h(\sigma^{j}\omega)\quad\text{ for }\omega\in F\:.
\]
Then $\tilde{h}\in L^{1}(\beta_{F})$ and,
\[
\int\tilde{h}\:d\beta_{F}=\frac{N}{\beta(F)}\int h\:d\beta\:.
\]
\end{lem}

Recall that for $\omega\in F$ and $1\le i\le s$ we write,
\[
\ell(\omega)=Nr_{F}(\omega)\text{ and }\rho(i,\omega)=\exp(\ell(\omega)(\tilde{\lambda}_{i}-\epsilon))\:.
\]
 Set $\rho_{0}(i,\omega)=1$ and for $n\ge1$ set,
\[
\rho_{n}(i,\omega)=\prod_{k=0}^{n-1}\rho(i,\sigma_{F}^{k}\omega),
\]
where $\sigma_{F}^{k}:=(\sigma_{F})^{k}$.
\begin{lem}
\label{lem:quotient of rho}Let $1\le i\le s$ be given. Then for
$\beta$-a.e. $\omega\in F$,
\[
\underset{n\rightarrow\infty}{\lim}\:\frac{\log\rho_{n}(i,\omega)}{\log\rho_{n-1}(i,\omega)}=1\:.
\]
\end{lem}

\begin{proof}
Let $M\ge1$ be an integer. Since $\sigma_{F}\beta_{F}=\beta_{F}$
and by (\ref{eq:Kac}),
\begin{multline*}
\sum_{n\ge1}\beta_{F}\{\omega\::\:r_{F}(\sigma_{F}^{n}\omega)\ge nM^{-1}\}=\sum_{n\ge1}\beta_{F}\{r_{F}\ge nM^{-1}\}\\
=\int Mr_{F}\:d\beta_{F}=M/\beta(F)<\infty\:.
\end{multline*}
From this and the Borel-Cantelli lemma it follows that for $\beta$-a.e.
$\omega\in F$ there exists $N_{\omega}\ge1$ such that,
\[
\frac{1}{n}r_{F}(\sigma_{F}^{n}\omega)<M^{-1}\text{ for all }n\ge N_{\omega},
\]
which shows,
\[
\underset{n\rightarrow\infty}{\lim}\:\frac{1}{n}r_{F}(\sigma_{F}^{n}\omega)=0\text{ for }\beta\text{-a.e. }\omega\in F\:.
\]
The lemma now follows easily from this and since $-\log\rho_{n}(i,\omega)\ge n(\epsilon-\tilde{\lambda}_{i})$
for each $n\ge1$.
\end{proof}
We resume with the proof of Proposition \ref{prop:transvers dim =00003D}.
Fix $1\le i\le s$, and recall that our aim is to show (\ref{eq:suffices to show}).

For $n\ge1$ and $\beta$-a.e. $\omega\in F$ we can set,
\[
K_{n}(\omega)=\log\frac{\beta_{\omega}^{\xi_{i-1}}(\Gamma_{i}(\omega,\rho_{n}(i,\omega)))}{\beta_{\sigma_{F}\omega}^{\xi_{i-1}}(\Gamma_{i}(\sigma_{F}\omega,\rho_{n-1}(i,\sigma_{F}\omega)))},
\]
\[
G_{n}(\omega)=\log\frac{\beta_{\omega}^{\xi_{i-1}}(\Gamma_{i}(\omega,\rho_{n}(i,\omega))\cap\mathcal{P}_{0}^{\ell(\omega)-1}(\omega))}{\beta_{\omega}^{\xi_{i-1}}(\Gamma_{i}(\omega,\rho_{n}(i,\omega)))},
\]
and,
\[
R_{l}(\omega)=\sum_{0\le j<Nr_{F}(\omega)}\mathrm{I}_{\beta}(\mathcal{P}\mid\widehat{\xi_{l}})(\sigma^{j}\omega)\text{ for }l=i,i-1\:.
\]
Then for $\beta$-a.e. $\omega\in F$,
\[
K_{n}(\omega)+G_{n}(\omega)=\log\frac{\beta_{\omega}^{\xi_{i-1}}(\xi_{i-1}(\omega)\cap\Gamma_{i}(\omega,\rho_{n}(i,\omega))\cap\mathcal{P}_{0}^{\ell(\omega)-1}(\omega))}{\beta_{\sigma_{F}\omega}^{\xi_{i-1}}(\Gamma_{i}(\sigma_{F}\omega,\rho_{n-1}(i,\sigma_{F}\omega)))}\:.
\]
Hence by Lemma \ref{lem:equiv of Gamma},
\[
K_{n}(\omega)+G_{n}(\omega)\le\log\frac{\beta_{\omega}^{\xi_{i-1}}(\sigma^{-\ell(\omega)}(\Gamma_{i}(\sigma_{F}\omega,\rho_{n-1}(i,\sigma_{F}\omega)))\cap\mathcal{P}_{0}^{\ell(\omega)-1}(\omega))}{\beta_{\sigma^{\ell(\omega)}\omega}^{\xi_{i-1}}(\Gamma_{i}(\sigma_{F}\omega,\rho_{n-1}(i,\sigma_{F}\omega)))}\:.
\]
From this and Lemma \ref{lem:form for con meas} we get that for $\beta$-a.e.
$\omega\in F$,
\begin{eqnarray*}
K_{n}(\omega)+G_{n}(\omega) & \le & \log\beta_{\omega}^{\xi_{i-1}}(\mathcal{P}_{0}^{\ell(\omega)-1}(\omega))\\
 & = & \sum_{k=1}^{\infty}1_{F_{k}}(\omega)\log\beta_{\omega}^{\xi_{i-1}}(\mathcal{P}_{0}^{kN-1}(\omega))\:.
\end{eqnarray*}
Thus by (\ref{eq:=00003Derg sum of info}),
\[
K_{n}(\omega)+G_{n}(\omega)\le-\sum_{k=1}^{\infty}1_{F_{k}}(\omega)\sum_{j=0}^{kN-1}\mathrm{I}_{\beta}(\mathcal{P}\mid\widehat{\xi_{i-1}})(\sigma^{j}\omega)=-R_{i-1}(\omega)\:.
\]
It follows that for $\beta$-a.e. $\omega\in F$ and any $n\ge1$,
\begin{eqnarray}
-\log\beta_{\omega}^{\xi_{i-1}}(\Gamma_{i}(\omega,\rho_{n}(i,\omega))) & = & \left(-\sum_{j=0}^{n-1}K_{n-j}(\sigma_{F}^{j}\omega)\right)-\log\beta_{\sigma_{F}^{n}\omega}^{\xi_{i-1}}(\Gamma_{i}(\sigma_{F}^{n}\omega,1))\nonumber \\
 & \ge & -\sum_{j=0}^{n-1}K_{n-j}(\sigma_{F}^{j}\omega)\label{eq:mass of ball as erg sums}\\
 & \ge & \sum_{j=0}^{n-1}(G_{n-j}(\sigma_{F}^{j}\omega)+R_{i-1}(\sigma_{F}^{j}\omega))\:.\nonumber 
\end{eqnarray}

By Lemma \ref{lem:integ in induced} it follows that for $l=i,i-1$
we have $R_{l}\in L^{1}(\beta_{F})$ with,
\[
\int R_{l}\:d\beta_{F}=\frac{N}{\beta(F)}\int\mathrm{I}_{\beta}(\mathcal{P}\mid\widehat{\xi_{l}})\:d\beta=\frac{N}{\beta(F)}\mathrm{H}_{\beta}(\mathcal{P}\mid\widehat{\xi_{l}})=\frac{N}{\beta(F)}\mathrm{H}_{l}\:.
\]
From this and Birkhoff's ergodic theorem it follows that for $\beta$-a.e.
$\omega\in F$,
\begin{equation}
\underset{n\rightarrow\infty}{\lim}\:\frac{1}{n}\sum_{j=0}^{n-1}R_{i-1}(\sigma_{F}^{j}\omega)=\frac{N}{\beta(F)}\mathrm{H}_{i-1}\:.\label{eq:erg avg of R_i-1}
\end{equation}
By Lemma \ref{lem:prep for maker} it follows that for $\beta$-a.e.
$\omega\in F$,
\[
\underset{n\rightarrow\infty}{\lim}\:G_{n}(\omega)=-R_{i}(\omega)\text{ and }\underset{n\ge1}{\sup}\:|G_{n}|\in L^{1}(\beta_{F})\:.
\]
Thus from Theorem \ref{thm:Maker} we get that for $\beta$-a.e. $\omega\in F$,
\begin{equation}
\underset{n\rightarrow\infty}{\lim}\:\frac{1}{n}\sum_{j=0}^{n-1}G_{n-j}(\sigma_{F}^{j}\omega)=\int-R_{i}\:d\beta_{F}=-\frac{N}{\beta(F)}\mathrm{H}_{i}\:.\label{eq:erg avg of G}
\end{equation}
For $\omega\in F$ and $n\ge1$,
\[
\log\rho_{n}(i,\omega)=\sum_{j=0}^{n-1}\log\rho(i,\sigma_{F}^{j}\omega)=(\tilde{\lambda}_{i}-\epsilon)N\sum_{j=0}^{n-1}r_{F}(\sigma_{F}^{j}\omega)\:.
\]
From this, Birkhoff's theorem and (\ref{eq:Kac}), it follows that
for $\beta$-a.e. $\omega\in F$,
\[
\underset{n\rightarrow\infty}{\lim}\:-\frac{1}{n}\log\rho_{n}(i,\omega)=(\epsilon-\tilde{\lambda}_{i})N\int r_{F}\:d\beta_{F}=\frac{(\epsilon-\tilde{\lambda}_{i})N}{\beta(F)}\:.
\]
Now from (\ref{eq:mass of ball as erg sums}), (\ref{eq:erg avg of R_i-1}),
(\ref{eq:erg avg of G}) and the last equality, we get that for $\beta$-a.e.
$\omega\in F$,
\begin{multline*}
\underset{n\rightarrow\infty}{\liminf}\:\frac{\log\beta_{\omega}^{\xi_{i-1}}(\Gamma_{i}(\omega,\rho_{n}(i,\omega)))}{\log\rho_{n}(i,\omega)}\\
\ge\underset{n\rightarrow\infty}{\lim}\:\frac{\frac{1}{n}\left(\sum_{j=0}^{n-1}(G_{n-j}(\sigma_{F}^{j}\omega)+R_{i-1}(\sigma_{F}^{j}\omega))\right)}{-\frac{1}{n}\log\rho_{n}(i,\omega)}=\frac{\mathrm{H}_{i}-\mathrm{H}_{i-1}}{\tilde{\lambda}_{i}-\epsilon}\:.
\end{multline*}
This together with Lemma \ref{lem:quotient of rho} shows that for
$\beta$-a.e. $\omega\in F$,
\[
\underset{r\downarrow0}{\liminf}\:\frac{\log\beta_{\omega}^{\xi_{i-1}}(\Gamma_{i}(\omega,r))}{\log r}\ge\frac{\mathrm{H}_{i}-\mathrm{H}_{i-1}}{\tilde{\lambda}_{i}-\epsilon},
\]
which gives (\ref{eq:suffices to show}) and completes the proof of
Proposition \ref{prop:transvers dim =00003D}.

\section{\label{sec:Proof-of-results}Proof of the main result}

In this Section we complete the proof of our main result Theorem \ref{thm:LY formula and ED}.
For $1\le k\le s$, $0\le i\le k$ and $\omega\in\Omega_{0}$ set,
\[
\overline{\gamma}_{i,k}(\omega)=\underset{r\downarrow0}{\limsup}\:\frac{\log\beta_{\omega}^{\xi_{i}}(\Gamma_{k}(\omega,r))}{\log r}\text{ and }\underline{\gamma}_{i,k}(\omega)=\underset{r\downarrow0}{\liminf}\:\frac{\log\beta_{\omega}^{\xi_{i}}(\Gamma_{k}(\omega,r))}{\log r}\:.
\]
Recall that for $0\le i<s$ we write,
\[
\vartheta_{i}(\omega)=\underset{r\downarrow0}{\liminf}\:\frac{\log\beta_{\omega}^{\xi_{i}}(\Gamma_{i+1}(\omega,r))}{\log r}\:.
\]

The proof of the following proposition is a modification of the argument
used in \cite[Section 6, Proof of (C3)]{Fe}. That argument in turn
follows the lines of the proof of \cite[Theorem 2.11]{FH}, which
was adapted from the original proof of \cite[Lemma 11.3.1]{LY}.
\begin{prop}
\label{prop:ineq for delta underline}For $1\le k\le s$, $0\le i<k$
and $\beta$-a.e. $\omega$,
\[
\underline{\gamma}_{i+1,k}(\omega)+\vartheta_{i}(\omega)\le\underline{\gamma}_{i,k}(\omega)\:.
\]
\end{prop}

\begin{proof}
Assume by contradiction that the proposition is false. Recall that
for $\omega\in\Omega_{0}$ we have $\pi\omega\notin\mathrm{P}(V_{\omega}^{0})$.
Since $(V_{\omega}^{0})^{\perp}\subset(V_{\omega}^{k})^{\perp}$ this
implies that also $P_{(V_{\omega}^{k})^{\perp}}\pi\omega\notin\mathrm{P}(V_{\omega}^{0})$.
Thus, since we assume that the proposition is false, there exist $1\le k\le s$,
$0\le i<k$, $0<\epsilon<1$ and $F\subset\Omega_{0}$, such that
$\beta(F)>0$ and for $\omega\in F$,
\[
d(P_{(V_{\omega}^{k})^{\perp}}\pi\omega,\mathrm{P}(V_{\omega}^{0}))\ge\epsilon\text{ and }\underline{\gamma}_{i+1,k}(\omega)+\vartheta_{i}(\omega)>\underline{\gamma}_{i,k}(\omega)\:.
\]
There exist $\alpha>0$ and real numbers $\underline{\gamma}_{i,k},\underline{\gamma}_{i+1,k},\vartheta_{i}$
such that,
\[
\underline{\gamma}_{i+1,k}+\vartheta_{i}>\underline{\gamma}_{i,k}+\alpha,
\]
and for any $\rho>0$ there exists $F_{\rho}\subset F$ with $\beta(F_{\rho})>0$,
so that for $\omega\in F_{\rho}$,
\begin{equation}
|\underline{\gamma}_{i,k}(\omega)-\underline{\gamma}_{i,k}|<\rho/2,\:|\underline{\gamma}_{i+1,k}(\omega)-\underline{\gamma}_{i+1,k}|<\rho/2\text{ and }|\vartheta_{i}(\omega)-\vartheta_{i}|<\rho/2\:.\label{eq:small diffs}
\end{equation}

Fix $0<\rho<\alpha/2$, then there exist $N_{1}\ge1$ and $F_{\rho}'\subset F_{\rho}$
with $\beta(F_{\rho}')>0$ and,
\begin{equation}
\beta_{\omega}^{\xi_{i+1}}(\Gamma_{k}(\omega,2e^{-n}))\le e^{-n(\underline{\gamma}_{i+1,k}-\rho)}\text{ for }\omega\in F_{\rho}'\text{ and }n\ge N_{1}\:.\label{eq:ub on mas of ba cen at F_rho'}
\end{equation}
By Lemma \ref{lem:pos density} there exist $c>0$, $N_{2}\ge N_{1}$
and $F_{\rho}''\subset F_{\rho}'$ such that $\beta(F_{\rho}'')>0$
and,
\begin{equation}
\frac{\beta_{\omega}^{\xi_{i}}(\Gamma_{k}(\omega,e^{-n})\cap F_{\rho}')}{\beta_{\omega}^{\xi_{i}}(\Gamma_{k}(\omega,e^{-n}))}>c\text{ for }\omega\in F_{\rho}''\text{ and }n\ge N_{2}\:.\label{eq:lb mas ba cen F''}
\end{equation}

Let $\omega\in F_{\rho}''$ and $n\ge N_{2}$, and write
\[
B_{\omega,n}=\{\eta\in\xi_{i}(\omega)\::\:\xi_{i+1}(\eta)\cap F_{\rho}'\cap\Gamma_{k}(\omega,e^{-n})\ne\emptyset\}\:.
\]
Let us show that $B_{\omega,n}\subset\Gamma_{i+1}(\omega,\epsilon^{-2}e^{-n})$.
Given $\eta\in B_{\omega,n}$ there exists,
\[
\zeta\in\xi_{i+1}(\eta)\cap F_{\rho}'\cap\Gamma_{k}(\omega,e^{-n})\:.
\]
Since $\xi_{0}(\zeta)=\xi_{0}(\eta)=\xi_{0}(\omega)$,
\[
V_{\omega}^{j}=V_{\eta}^{j}=V_{\zeta}^{j}\text{ for all }0\le j\le s\:.
\]
If $i+1<s$ then from $\omega,\zeta\in F_{\rho}'\subset F$,
\[
d(P_{(V_{\omega}^{k})^{\perp}}\pi\omega,\mathrm{P}(V_{\omega}^{i+1}))\ge d(P_{(V_{\omega}^{k})^{\perp}}\pi\omega,\mathrm{P}(V_{\omega}^{0}))\ge\epsilon,
\]
and,
\[
d(P_{(V_{\omega}^{k})^{\perp}}\pi\zeta,\mathrm{P}(V_{\omega}^{i+1}))\ge d(P_{(V_{\omega}^{k})^{\perp}}\pi\zeta,\mathrm{P}(V_{\omega}^{0}))=d(P_{(V_{\zeta}^{k})^{\perp}}\pi\zeta,\mathrm{P}(V_{\zeta}^{0}))\ge\epsilon\:.
\]
Since $i<k$ we have $(V_{\omega}^{i+1})^{\perp}\subset(V_{\omega}^{k})^{\perp}$.
From this and Lemma \ref{lem:ub on dist of proj},
\begin{eqnarray*}
d(P_{(V_{\omega}^{i+1})^{\perp}}\pi\zeta,P_{(V_{\omega}^{i+1})^{\perp}}\pi\omega) & = & d(P_{(V_{\omega}^{i+1})^{\perp}}P_{(V_{\omega}^{k})^{\perp}}\pi\zeta,P_{(V_{\omega}^{i+1})^{\perp}}P_{(V_{\omega}^{k})^{\perp}}\pi\omega)\\
 & \le & \epsilon^{-2}d(P_{(V_{\omega}^{k})^{\perp}}\pi\zeta,P_{(V_{\omega}^{k})^{\perp}}\pi\omega)\:.
\end{eqnarray*}
Note that since $V_{\omega}^{s}=\{0\}$ this inequality is trivial
when $i+1=s$. Since $\zeta\in\xi_{i+1}(\eta)$ we have $P_{(V_{\omega}^{i+1})^{\perp}}\pi\eta=P_{(V_{\omega}^{i+1})^{\perp}}\pi\zeta$.
Now combining these facts with $\zeta\in\Gamma_{k}(\omega,e^{-n})$
we get,
\begin{eqnarray*}
d(P_{(V_{\omega}^{i+1})^{\perp}}\pi\eta,P_{(V_{\omega}^{i+1})^{\perp}}\pi\omega) & = & d(P_{(V_{\omega}^{i+1})^{\perp}}\pi\zeta,P_{(V_{\omega}^{i+1})^{\perp}}\pi\omega)\\
 & \le & \epsilon^{-2}d(P_{(V_{\omega}^{k})^{\perp}}\pi\zeta,P_{(V_{\omega}^{k})^{\perp}}\pi\omega)\\
 & \le & \epsilon^{-2}e^{-n},
\end{eqnarray*}
which shows $B_{\omega,n}\subset\Gamma_{i+1}(\omega,\epsilon^{-2}e^{-n})$.

Next let us show that,
\begin{equation}
\beta_{\eta}^{\xi_{i+1}}(\Gamma_{k}(\omega,e^{-n})\cap F_{\rho}')\le e^{-n(\underline{\gamma}_{i+1,k}-\rho)}\quad\text{ for }\eta\in B_{\omega,n}\:.\label{eq:ub mas of ba cen B}
\end{equation}
Let $\eta$ and $\zeta$ be as in the last paragraph. Since $d(P_{(V_{\omega}^{k})^{\perp}}\pi\zeta,P_{(V_{\omega}^{k})^{\perp}}\pi\omega)\le e^{-n}$,
\[
\Gamma_{k}(\omega,e^{-n})\cap F_{\rho}'\subset\Gamma_{k}(\zeta,2e^{-n})\:.
\]
From this, $\beta_{\eta}^{\xi_{i+1}}=\beta_{\zeta}^{\xi_{i+1}}$ and
(\ref{eq:ub on mas of ba cen at F_rho'}),
\begin{eqnarray*}
\beta_{\eta}^{\xi_{i+1}}(\Gamma_{k}(\omega,e^{-n})\cap F_{\rho}') & = & \beta_{\zeta}^{\xi_{i+1}}(\Gamma_{k}(\omega,e^{-n})\cap F_{\rho}')\\
 & \le & \beta_{\zeta}^{\xi_{i+1}}(\Gamma_{k}(\zeta,2e^{-n}))\\
 & \le & e^{-n(\underline{\gamma}_{i+1,k}-\rho)},
\end{eqnarray*}
which gives (\ref{eq:ub mas of ba cen B}).

Now from (\ref{eq:lb mas ba cen F''}), $B_{\omega,n}\subset\Gamma_{i+1}(\omega,\epsilon^{-2}e^{-n})$
and (\ref{eq:ub mas of ba cen B}), it follows that for $\beta$-a.e.
$\omega\in F_{\rho}''$ and every $n\ge N_{2}$,
\begin{eqnarray*}
\beta_{\omega}^{\xi_{i}}(\Gamma_{k}(\omega,e^{-n})) & \le & c^{-1}\beta_{\omega}^{\xi_{i}}(\Gamma_{k}(\omega,e^{-n})\cap F_{\rho}')\\
 & = & c^{-1}\int_{B_{\omega,n}}\beta_{\eta}^{\xi_{i+1}}(\Gamma_{k}(\omega,e^{-n})\cap F_{\rho}')\:d\beta_{\omega}^{\xi_{i}}(\eta)\\
 & \le & c^{-1}\beta_{\omega}^{\xi_{i}}(\Gamma_{i+1}(\omega,\epsilon^{-2}e^{-n}))e^{-n(\underline{\gamma}_{i+1,k}-\rho)}\:.
\end{eqnarray*}
Thus, by taking logarithm on both sides, dividing by $-n$ and letting
$n$ tend to infinity,
\[
\underline{\gamma}_{i,k}(\omega)\ge\vartheta_{i}(\omega)+\underline{\gamma}_{i+1,k}-\rho\quad\text{ for }\beta\text{-a.e. }\omega\in F_{\rho}''\:.
\]
By (\ref{eq:small diffs}) we now get,
\[
\underline{\gamma}_{i,k}+2\rho\ge\vartheta_{i}+\underline{\gamma}_{i+1,k}\:.
\]
But this contradicts $\rho<\alpha/2$ and $\underline{\gamma}_{i+1,k}+\vartheta_{i}>\underline{\gamma}_{i,k}+\alpha$,
which completes the proof of the proposition.
\end{proof}
The proof of the following proposition follows the lines of the argument
used in \cite[Section 6, Proof of (C2)]{Fe}. That argument in turn
is modified from \cite[§10.2]{LY} and the proof of \cite[Theorem 2.11]{FH}.
\begin{prop}
\label{prop:ineq for delta overline}For $1\le k\le s$, $0\le i<k$
and $\beta$-a.e. $\omega$,
\[
\frac{\mathrm{H}_{i+1}-\mathrm{H}_{i}}{\tilde{\lambda}_{i+1}}\ge\overline{\gamma}_{i,k}(\omega)-\overline{\gamma}_{i+1,k}(\omega)\:.
\]
\end{prop}

\begin{proof}
Assume by contradiction that the proposition is false. Then there
exist $1\le k\le s$, $0\le i<k$ and $F\subset\Omega_{0}$ with $\beta(F)>0$
and,
\[
\frac{\mathrm{H}_{i+1}-\mathrm{H}_{i}}{\tilde{\lambda}_{i+1}}<\overline{\gamma}_{i,k}(\omega)-\overline{\gamma}_{i+1,k}(\omega)\text{ for }\omega\in F\:.
\]
Thus there exist $\alpha>0$ and real numbers $\overline{\gamma}_{i,k}$
and $\overline{\gamma}_{i+1,k}$ such that,
\begin{equation}
\frac{\mathrm{H}_{i+1}-\mathrm{H}_{i}}{\tilde{\lambda}_{i+1}}<\overline{\gamma}_{i,k}-\overline{\gamma}_{i+1,k}-\alpha,\label{eq:<-alpha}
\end{equation}
and for any $\epsilon>0$ there exists $B_{\epsilon}\subset F$ with
$\beta(B_{\epsilon})>0$, so that for $\omega\in B_{\epsilon}$,
\[
|\overline{\gamma}_{i,k}(\omega)-\overline{\gamma}_{i,k}|<\epsilon/2\quad\text{ and }\quad|\overline{\gamma}_{i+1,k}(\omega)-\overline{\gamma}_{i+1,k}|<\epsilon/2\:.
\]

Fix $0<\epsilon<-\tilde{\lambda}_{i+1}/6$, and for $\omega\in\Omega_{0}$
and $n\ge1$ write,
\[
D_{\omega,n}=\Gamma_{k}(\omega,e^{n(\tilde{\lambda}_{i+1}+5\epsilon)})\:.
\]
Recall the sets $Q_{n,\epsilon}$ defined in (\ref{eq:def of Q_n,eps}).
By removing a subset of zero $\beta$-measure from $B_{\epsilon}$
without changing the notation, it follows that there exists a Borel
function $n_{0}:B_{\epsilon}\rightarrow\mathbb{N}$ such that for
$\omega\in B_{\epsilon}$ and $n\ge n_{0}(\omega)$,
\begin{enumerate}
\item \label{enu:< delta_i+1}$\frac{\log\beta_{\omega}^{\xi_{i+1}}(D_{\omega,n})}{n(\tilde{\lambda}_{i+1}+5\epsilon)}<\overline{\gamma}_{i+1,k}+\epsilon$;
\item \label{>H_i+1 -eps}$-\frac{1}{n}\log\beta_{\omega}^{\xi_{i+1}}(\mathcal{P}_{0}^{n-1}(\omega))>\mathrm{H}_{i+1}-\epsilon$
$\quad$(by Lemma \ref{lem:asym of ergo sum of info});
\item \label{enu:cont in L}$Q_{n,\epsilon}\cap\xi_{i}(\omega)\cap\mathcal{P}_{0}^{n-1}(\omega)\subset D_{\omega,n}$
$\quad$(by Proposition \ref{prop:int cont in ball});
\item \label{enu:<H_i+eps}$-\frac{1}{n}\log\beta_{\omega}^{\xi_{i}}(Q_{n,\epsilon}\cap\mathcal{P}_{0}^{n-1}(\omega))<\mathrm{H}_{i}+\epsilon$
$\quad$(by Lemmas \ref{lem:asym of ergo sum of info} and \ref{lem:density of Q});
\end{enumerate}
Let $N_{0}\ge1$ be such that for,
\[
\Delta=\{\omega\in B_{\epsilon}\::\:n_{0}(\omega)\le N_{0}\},
\]
we have $\beta(\Delta)>0$. By Lemma \ref{lem:pos density} there
exist $0<c<1$ and $\Delta'\subset\Delta$, with $\beta(\Delta')>0$,
so that for $\omega\in\Delta'$ there exits $n=n(\omega)\ge N_{0}$
such that,
\begin{enumerate}
\item [(5)]\setcounter{enumi}{5}$\frac{\beta_{\omega}^{\xi_{i+1}}(D_{\omega,n}\cap\Delta)}{\beta_{\omega}^{\xi_{i+1}}(D_{\omega,n})}>c$;
\item \label{enu:ub mas of ball mu^xi_i}$\frac{\log\beta_{\omega}^{\xi_{i}}(\Gamma_{k}(\omega,2e^{n(\tilde{\lambda}_{i+1}+5\epsilon)}))}{n(\tilde{\lambda}_{i+1}+5\epsilon)}>\overline{\gamma}_{i,k}-\epsilon$;
\item \label{enu:-log c /n <}$\frac{-\log c}{n}<\epsilon$;
\end{enumerate}
Fix $\omega\in\Delta'$ such that all of the conditions (\ref{enu:< delta_i+1})--(\ref{enu:-log c /n <})
are satisfied with $n=n(\omega)$. By (5) and (\ref{enu:< delta_i+1}),
\begin{equation}
\beta_{\omega}^{\xi_{i+1}}(D_{\omega,n}\cap\Delta)>c\beta_{\omega}^{\xi_{i+1}}(D_{\omega,n})>c\exp(n(\tilde{\lambda}_{i+1}+5\epsilon)(\overline{\gamma}_{i+1,k}+\epsilon))\:.\label{eq:lb mas xi_i+1(L cup Del)}
\end{equation}
Write,
\[
\mathcal{E}=\{P\in\mathcal{P}_{0}^{n-1}\::\:P\cap\xi_{i}(\omega)\cap D_{\omega,n}\cap\Delta\ne\emptyset\},
\]
and,
\[
\mathcal{E}'=\{P\in\mathcal{P}_{0}^{n-1}\::\:P\cap\xi_{i+1}(\omega)\cap D_{\omega,n}\cap\Delta\ne\emptyset\}\:.
\]
From $\xi_{i+1}(\omega)\subset\xi_{i}(\omega)$ it follows that $\mathcal{E}'\subset\mathcal{E}$.

Given $P\in\mathcal{E}'$ there exists $\eta\in\xi_{i+1}(\omega)\cap D_{\omega,n}\cap\Delta$
with $P=\mathcal{P}_{0}^{n-1}(\eta)$. Hence from (\ref{>H_i+1 -eps}),
\[
\beta_{\omega}^{\xi_{i+1}}(P)=\beta_{\eta}^{\xi_{i+1}}(\mathcal{P}_{0}^{n-1}(\eta))<\exp(-n(\mathrm{H}_{i+1}-\epsilon))\:.
\]
Thus,
\[
\beta_{\omega}^{\xi_{i+1}}(D_{\omega,n}\cap\Delta)\le\sum_{P\in\mathcal{E}'}\beta_{\omega}^{\xi_{i+1}}(P)<|\mathcal{E}'|\exp(-n(\mathrm{H}_{i+1}-\epsilon))\:.
\]
This together with (\ref{eq:lb mas xi_i+1(L cup Del)}) implies,
\begin{equation}
|\mathcal{E}'|>c\exp(n(\tilde{\lambda}_{i+1}+5\epsilon)(\overline{\gamma}_{i+1,k}+\epsilon))\exp(n(\mathrm{H}_{i+1}-\epsilon))\:.\label{eq:lb card E'}
\end{equation}

Let us show that,
\begin{equation}
Q_{n,\epsilon}\cap\xi_{i}(\omega)\cap P\subset\Gamma_{k}(\omega,2e^{n(\tilde{\lambda}_{i+1}+5\epsilon)})\text{ for }P\in\mathcal{E},\label{eq:int cont in ball}
\end{equation}
and,
\begin{equation}
\beta_{\omega}^{\xi_{i}}(Q_{n,\epsilon}\cap P)>e^{-n(\mathrm{H}_{i}+\epsilon)}\text{ for }P\in\mathcal{E}\:.\label{eq:mas mu^xi_i(P)>}
\end{equation}
Given $P\in\mathcal{E}$ there exists $\eta\in P\cap\xi_{i}(\omega)\cap D_{\omega,n}\cap\Delta$.
Since $\eta\in D_{\omega,n}$ we have,
\[
d(P_{(V_{\omega}^{k})^{\perp}}\pi\eta,P_{(V_{\omega}^{k})^{\perp}}\pi\omega)\le e^{n(\tilde{\lambda}_{i+1}+5\epsilon)}\:.
\]
By (\ref{enu:cont in L}) it follows,
\[
d(P_{(V_{\omega}^{k})^{\perp}}\pi\eta,P_{(V_{\omega}^{k})^{\perp}}\pi\zeta)\le e^{n(\tilde{\lambda}_{i+1}+5\epsilon)}\quad\text{ for }\zeta\in Q_{n,\epsilon}\cap\xi_{i}(\eta)\cap\mathcal{P}_{0}^{n-1}(\eta)\:.
\]
Thus,
\[
Q_{n,\epsilon}\cap\xi_{i}(\omega)\cap P=Q_{n,\epsilon}\cap\xi_{i}(\eta)\cap\mathcal{P}_{0}^{n-1}(\eta)\subset\Gamma_{k}(\omega,2e^{n(\tilde{\lambda}_{i+1}+5\epsilon)}),
\]
which gives (\ref{eq:int cont in ball}). Since $\eta\in\Delta$ it
follows by (\ref{enu:<H_i+eps}),
\[
\beta_{\omega}^{\xi_{i}}(Q_{n,\epsilon}\cap P)=\beta_{\eta}^{\xi_{i}}(Q_{n,\epsilon}\cap\mathcal{P}_{0}^{n-1}(\eta))>e^{-n(\mathrm{H}_{i}+\epsilon)},
\]
which gives (\ref{eq:mas mu^xi_i(P)>}).

From (\ref{eq:int cont in ball}), (\ref{eq:mas mu^xi_i(P)>}), $\mathcal{E}'\subset\mathcal{E}$
and (\ref{eq:lb card E'}) we now get,
\begin{eqnarray*}
\beta_{\omega}^{\xi_{i}}(\Gamma_{k}(\omega,2e^{n(\tilde{\lambda}_{i+1}+5\epsilon)})) & \ge & \sum_{P\in\mathcal{E}}\beta_{\omega}^{\xi_{i}}(Q_{n,\epsilon}\cap P)\\
 & \ge & |\mathcal{E}|e^{-n(\mathrm{H}_{i}+\epsilon)}\\
 & \ge & c\exp(n(\tilde{\lambda}_{i+1}+5\epsilon)(\overline{\gamma}_{i+1,k}+\epsilon)+n(\mathrm{H}_{i+1}-\epsilon)-n(\mathrm{H}_{i}+\epsilon))\:.
\end{eqnarray*}
This together with (\ref{enu:ub mas of ball mu^xi_i}) gives,
\[
c\exp(n(\tilde{\lambda}_{i+1}+5\epsilon)(\overline{\gamma}_{i+1,k}+\epsilon)+n\mathrm{H}_{i+1}-n\mathrm{H}_{i}-2n\epsilon)<\exp(n(\overline{\gamma}_{i,k}-\epsilon)(\tilde{\lambda}_{i+1}+5\epsilon))\:.
\]
Now by taking logarithm on both sides and by dividing by $n$ it follows
from (\ref{enu:-log c /n <}) that,
\[
(\tilde{\lambda}_{i+1}+5\epsilon)(\overline{\gamma}_{i+1,k}+\epsilon)+\mathrm{H}_{i+1}-\mathrm{H}_{i}-3\epsilon<(\overline{\gamma}_{i,k}-\epsilon)(\tilde{\lambda}_{i+1}+5\epsilon)\:.
\]
Since this holds for arbitrarily small $\epsilon>0$ we obtain,
\[
\frac{\mathrm{H}_{i+1}-\mathrm{H}_{i}}{\tilde{\lambda}_{i+1}}\ge\overline{\gamma}_{i,k}-\overline{\gamma}_{i+1,k},
\]
which contradicts (\ref{eq:<-alpha}) and completes the proof of the
proposition.
\end{proof}
Combining Propositions \ref{prop:transvers dim =00003D}, \ref{prop:ineq for delta underline}
and \ref{prop:ineq for delta overline} together, we obtain the following.
\begin{claim}
\label{cla:gam up =00003D gam low}For $1\le k\le s$ and $0\le i\le k$
we have,
\begin{equation}
\overline{\gamma}_{i,k}(\omega)=\underline{\gamma}_{i,k}(\omega)=\sum_{j=i}^{k-1}\frac{\mathrm{H}_{j+1}-\mathrm{H}_{j}}{\tilde{\lambda}_{j+1}}\text{ for }\beta\text{-a.e. }\omega\:.\label{eq:up gak low gam eq}
\end{equation}
\end{claim}

\begin{proof}
Fix $1\le k\le s$. We prove the claim by backward induction on $i$.
By the definition of $\xi_{k}$ it follows that for $\beta$-a.e.
$\omega$,
\[
P_{(V_{\omega}^{k})^{\perp}}\pi\eta=P_{(V_{\omega}^{k})^{\perp}}\pi\omega\qquad\text{ for }\beta_{\omega}^{\xi_{k}}\text{-a.e. }\eta\:.
\]
From this it follows directly that $\overline{\gamma}_{k,k}(\omega)=\underline{\gamma}_{k,k}(\omega)=0$
for $\beta$-a.e. $\omega$, which gives (\ref{eq:up gak low gam eq})
in the case $i=k$.

Now let $0\le i<k$ and suppose that (\ref{eq:up gak low gam eq})
has been proven for $i+1$. By Proposition \ref{prop:ineq for delta overline},
\[
\frac{\mathrm{H}_{i+1}-\mathrm{H}_{i}}{\tilde{\lambda}_{i+1}}\ge\overline{\gamma}_{i,k}(\omega)-\overline{\gamma}_{i+1,k}(\omega)\qquad\text{ for }\beta\text{-a.e. }\omega,
\]
by Proposition \ref{prop:transvers dim =00003D},
\[
\vartheta_{i}(\omega)\ge\frac{\mathrm{H}_{i+1}-\mathrm{H}_{i}}{\tilde{\lambda}_{i+1}}\qquad\text{ for }\beta\text{-a.e. }\omega,
\]
and by Proposition \ref{prop:ineq for delta underline},
\[
\underline{\gamma}_{i+1,k}(\omega)+\vartheta_{i}(\omega)\le\underline{\gamma}_{i,k}(\omega)\qquad\text{ for }\beta\text{-a.e. }\omega\:.
\]
Combining these facts with the induction hypothesis, we obtain that
for $\beta$-a.e. $\omega$,
\begin{eqnarray*}
\overline{\gamma}_{i,k}(\omega)-\sum_{j=i+1}^{k-1}\frac{\mathrm{H}_{j+1}-\mathrm{H}_{j}}{\tilde{\lambda}_{j+1}} & = & \overline{\gamma}_{i,k}(\omega)-\overline{\gamma}_{i+1,k}(\omega)\\
 & \le & \frac{\mathrm{H}_{i+1}-\mathrm{H}_{i}}{\tilde{\lambda}_{i+1}}\\
 & \le & \vartheta_{i}(\omega)\\
 & \le & \underline{\gamma}_{i,k}(\omega)-\underline{\gamma}_{i+1,k}(\omega)\\
 & = & \underline{\gamma}_{i,k}(\omega)-\sum_{j=i+1}^{k-1}\frac{\mathrm{H}_{j+1}-\mathrm{H}_{j}}{\tilde{\lambda}_{j+1}}\:.
\end{eqnarray*}
This proves (\ref{eq:up gak low gam eq}) also for $i$, which completes
the proof of the claim.
\end{proof}
We are finally ready to complete the proof of Theorem \ref{thm:LY formula and ED},
which follows easily from the last claim. Recall from Section \ref{subsec:boundary map}
that $\nu=\pi\beta$ is the Furstenberg measure corresponding to $\mu=\sum_{l\in\Lambda}p_{l}\delta_{A_{l}}$.
Also, recall from Section \ref{subsec:Dimension-formulas} that for
a proper linear subspace $W$ of $V$ the partition $\zeta_{W}$ of
$\mathrm{P}(V)\setminus\mathrm{P}(W)$ is define by,
\[
\zeta_{W}(\overline{x})=\{\overline{y}\in\mathrm{P}(V)\setminus\mathrm{P}(W)\::\:P_{W^{\perp}}\overline{y}=P_{W^{\perp}}\overline{x}\}\:.
\]

By remark \ref{rem:only last part needed} in Section \ref{subsec:Dimension-formulas},
in order to prove Theorem \ref{thm:LY formula and ED} we only need
to establish part (\ref{enu:LY for proj of slices}) of that theorem,
whose statement we now recall.
\begin{thm*}
For $\beta$-a.e. $\omega$, $\nu$-a.e. $\overline{x}$ and every
$0\le i<k\le s$, the measure $P_{(V_{\omega}^{k})^{\perp}}\nu_{\overline{x}}^{\zeta_{V_{\omega}^{i}}}$
is exact dimensional with,
\[
\dim P_{(V_{\omega}^{k})^{\perp}}\nu_{\overline{x}}^{\zeta_{V_{\omega}^{i}}}=\sum_{j=i}^{k-1}\frac{\mathrm{H}_{j+1}-\mathrm{H}_{j}}{\tilde{\lambda}_{j+1}}\:.
\]
\end{thm*}
\begin{proof}
Let $\mathbb{Z}_{\ge0}$ and $\mathbb{Z}_{<0}$ denote the sets of
nonnegative and negative integers respectively. Write $\Omega^{+}$
for the space of sequences $(\omega_{n})_{n\ge0}\in\Lambda^{\mathbb{Z}_{\ge0}}$,
and $\Omega^{-}$ for the space of sequences $(\omega_{n})_{n<0}\in\Lambda^{\mathbb{Z}_{<0}}$.
We equip each of these spaces with its Borel $\sigma$-algebra generated
by cylinder sets. Let $q^{+}:\Omega\rightarrow\Omega^{+}$ and $q^{-}:\Omega\rightarrow\Omega^{-}$
be the projections onto the nonnegative and negative coordinates respectively.
Note that since $\beta$ is a Bernoulli measure the maps $q^{+}$
and $q^{-}$ are independent as random elements on $(\Omega,\beta)$.
Write $\beta^{+}$ and $\beta^{-}$ for the Bernoulli measures corresponding
to $p$ on $\Omega^{+}$ and $\Omega^{-}$ respectively, that is $\beta^{+}=p^{\mathbb{Z}_{\ge0}}$
and $\beta^{-}=p^{\mathbb{Z}_{<0}}$.

Recall that $\pi$ only depends on the nonnegative coordinates. Thus
there exists a Borel map $\pi_{+}:\Omega^{+}\rightarrow\mathrm{P}(V)$
such that $\pi\omega=\pi_{+}q^{+}\omega$ for $\omega\in\Omega_{0}$.
Since $\nu=\pi\beta$ it follows that $\nu=\pi_{+}\beta^{+}$. Also,
recall that for each $0\le i\le s$ the map which takes $\omega\in\Omega_{0}$
to the $V_{\omega}^{i}$ depends only on the negative coordinates.
Thus we may write $V_{q^{-}\omega}^{i}$ in place of $V_{\omega}^{i}$
for $\omega\in\Omega_{0}$.

For a proper linear subspace $W$ of $V$ write $\xi_{W}$ for the
partition of $\Omega^{+}\setminus\pi_{+}^{-1}\mathrm{P}(W)$, such
that for $\omega$ in this set,
\[
\xi_{W}(\omega)=\{\eta\in\Omega^{+}\setminus\pi_{+}^{-1}\mathrm{P}(W)\::\:P_{W^{\perp}}\pi_{+}\eta=P_{W^{\perp}}\pi_{+}\omega\}\:.
\]
Note that,
\begin{equation}
\xi_{W}(\omega)=\pi_{+}^{-1}\zeta_{W}(\pi_{+}\omega)\:.\label{eq:pull back of zeta_W}
\end{equation}
Since $W\ne V$ we have,
\[
\beta^{+}(\pi_{+}^{-1}\mathrm{P}(W))=\nu(\mathrm{P}(W))=0,
\]
so the conditional measures $\{(\beta^{+})_{\omega}^{\xi_{W}}\}_{\omega\in\Omega^{+}}\subset\mathcal{M}(\Omega^{+})$
are $\beta^{+}$-a.e. defined.

From $\nu=\pi_{+}\beta^{+}$, (\ref{eq:pull back of zeta_W}) and
Lemma \ref{lem:push of slices},
\[
\pi_{+}(\beta^{+})_{\omega}^{\xi_{W}}=\nu_{\pi_{+}\omega}^{\zeta_{W}}\text{ for }\beta^{+}\text{-a.e. }\omega\:.
\]
Since $\beta=\beta^{-}\times\beta^{+}$, it is easy to verify (by
using \cite[Proposition 5.19]{EW} for instance) that for each $0\le i\le s$,
\begin{equation}
\beta_{\omega}^{\xi_{i}}=\delta_{q^{-}\omega}\times\biggl(\beta^{+}\biggr)_{q^{+}\omega}^{\xi_{V_{q^{-}\omega}^{i}}}\text{ for }\beta\text{-a.e. }\omega\:.\label{eq:expl form for beta^xi_i}
\end{equation}
From these facts together with $\pi=\pi_{+}q^{+}$, it follows that
for $0\le i\le s$ and $\beta$-a.e. $\omega$,
\begin{equation}
\pi\beta_{\omega}^{\xi_{i}}=\pi_{+}(\beta^{+})_{q^{+}\omega}^{\xi_{V_{q^{-}\omega}^{i}}}=(\nu)_{\pi_{+}q^{+}\omega}^{\zeta_{V_{q^{-}\omega}^{i}}}\:.\label{eq:push of cond}
\end{equation}

Now let $0\le i<k\le s$ be given and write,
\[
\alpha=\sum_{j=i}^{k-1}\frac{\mathrm{H}_{j+1}-\mathrm{H}_{j}}{\tilde{\lambda}_{j+1}}\:.
\]
From Claim \ref{cla:gam up =00003D gam low}, and the definitions
of $\overline{\gamma}_{i,k}$ and $\underline{\gamma}_{i,k}$, it
follows that for $\beta$-a.e. $\omega$,
\[
\underset{r\downarrow0}{\lim}\:\frac{\log\beta_{\omega}^{\xi_{i}}(\Gamma_{k}(\omega,r))}{\log r}=\alpha\:.
\]
By the definition of $\Gamma_{k}(\omega,r)$ this implies that for
$\beta$-a.e. $\omega$,
\[
\underset{r\downarrow0}{\lim}\:\frac{\log P_{(V_{q^{-}\omega}^{k})^{\perp}}\pi\beta_{\omega}^{\xi_{i}}(B(P_{(V_{q^{-}\omega}^{k})^{\perp}}\pi\omega,r))}{\log r}=\alpha\:.
\]
From this and (\ref{eq:push of cond}) it follows that for $\beta$-a.e.
$\omega$,
\[
\underset{r\downarrow0}{\lim}\:\frac{\log P_{(V_{q^{-}\omega}^{k})^{\perp}}(\nu)_{\pi_{+}q^{+}\omega}^{\zeta_{V_{q^{-}\omega}^{i}}}(B(P_{(V_{q^{-}\omega}^{k})^{\perp}}\pi_{+}q^{+}\omega,r))}{\log r}=\alpha\:.
\]
From this, since $\pi_{+}q^{+}\beta=\nu$ and since $q^{+}$ and $q^{-}$
are $\beta$-independent elements, it follows that for $\beta^{-}$-a.e.
$\omega$ and $\nu$-a.e. $\overline{x}$,
\[
\underset{r\downarrow0}{\lim}\:\frac{\log P_{(V_{\omega}^{k})^{\perp}}\nu_{\overline{x}}^{\zeta_{V_{\omega}^{i}}}(B(P_{(V_{\omega}^{k})^{\perp}}\overline{x},r))}{\log r}=\alpha\:.
\]
Note that for $\beta^{-}$-a.e. $\omega$, $\nu$-a.e. $\overline{x}$
and $\nu_{\overline{x}}^{\zeta_{V_{\omega}^{i}}}$-a.e. $\overline{y}$
the last equality holds with $\overline{y}$ in place of $\overline{x}$.
Since for $\beta^{-}$-a.e. $\omega$ and $\nu$-a.e. $\overline{x}$,
\[
\nu_{\overline{y}}^{\zeta_{V_{\omega}^{i}}}=\nu_{\overline{x}}^{\zeta_{V_{\omega}^{i}}}\qquad\text{ for }\nu_{\overline{x}}^{\zeta_{V_{\omega}^{i}}}\text{-a.e. }\overline{y},
\]
this completes the proof of the theorem.
\end{proof}
In the next lemma we show that the different definitions for $\mathrm{H}_{i}$,
given in Section \ref{subsec:Dimension-formulas} and Section \ref{sec:measurable partitions},
yield the same value.
\begin{lem}
\label{lem:H_i same val}Let $\mathcal{B}$ be the Borel $\sigma$-algebra
of $\mathrm{P}(V)$. Then for $0\le i\le s$ we have,
\[
\mathrm{H}_{\beta}(\mathcal{P}\mid\widehat{\xi_{i}})=\int\mathrm{H}_{\beta}(\mathcal{P}\mid\pi^{-1}P_{(V_{\omega}^{i})^{\perp}}^{-1}\mathcal{B})\:d\beta(\omega)\:.
\]
\end{lem}

\begin{proof}
We use here the notations introduced in the last proof. Let $\mathcal{P}^{+}$
be the partition of $\Omega^{+}$ according to the $0$-coordinate.
Given $\theta\in\mathcal{M}(\Omega^{+})$ it will be convenient to
write $\mathrm{H}(\mathcal{P}^{+};\theta)$ in place of $\mathrm{H}_{\theta}(\mathcal{P}^{+})$.
By the definitions of the conditional measures and entropy, and by
(\ref{eq:expl form for beta^xi_i}), we get
\begin{eqnarray*}
\mathrm{H}_{\beta}(\mathcal{P}\mid\widehat{\xi_{i}}) & = & \int\mathrm{H}_{\beta_{\omega}^{\xi_{i}}}(\mathcal{P})\:d\beta(\omega)\\
 & = & \int\mathrm{H}\left(\mathcal{P}^{+};\biggl(\beta^{+}\biggr)_{\omega_{1}}^{\xi_{V_{\omega_{2}}^{i}}}\right)\:d\beta^{+}(\omega_{1})\:d\beta^{-}(\omega_{2})\\
 & = & \int\mathrm{H}_{\beta^{+}}(\mathcal{P}^{+}\mid\widehat{\xi_{V_{\omega_{2}}^{i}}})\:d\beta^{-}(\omega_{2})\\
 & = & \int\mathrm{H}_{\beta}(\mathcal{P}\mid\pi^{-1}P_{(V_{\omega}^{i})^{\perp}}^{-1}\mathcal{B})\:d\beta(\omega),
\end{eqnarray*}
which completes the proof of the lemma.
\end{proof}
As a Corollary of Theorem \ref{thm:LY formula and ED} we can now
prove the following lemma, which was used in Section \ref{subsec:The-Lyapunov-dimension}
when the Lyapunov dimension was discussed. Recall the numbers $d_{0},...,d_{s}$
from Theorem \ref{thm:from Oseledets}.
\begin{lem}
\label{lem:ub on dif H}Let $0\le i<s$ be given, then
\[
0\le\mathrm{H}_{i}-\mathrm{H}_{i+1}\le-\tilde{\lambda}_{i+1}d_{i+1}\:.
\]
\end{lem}

\begin{proof}
For $\omega\in\Omega_{0}$ and $\overline{x}\in\mathrm{P}(V)\setminus\mathrm{P}(V_{\omega}^{i})$,
\[
\zeta_{V_{\omega}^{i}}(\overline{x})=\mathrm{P}(\overline{x}\oplus V_{\omega}^{i})\setminus\mathrm{P}(V_{\omega}^{i})\:.
\]
Hence,
\begin{equation}
P_{(V_{\omega}^{i+1})^{\perp}}(\zeta_{V_{\omega}^{i}}(\overline{x}))\subset\mathrm{P}(P_{(V_{\omega}^{i+1})^{\perp}}\overline{x}\oplus P_{(V_{\omega}^{i+1})^{\perp}}V_{\omega}^{i})\:.\label{eq:proj of slice contained}
\end{equation}
Note that,
\[
\dim P_{(V_{\omega}^{i+1})^{\perp}}V_{\omega}^{i}=\dim V_{\omega}^{i}-\dim V_{\omega}^{i+1}=d_{i+1},
\]
and so,
\[
\mathrm{P}(P_{(V_{\omega}^{i+1})^{\perp}}\overline{x}\oplus P_{(V_{\omega}^{i+1})^{\perp}}V_{\omega}^{i}),
\]
is a smooth $d_{i+1}$-dimensional manifold. This together with (\ref{eq:proj of slice contained})
gives,
\begin{equation}
\dim_{H}\left(P_{(V_{\omega}^{i+1})^{\perp}}(\zeta_{V_{\omega}^{i}}(\overline{x}))\right)\le d_{i+1},\label{eq:H-dim of proj of slice}
\end{equation}
where $\dim_{H}$ stands from Hausdorff dimension.

Additionally, for $\beta$-a.e. $\omega$ and $\nu$-a.e. $\overline{x}$,
\[
P_{(V_{\omega}^{i+1})^{\perp}}\nu_{\overline{x}}^{\zeta_{V_{\omega}^{i}}}\left(P_{(V_{\omega}^{i+1})^{\perp}}(\zeta_{V_{\omega}^{i}}(\overline{x}))\right)=1\:.
\]
Hence if $P_{(V_{\omega}^{i+1})^{\perp}}\nu_{\overline{x}}^{\zeta_{V_{\omega}^{i}}}$
is also exact dimensional, then from (\ref{eq:H-dim of proj of slice})
and (\ref{eq:def of h-dim}) we obtain that its dimension is at most
$d_{i+1}$. It now follows from part (\ref{enu:LY for proj of slices})
of Theorem \ref{thm:LY formula and ED} that for $\beta$-a.e. $\omega$
and $\nu$-a.e. $\overline{x}$,
\[
0\le\mathrm{H}_{i}-\mathrm{H}_{i+1}=-\tilde{\lambda}_{i+1}\dim P_{(V_{\omega}^{i+1})^{\perp}}\nu_{\overline{x}}^{\zeta_{V_{\omega}^{i}}}\le-\tilde{\lambda}_{i+1}d_{i+1},
\]
which completes the proof of the lemma.
\end{proof}

$\newline$$\newline$\textsc{Centre for Mathematical Sciences,\newline Wilberforce Road, Cambridge CB3 0WA, UK}$\newline$$\newline$\textit{E-mail: }
\texttt{ariel.rapaport@mail.huji.ac.il}
\end{document}